\documentclass[10pt,reno]{amsart}
\usepackage{amssymb}
\usepackage{amsmath}
\usepackage{enumitem}
\usepackage{mathrsfs}
\usepackage[english]{babel}
\usepackage[all]{xy}
\usepackage{hyperref}
\usepackage{marginnote}

\usepackage[new]{old-arrows}
\usepackage{stmaryrd}
\usepackage{mathtools}
\usepackage{xcolor}
\usepackage{tikz-cd}

\usepackage[margin=1.15in]{geometry}

\RequirePackage{mathrsfs}

\newtheorem{theorem}{Theorem}[section]
\newtheorem{lemma}[theorem]{Lemma}
\newtheorem{proposition}[theorem]{Proposition}
\newtheorem{corollary}[theorem]{Corollary}
\newtheorem{conjecture}[theorem]{Conjecture}

\theoremstyle{definition}
\newtheorem*{ack}{Acknowledgements}
\newtheorem*{con}{Conventions}
\newtheorem{remark}[theorem]{Remark}
\newtheorem{example}[theorem]{Example}
\newtheorem{definition}[theorem]{Definition}

\numberwithin{equation}{section} \numberwithin{figure}{section}

 \DeclareMathOperator{\NS}{NS}

\DeclareMathOperator{\Spec}{Spec}

\DeclareMathOperator{\an}{an}

\DeclareMathOperator{\Hom}{Hom} 
\DeclareMathOperator{\im}{Im}

\newcommand{\Qbar}{\overline{\QQ}}

\newcommand\ZZ{\mathbb{Z}}

\newcommand\QQ{\mathbb{Q}}

\newcommand\CC{\mathbb{C}}
\newcommand\GG{\mathbb{G}}

\newcommand\Gm{\GG_\mathrm{m}}
\newcommand\Gmu[1]{\GG_{\mathrm{m},#1}}
\newcommand\OO{\mathcal{O}}

\newcommand{\univcurvebar}[1]{\overline{\mathcal{U}_{#1}}}

\usepackage{color}

\definecolor{orange}{rgb}{1,0.5,0}

\title[Boundedness in families]{Boundedness in families   with applications to arithmetic hyperbolicity}

 \author{Raymond van Bommel}
 \address{Raymond van Bommel \\
Institut f\"{u}r Mathematik\\
Johannes Gutenberg-Universit\"{a}t Mainz\\
Staudingerweg 9, 55128 Mainz\\
Germany.}
\email{bommel@uni-mainz.de}

\author{Ariyan Javanpeykar}
\address{Ariyan Javanpeykar \\
Institut f\"{u}r Mathematik\\
Johannes Gutenberg-Universit\"{a}t Mainz\\
Staudingerweg 9, 55128 Mainz\\
Germany.}
\email{peykar@uni-mainz.de}

\author{Ljudmila Kamenova}
\address{Ljudmila Kamenova \\ 
Department of Mathematics \\
Stony Brook University \\ 
Stony Brook, NY 11794-3651 \\ 
USA.}
\email{kamenova@math.stonybrook.edu}



\begin{document}

\begin{abstract} Motivated by conjectures of Demailly, Green--Griffiths,   Lang, and Vojta, we show that several  notions related to hyperbolicity behave similarly in families. We apply our results to show the persistence of arithmetic hyperbolicity along field extensions for projective normal surfaces with nonzero irregularity. These results rely on the mild boundedness of semi-abelian varieties. We also introduce and study the notion of pseudo-algebraic hyperbolicity which extends Demailly's notion of algebraic hyperbolicity for projective schemes.
\end{abstract}

\maketitle
 
\tableofcontents

\thispagestyle{empty}

\section{Introduction}  
The aim of this paper is to provide   evidence for the following conjecture due in part to Demailly, Green--Griffiths,   Lang and Vojta. The following conjecture is a consequence of conjectures appearing in  \cite[\S0.3]{Abr},  \cite[Conj.~XV.4.3]{CornellSilverman}, \cite{Demailly}, \cite{GrGr}, \cite[\S6]{Javan2},  \cite{Lang2}, and \cite[Conj.~4.3]{Vojta3}. 
\begin{conjecture}[Demailly, Green--Griffiths, Lang, Vojta]\label{conj} Let $X$ be a projective variety over a field $F$ of characteristic zero. Then the following statements are equivalent.
\begin{enumerate}[label=(\roman*)]
\item The projective variety $X$ is algebraically hyperbolic over $F$.
\item The projective variety $X$ is bounded over $F$.
\item For all $n\geq 1$ and $m\geq 1$,  the projective variety $X$ is  $(n,m)$-bounded over $F$.
\item Every integral closed subvariety of $X$ is of general type.
\item The projective variety $X$ is groupless over $F$.
\end{enumerate}
\end{conjecture}

Let us explain some of the terminology appearing in Conjecture \ref{conj}.  First, throughout this paper,  we will let $k$ be an algebraically closed field of characteristic zero.  Now,  a projective variety $X$ over  $k$   is said to be  \emph{algebraically hyperbolic over $k$} if there is an ample line bundle $\mathcal{L}$ on $X$ and a real number $\alpha_{X,\mathcal{L}}$ such that, for every smooth projective irreducible curve $C$ over $k$ of genus $g$ and every morphism $f:C\to X$, the inequality \[ \deg_C f^\ast \mathcal{L} \leq \alpha_{X,\mathcal{L}} \cdot g\] holds.

For $n\geq 1$ and $m\geq 0$ integers,  we   follow the terminology introduced in \cite[\S4]{JKa} and refer to a projective variety as $(n,m)$-bounded over the algebraically closed field $k$  if, for  every projective normal integral scheme $Y$ over $k$ of dimension at most $n$, all pairwise distinct points $y_1,\ldots, y_m \in Y(k)$, and all $x_1,\ldots,x_m\in X(k)$, the scheme \[
\underline{\mathrm{Hom}}_k([Y,y_1,\ldots,y_m], [X,x_1,\ldots,x_m])
\] parametrizing morphisms $f:Y\to X$ with $f(y_1) = x_1, \ldots, f(y_m) = x_m$
 is of finite type over $k$; we refer the reader to \cite[\S 3]{Debarrebook1} for basic definitions and properties of Hom-schemes. We say that $X$ is $n$-bounded over $k$ if it is $(n,0)$-bounded over $k$ and we say that $X$ is bounded over $k$ if it is $n$-bounded over $k$ for every integer $n$.  We refer the reader to \cite{JKa} for a discussion of the relations between algebraic hyperbolicity, boundedness, and $(n,m)$-boundedness.

 More generally,  if $X$ is a projective variety over a (not necessarily algebraically closed) field $F$ of characteristic zero, we say that $X$ is algebraically hyperbolic over $F$ if $X_{\overline{F}}$ is algebraically hyperbolic over $\overline{F}$, where $F\to \overline{F}$ is some algebraic closure. We define the notions of $n$-boundedness, boundedness and $(n,m)$-boundedness over $F$ in a similar manner.  

We   follow standard terminology and say that an integral proper scheme $X$ over    $F$ is of general type if it has a desingularisation $X'\to X$ such that $\omega_{X'/k}$ is a big line bundle. Also, we will say that a proper scheme $X$ over $F$  is of general type if, for every irreducible component $Y$ of $X$, the  reduced closed subscheme $Y_\mathrm{red}$ is of general type. Finally, a proper variety $X$ over $F$ is groupless if, for every abelian variety $A$ over $\overline{F}$, every morphism $A\to X_{\overline{F}}$ is constant; see  \cite{JAut, JKa, JVez, JXie} for basic properties of groupless varieties.

   Our starting point in this paper is the fact that the notion of being of general type is an open condition in families of projective varieties. This statement can be deduced from results of Siu, Kawamata, and Nakayama (see  \cite{NakayamaBook}).

\begin{theorem}[Nakayama]\label{thm:nakayama}
Let $X\to S$ be a proper morphism of schemes. Then the set of $s$ in $S$ such that $X_s$ is of general type is an open subscheme of $S$.
\end{theorem}

Lang notes that ``the extent to which hyperbolicity is   open for the Zariski topology in families (of projective varieties)'' is unclear \cite[p.~176]{Lang2}. Our aim in  this paper is to investigate  how every notion of hyperbolicity appearing in Conjecture \ref{conj:new} behaves in families and to show that all these notions are ``Zariski-countable open''.  It seems worth stressing that  it is not known whether \emph{any} notion of hyperbolicity appearing in Conjecture \ref{conj:new} is a Zariski open condition in families.  

\subsection{Stable under generisation}
  Nakayama's theorem implies that  the locus of $s$ in $S$ such that every subvariety of $X_s$ is of general type   is stable under generisation.  Our first result confirms that every notion appearing in Conjecture \ref{conj} is in fact stable under generisation.

\begin{theorem}[Generisation]\label{thm:generisation}
Let $S$ be an integral noetherian scheme and let $X \to S$ be a projective   morphism. Let $s \in S$ be a closed point with   residue field $\kappa$ of characteristic 0. Let $X_{\overline{{K(S)}}}$ be the geometric generic fibre of $X \to S$.
\begin{enumerate}[label=(\roman*)]
\item If every integral closed subvariety of $X_{s}$ is of general type, then every integral closed subvariety of $X_{ \overline{{K(S)}}}$ is of general type.
\item If $X_s$ is groupless, then $X_{ \overline{{K(S)}}}$ is groupless.
\item If $X_s$ is an algebraically hyperbolic projective variety, then $X_{\overline{ {K(S)}}}$ is algebraically hyperbolic.
\item If $X_s$ is a bounded   projective variety over $\kappa$, then $X_{\overline{ {K(S)}}}$ is bounded over $ \overline{{K(S)}}$.
\item Let $n \geq 1$ and $m \geq 0$ be integers.  If $X_s$ is an $(n,m)$-bounded projective variety over $\kappa$, then $X_{ \overline{{K(S)}}}$ is $(n,m)$-bounded over $\overline{ {K(S)}}$.
\end{enumerate}
\end{theorem}

Let us also mention the complex-analytic analogue of Theorem \ref{thm:generisation}. Namely,  let  $\mathfrak{X}\to S$ be a surjective  holomorphic map of complex analytic spaces  with compact fibres.  If there is a point $s$ in $S$ such that the fibre $\mathfrak{X}_s$ is Kobayashi hyperbolic, then there is an analytic open neighbourhood $U\subset S$ of $s$ such that, for every $u$ in $U$, the fibre $\mathfrak{X}_u$ is Kobayashi hyperbolic; see \cite[Theorem~3.11.1]{KobayashiBook}.

The first statement on varieties of general type in Theorem \ref{thm:generisation} follows   from  Nakayama's result stated above (see Section \ref{section:nakayama}).  The  second statement on grouplessness is proven using non-archimedean methods in   \cite{JVez}. As we will explain below, the third statement on algebraic hyperbolicity follows from  a mild generalisation  of a theorem of Demailly (see Theorem \ref{thm:demailly} below). The last two statements are proven   in Section \ref{section:zco}. In fact, we deduce these two statements from the fact that the locus of $s$ in $S$   such that $X_s$ is bounded (respectively $(n,m)$-bounded) is a Zariski-countable open in the sense defined below.

\subsection{Countable-openness of the hyperbolic locus}
Given a projective morphism  $X\to S$ with $S$ a complex algebraic variety, 
 Demailly showed that the locus of $s$ in $S(\CC)$ such that $X_s$ is algebraically hyperbolic is   an open subset of $S(\CC)$ in the countable-Zariski topology. Recall that, 
if  $(X, \mathcal{T})$ is a noetherian topological space, then there exists another topology $\mathcal{T}^{\mathrm{cnt}}$, or $\mathcal{T}$-countable, on $X$ whose closed sets are the countable unions of $\mathcal{T}$-closed sets (see Lemma \ref{lem:countabletopology}). If $S$ is a noetherian scheme, a subset $Z\subset S$ is a Zariski-countable closed if it is a countable union of closed subschemes $Z_1, Z_2, \ldots \subset S$. Note that a Zariski-countable closed subset is closed under specialization. In case $S$ is countable, then this is an equivalence.

 With the aforementioned terminology at hand, Demailly essentially proved the following result.

 \begin{theorem}[Demailly] \label{thm:demailly}
  Let $S$ be a noetherian scheme over $\QQ$ and let $X \to S$ be a projective   morphism.  Then, the set of $s$ in $S$ such that $X_s$ is algebraically hyperbolic  is  Zariski-countable open in $S$.
\end{theorem}
 This is not the exact result proven by Demailly. Indeed, Demailly proved that, if $k=\CC$ and $S^{\textrm{not-ah}}$ is the set of $s$ in $S$ such that $X_s$ is not algebraically hyperbolic, then $S^{\textrm{not-ah}}\cap S(\CC)$ is closed in the countable topology on $S(\CC)$. This, strictly speaking, does not imply that  $S^{\textrm{not-ah}}$ is closed in the countable topology on   $S$. For example, if $S$ is an integral curve over $\CC$ and $\eta$ is the generic point of $S$, then $\{\eta\}$ is not a Zariski-countable open of $S$, whereas $\{\eta\}\cap S(\CC)= \emptyset$ is a Zariski-countable open of $S(\CC)$. 

We give a proof of Theorem \ref{thm:demailly} which is similar to Demailly's proof, but more adapted to the scheme-theoretic setting. Moreover, 
note that Demailly's theorem as stated above actually implies that the notion of being algebraically hyperbolic is stable under generisation. 
Furthermore, to prove Theorem \ref{thm:demailly} we  replace part of Demailly's       line of reasoning   by  stack-theoretic arguments. Finally, using similar (but slightly more involved) arguments, we obtain analogous  results on boundedness and $(n,m)$-boundedness.

 \begin{theorem}[Countable-openness of boundedness] \label{thm:demailly-b}
  Let $S$ be a noetherian scheme over $\QQ$  and let $X \to S$ be a projective   morphism.  Then, the set of $s$ in $S$ such that $X_s$ is  bounded over $k(s)$ is  Zariski-countable open in $S$.  
\end{theorem}

 \begin{theorem}[Countable-openness of $(n,m)$-boundedness] \label{thm:demailly-mn}  
  Let $S$ be a noetherian scheme over $\QQ$, let $n\geq 1$ and $m\geq 0$ be integers. If  $X \to S$ is projective, then the set of $s$ in $S$ such that $X_s$ is  $(n,m)$-bounded over $k(s)$ is  Zariski-countable open in $S$.  
\end{theorem}

As   algebraic hyperbolicity is conjecturally equivalent to   every subvariety being of general type (Conjecture \ref{conj}), one expects a similar  Zariski-countable openness property to hold for the latter notion. We use Nakayama's theorem and  the fact that the stack of proper schemes of general type is a countable union of finitely presented algebraic stacks to prove the  following  result.

\begin{theorem}[Countable-openness of every subvariety being of general type]\label{thm:specialisation}   Let $S$ be a noetherian scheme over $\QQ$ and let $X \to S$ be a projective morphism. Then, the set of $s$ in $S$ such that  every integral closed subvariety of $X_s$  is of general type is Zariski-countable open in $S$.
\end{theorem}

Finally, in \cite{JVez} it is shown that the set of $s$ in $S$ such that $X_s$ is groupless is open in the Zariski-countable topology on $S$. Thus, for every property appearing in Conjecture \ref{conj}, the locus of $s$ in $S$ such that $X_s$ has this property  is Zariski-countable open in $S$.

\subsection{Mildly bounded varieties}\label{section:mb}
The notions of algebraic hyperbolicity, boundedness, and grouplessness discussed above are expected to  coincide.  For example, a non-zero abelian variety is not hyperbolic, neither algebraically hyperbolic, nor bounded. 

In \cite[Definition~4.1]{JAut},   a  ``weak''  notion of  boundedness that suffices for certain arithmetic applications (see Theorem  \ref{thm:gen_crit} below) is introduced. The precise definition of this notion reads as follows. (Recall that $k$ denotes an algebraically closed field of characteristic zero.)

\begin{definition}[Mildly bounded varieties]\label{defn:mild_bounded}
Let $F$ be a field with algebraic closure $F\to {k}$.
A finite type separated scheme $X$ over  $F$ is \emph{mildly bounded over $F$} if, for every smooth quasi-projective irreducible curve $C$ over ${k}$,  there is an integer $m$ and points $c_1,\ldots, c_m$ in $C({k})$ such that, for every $x_1,\ldots, x_m$ in $X({k})$, the set of morphisms $f:C\to X_{{k}}$ with $f(c_1) = x_1, \ldots, f(c_m) = x_m$ is finite.
\end{definition}

Note that the notion of being mildly bounded over $F$ depends only on the isomorphism class of the variety over the algebraic closure $k$ of $F$. 

It is not hard to show that mildly bounded proper varieties  have no rational curves. More generally, if $X$ is a mildly bounded variety over $k$, then every morphism $\mathbb{A}^1_k \to X$ is constant (see Proposition \ref{prop:a1}).

Quite surprisingly, we are able to  prove that every semi-abelian variety over a field of characteristic zero is mildly bounded, so that the notion of mild boundedness is \emph{strictly weaker} than any notion of hyperbolicity or boundedness discussed above (including grouplessness).  That is, although abelian varieties are \textbf{not} bounded, they are mildly bounded.

\begin{proposition}\label{prop:semiabvar_is_mildbounded} 
If  $X$ is a semi-abelian variety over $k$,  then $X$ is mildly bounded over $k$. 
\end{proposition}

 Proposition \ref{prop:semiabvar_is_mildbounded} shows that mildly bounded varieties are not necessarily hyperbolic, nor even of general type. 
Related to this proposition we show the following global boundedness result for families of abelian varieties. Its proof relies on  Silverman's specialisation theorem \cite{SilvermanSpec}.

To state it, recall that an integral curve $S$ over $k$ is \emph{hyperbolic} if its normalization $S'$ has negative Euler characteristic, i.e., if $\overline{S'}$ is the smooth projective model of $S'$, then $S$ is hyperbolic if and only if $2\cdot\mathrm{genus}(\overline{S'}) -2 + \#(\overline{S'}\setminus S') >0$.

\begin{theorem}\label{thm:families_of_ab_var_are_mb}  Let  $S$ be a hyperbolic integral curve over $k$, and let  $\mathcal{X}\to S$ be a semi-abelian scheme over $S$. Then $\mathcal{X}$ is mildly bounded over $k$.\end{theorem}

We conjecture that, for projective varieties, the only obstruction to being mildly bounded is the presence of a rational curve.  

\begin{conjecture}\label{conj:new}
If  $X$ is a projective variety over $k$  with no rational curves,  then $X$ is mildly bounded over $k$.
\end{conjecture}
The following result says that our 
 conjecture holds for    surfaces, under a suitable assumption on the Albanese variety (commonly referred to as being of maximal Albanese dimension). Its proof crucially uses the mild boundedness of   abelian varieties (Corollary \ref{cor:mb_of_total_space}).

\begin{theorem}\label{thm:surfaces_are_mb}  
Let $X$ be a projective integral surface over  $k$  with no rational curves.  If there is an abelian variety $A$ and  a morphism $X\to A$ which is generically finite onto its image, then $X$ is mildly bounded over $k$.
\end{theorem}

We can also prove the conjecture for groupless projective surfaces which admit a non-constant map to  some abelian variety. In particular, the conjecture holds if $X$ is a groupless projective normal surface with non-zero irregularity $q(X) := \mathrm{h}^1(X,\mathcal{O}_X)$. 
\begin{theorem}\label{thm:surfaces_q}
Let $X$ be a projective groupless surface over $k$. If $X$ admits a non-constant map to some abelian variety over $k$, then $X$ is mildly bounded over $k$.
\end{theorem}

 Recall that, if $X\to S$ is a projective morphism of noetherian schemes over $\QQ$,  the set of $s$ in $S$ such that $X_{\overline{k(s)}}$ has no rational curve is  Zariski-countable open in $S$ (see for instance \cite{Debarrebook1}). Our next result verifies that the locus of $s$ in $S$ with $X_s$ mildly bounded is  Zariski-countable open. In light of the preceding statement about rational curves, this result is in accordance with Conjecture \ref{conj:new}.

\begin{theorem}\label{thm:mildly-bounded-is-zco}
Let $S$ be a noetherian scheme over $\QQ$ and let $X\to S$ be a projective morphism. Then the set of $s$ in $S$ such that $X_s $ is mildly bounded  over $k(s)$   is Zariski-countable open in $S$.
\end{theorem}

It seems worthwhile to stress that the proof of Theorem \ref{thm:mildly-bounded-is-zco} follows a similar line of reasoning as the proofs of Theorems \ref{thm:demailly}, \ref{thm:demailly-b}, and \ref{thm:demailly-mn}. However, the proof of Theorem \ref{thm:mildly-bounded-is-zco} is arguably the most involved, due to the fact that the condition of mild boundedness is \emph{much weaker} than the conditions of being \emph{algebraically hyperbolic, bounded, or $(n,m)$-bounded}, respectively.

As before, the Zariski-countable openness of the locus of $s$ in $S$ such that $X_s$ is mildly bounded implies that this locus is stable under generisation.
\begin{corollary}\label{cor:mildly-bounded-gens}
Let $S$ be an integral noetherian scheme over $\QQ$ and let $X\to S$ be a projective morphism. If there is an $s$ in $S$ such that $X_s$ is mildly bounded over $k(s)$, then the generic fibre $X_{K(S)}$ is mildly bounded.
\end{corollary}

It is not clear that, given a mildly bounded variety $X$ over $k$ and a field extension $k\subset L$, the variety $X_L$ is mildly bounded over $L$.  Using the fact that being mildly bounded is stable under generisation in projective families, we are able to deduce the  persistence of mild boundedness over field extensions of \emph{finite} transcendence degree.

\begin{corollary}\label{cor:mb_persists}
Let $k\subset L$ be an extension of  algebraically closed fields of characteristic zero, and let $X$ be a  projective mildly bounded variety over $k$. If $k\subset L$ has finite transcendence degree, then $X_L$ is mildly bounded over $L$.
\end{corollary}

\subsection{Persistence of arithmetic hyperbolicity} 
The notion of mildly bounded varieties was introduced in \cite{JAut} with the aim of giving arithmetic applications; see for instance \cite[Theorem~1.6]{JAut}.   In Section \ref{section:faltingsvojta}  we give  such applications based on the results in Section \ref{section:mb} and \cite{JAut}.  For example, we prove the following new result on rational points on surfaces.  

\begin{theorem}\label{thm:surfaces_intro} Let $X$ be a projective integral surface over a number field $K$  such  that   there is a non-constant morphism $X_{\overline{K}} \to A$ to an abelian variety $A$ over $\overline{K}$.   Assume that for every \textbf{number field} $L$ over $K$, the set $ {X}(L)$ is finite. Then, for every  \textbf{finitely generated field} $M$ of characteristic zero, the set   $X(M)$  is finite.
\end{theorem}

Note that a special case of Theorem \ref{thm:surfaces_intro} was already proven   in  \cite[Theorem~4.9]{JAut}, under the additional assumption that the Albanese map of $X$ is   finite. The proof in \emph{loc.\ cit.}\ relies on Yamanoi's extension of Bloch-Ochiai-Kawamata's theorem to finite covers of abelian varieties \cite{Yamanoi}. However, as Yamanoi's theorem does not apply to surfaces with non-maximal Albanese dimension, we can not use his results to prove Theorem \ref{thm:surfaces_intro}. Instead, to prove   Theorem \ref{thm:surfaces_intro}, we will rely   on the fact that  abelian varieties are   mildly bounded (Proposition \ref{prop:semiabvar_is_mildbounded}).

Theorem \ref{thm:surfaces_intro} verifies a prediction implied by   the Lang--Vojta conjecture. To explain this, we recall that  Lang introduced the notion of arithmetic hyperbolicity (sometimes also referred to as \emph{Mordellicity}) over $\Qbar$   to capture the property of having only finitely many rational points  over number fields.  This notion is studied for instance in   \cite{Autissier1, Autissier2, JBook, JLitt, JAut, JLalg},    \cite[\S 2]{UllmoShimura}, and \cite{VojtaLangExc}.   Let us start with extending Lang's notion to varieties over arbitrary algebraically closed fields $k$ of characteristic zero. 

\begin{definition}[Arithmetic hyperbolicity]\label{defn:arhyp}   A finite type separated scheme
  $X$ over $k$ is \emph{arithmetically hyperbolic over $k$} 
if  there is a  $\ZZ$-finitely generated subring $A\subset k$ and a finite type separated $A$-scheme $\mathcal{X}$ with $\mathcal{X}_k \cong X$ over $k$ such that, for all $\ZZ$-finitely generated subrings $ A'\subset k$ containing $A$,  the set of $A'$-points $\mathcal{X}(A'):=\Hom_A(\Spec A', \mathcal{X})$ on $\mathcal{X}$ is finite.
\end{definition} 

For example, the variety $\mathbb{A}^1_k\setminus\{0,1\}$ is arithmetically hyperbolic over $k$. Indeed, the $R$-integral points on $\mathbb{A}^1_R\setminus \{0,1\}$ correspond to solutions of the $R$-unit equation $u+v=1$ with $u,v \in R^\ast$, and these are finite by the theorem of Siegel-Mahler-Lang (see below for a slightly more detailed discussion).

 Lang--Vojta's arithmetic conjecture (for projective varieties) says that grouplessness (and thus also algebraic hyperbolicity) is equivalent   to being arithmetically hyperbolic.

\begin{conjecture}[Arithmetic Lang--Vojta]
A projective variety $X$ over $k$ is groupless over $k$ if and only if $X$ is arithmetically hyperbolic over $k$.
\end{conjecture}

Faltings's theorems show that the Arithmetic Lang--Vojta conjecture holds if $X$ is one-dimensional, or when $X$ is a closed subvariety of an abelian variety; see \cite{Faltings2, FaltingsComplements, Faltings3, FaltingsLang1}. Also, this conjecture is   known for projective varieties which admit a finite morphism to the moduli stack of polarized abelian varieties by Faltings's proof of the Shafarevich conjecture. Although it is still open for ramified covers of abelian varieties, some progress was made in \cite{CDJLZ}.

The Arithmetic Lang--Vojta conjecture has many interesting consequences. For example, in light of the aforementioned properties of grouplessness, it predicts that ``being arithmetically hyperbolic'' is stable under generisation and   even a Zariski-countable open condition in projective families of varieties. Verifying these predictions of Lang--Vojta's conjecture seems currently out of reach. Slightly more reasonable  seems to be the Arithmetic Persistence Conjecture stated below. Under the additional assumption  that $X$ is projective, the Arithmetic Persistence Conjecture  is indeed a consequence of the Arithmetic Lang--Vojta conjecture. However, as the conjecture also seems to be reasonable   in the quasi-projective setting, we state it in this    generality.

\begin{conjecture}[Arithmetic Persistence Conjecture]\label{conj:pers}
Let $k\subset L$ be an extension of algebraically closed fields of characteristic zero. If $X$ is an arithmetically hyperbolic finite type separated scheme over $k$, then $X_L$ is arithmetically hyperbolic over $L$.
\end{conjecture} 

To give a better idea of what the Arithmetic Persistence Conjecture entails, let us consider an  affine finite type scheme $\mathcal{X}$ over $\ZZ$ with only finitely many integral points, i.e., for every number field $K$ and every finite set of finite places $S$ of $K$, the set of $\OO_{K,S}$-points of $\mathcal{X}$ is finite.  Then, the Arithmetic Persistence Conjecture says that  for every \emph{finitely generated integral domain  $A$ of characteristic zero}, the set $\mathcal{X}(A)$ is finite. This  is not an unreasonable expectation, as it can be verified in many cases   (see for instance \cite{EvertseGyory} or \cite{JLitt}). In fact, 
a point of view emphasized by Lang in his seminal paper \cite[p.~202]{Lang2} is that finiteness statements involving  integral points over rings of integers should continue to hold (or "persist") over  $\ZZ$-finitely generated integral domains of characteristic zero.  In this direction, after  Siegel, Mahler, and Parry's work \cite{Siegel1, Mahler1, Parry} in the classical setting of rings of ($S$-)integers, Lang proved that the unit equation
\begin{align*}
u+v=1, \quad u,v\in A^*,
\end{align*}
has only finitely many solutions when $A$ is any $\ZZ$-finitely generated integral domain of characteristic zero.  Lang first proved these results for rings of $S$-integers in number fields, and then used a specialization argument to deduce the general case. In other words, he proved the arithmetic hyperbolicity of $\mathbb{P}^1\setminus \{0,1,\infty\}$ over $\Qbar$ and also the Arithmetic Persistence Conjecture for $\mathbb{P}^1\setminus \{0,1,\infty\}$.

Theorem \ref{thm:surfaces_intro} can be reformulated as saying that the Arithmetic Persistence Conjecture holds for projective surfaces $X$ over $\Qbar$  which admit a non-constant map to some abelian variety over $\Qbar$.    Indeed,   if $X$ is a projective arithmetically hyperbolic surface over $\Qbar$ which admits a non-constant morphism to an abelian variety and $\Qbar\subset k$ is an extension of algebraically closed fields, then Theorem \ref{thm:surfaces_intro} says that  $X_k$ is arithmetically hyperbolic over $k$, and this can be shown to imply the finiteness of rational points $X(M)$ with $M$ as in Theorem \ref{thm:surfaces_intro} (see Section \ref{section:weird}).  The restriction to number fields  in Theorem \ref{thm:surfaces_intro} is unnecessary, and was made above only for the sake of simplifying the statement; see Theorem \ref{thm:surfaces_text2} for a more general statement.

Our next result solves the Arithmetic Persistence Conjecture for varieties which admit a quasi-finite morphism to some semi-abelian variety. 

\begin{theorem}\label{thm:semi_abvars_intro} Let $A$ be a $\ZZ$-finitely generated integral domain of characteristic zero with fraction field $K$ and let $\mathcal{X}$ be a finite type separated scheme over $A$ such that $\mathcal{X}_{\overline{K}}$ admits a quasi-finite morphism to a semi-abelian variety over $\overline{K}$. Assume that, for every \textbf{finite} extension $L/K$ and every $\ZZ$-finitely generated subring $A'\subset L$ containing $A$, the set $\mathcal{X}(A')$ is finite. Then,  for every \textbf{finitely generated} field extension $M/K$ and every $\ZZ$-finitely generated subring $B\subset M$ containing $A$, the set $\mathcal{X}(B)$ is finite.\end{theorem}

Theorem \ref{thm:semi_abvars_intro} is concerned with the finiteness of  integral points on (not necessarily proper) varieties which map quasi-finitely (possibly surjectively) to some semi-abelian variety. As before, we  note that   Theorem \ref{thm:semi_abvars_intro} was proved using Yamanoi's  results on abelian varieties in the case that $\mathcal{X}$ is   proper and smooth in  \cite{JAut}. Removing the properness condition in Yamanoi's work would require a substantial amount of new ideas. In this paper, we prove Theorem \ref{thm:semi_abvars_intro} by appealing to the simple fact that semi-abelian varieties are mildly bounded (Proposition \ref{prop:semiabvar_is_mildbounded}).

\subsection{Diophantine applications}
Let us briefly explain in what Diophantine context one can apply our work on the Arithmetic Persistence Conjecture. 

For example, using our results we  can prove Faltings's finiteness theorem for integral points on complements of ample effective divisors in abelian varieties over arbitrary finitely generated fields of characteristic zero. Such a result is indicated (without proof) in Faltings's paper \cite{Faltings3}. Using the mild boundedness of abelian varieties (in a uniform sense), we give a short proof of this extension of Faltings's theorem in Theorem \ref{thm:fal}. It seems reasonable to suspect that one can also prove similar extensions to semi-abelian varieties over finitely generated fields by appealing to Vojta's more general finiteness results for semi-abelian varieties over number fields \cite{Vojta1a, Vojta2b}.

Note that we also make an observation on   Hassett-Tschinkel's puncturing problem for simple abelian varieties; see Remark \ref{remark:ht} for a discussion.

\subsection{Predictions made by Vojta's geometric conjecture}
Part of the conjecture of Demailly, Green-Griffiths, and Lang as stated above (Conjecture \ref{conj}) is implied by Vojta's  conjecture that a projective variety is of general type if and only if it is  pseudo-algebraically hyperbolic; see \cite{Vojta11} or \cite{VojtaABC}. Here,  for $X$ a projective variety  over an algebraically closed field $k$ of characteristic zero and $\Delta\subset X$ a closed subset, we say that $X$  is \emph{ algebraically hyperbolic modulo $\Delta$ over $k$} if there is  an ample line bundle $\mathcal L$ on $X$ and  a real number $\alpha_{X, \Delta, \mathcal L}$ depending only on $X$, $\Delta$, and $\mathcal L$  such that, for every smooth projective curve $C$ over $ {k}$ and every  morphism $f:C\to X$ with $f(C)\not\subset \Delta$, the inequality 
\[
\deg_C f^\ast \mathcal L \leq \alpha_{X, \Delta, \mathcal L} \cdot \mathrm{genus}(C)
\] holds.   Also, a projective variety $X$ over an algebraically closed field $k$ of characteristic zero is \emph{pseudo-algebraically hyperbolic over $k$} if there is a proper closed subset $\Delta \subsetneq X$ such that $X$ is  algebraically hyperbolic modulo $\Delta$ over ${k}$.  More generally, a projective variety over an arbitrary field $F$ of characteristic zero is \emph{pseudo-algebraically hyperbolic over $F$} if $X_{\overline{F}}$ is pseudo-algebraically hyperbolic over the algebraic closure $\overline{F}$ of $F$.

This use of the    word ``pseudo'' was introduced by Kiernan-Kobayashi \cite{KiernanKobayashi} and also employed by Lang \cite{Lang2} to indicate the presence of an "exceptional locus", and is used throughout the literature (see also, for example, \cite{RousseauTurchetWang}).

\begin{conjecture}[Vojta's conjecture]\label{conj:vojta}
A projective variety $X$ over $k$ is of general type over $k$ if and only if it is pseudo-algebraically hyperbolic over $k$.
\end{conjecture}

To motivate our next result, let $S$ be an integral variety over  $k$ and let $X \to S$ be a projective   morphism whose  generic fibre $X_{{K(S)}}$ is of general type. Then,  for a    general $s$ in $S(k)$, the projective scheme $X_s$  is of general type by Nakayama's theorem (Theorem \ref{thm:nakayama}).
 Vojta's conjecture (Conjecture \ref{conj:vojta}) predicts   that   pseudo-algebraic hyperbolicity specialises in a similar way that varieties of general type do. Our next result verifies part of this prediction. This result is  a consequence of Proposition \ref{prop:specialising_hyperbolicity}.

\begin{theorem}[Specialising pseudo-algebraic hyperbolicity]\label{thm:specialisation2} Let $S$ be an integral variety over $k$ and let $X \to S$ be a projective   morphism whose generic fibre $X_{ {K(S)}}$ is pseudo-algebraically hyperbolic. Then,    for a very general  $s$  in $S(k)$, the projective scheme  $X_s$ is pseudo-algebraically hyperbolic over $k(s)$.  (That is, there are countably many proper closed subsets  $\Delta_i\subset S(k)$ such that, for every $s$ in $S(k)\setminus \cup_{i=1}^\infty \Delta_i$, the projective scheme $X_s$ is pseudo-algebraically hyperbolic over $k(s)$.)\end{theorem}

Moreover, 
  since being of general type persists over field extensions, our next result is also in accordance with Vojta's conjecture (Conjecture \ref{conj:vojta}). 

\begin{theorem}[Persistence of pseudo-algebraic hyperbolicity]\label{thm:ps_alg}
Let $k\subset L$ be an extension of   algebraically closed fields of characteristic zero. 
If $X$ is a pseudo-algebraically hyperbolic projective variety over $k$, then $X_L$ is pseudo-algebraically hyperbolic over $L$.
\end{theorem}

Similar  results are proven in Section \ref{section:pseudo} for the notion of boundedness, and these results are applied, for example, in \cite{JMRL}. We   ``pseudofy''  this notion of boundedness as follows:    a projective scheme $X$ over an algebraically closed field $k$ of characteristic zero is said to be \emph{pseudo-bounded over $k$} if there is a proper closed subscheme $\Delta\subsetneq X$ such that, for every smooth projective connected variety $Y$ over $k$, the scheme \[\underline{\Hom}_k(Y,X)\setminus \underline{\Hom}_k(Y,\Delta)\] parametrizing morphisms $f:Y\to X$ with $f(Y)\not\subset \Delta$ is of finite type over $k$.
In this case, we also say that $X$ is \emph{bounded modulo $\Delta$}. 
 The relation to pseudo-algebraic hyperbolicity is as follows.

\begin{theorem}[From curves to varieties]\label{thm:alghyp_to_bounded}  
If $X$ is algebraically hyperbolic modulo $\Delta$ over $k$, then $X$ is bounded modulo $\Delta$ over $k$.
\end{theorem}

The following result shows that mere boundedness implies the existence of a uniform bound for maps from a fixed curve in the genus of that curve.   Therefore, the  \emph{a priori} difference between boundedness and algebraic hyperbolicity is the \emph{linearity} of the dependence on the genus in the definition. 
 
\begin{theorem}[From boundedness to uniformity]\label{thm:uniformity}  
  If $X$ is bounded modulo $\Delta$ over $k$, then, for every ample line bundle $\mathcal{L}$ on $X$ and every integer $g\geq 0$, there is a real number $\alpha(X, \Delta, \mathcal{L}, g)$ such that, for every smooth projective irreducible curve $C$ of genus $g$ over $k$  and every morphism $f:C\to X$ with $f(C)\not\subset \Delta$, the inequality 
  \[
  \deg_C f^\ast \mathcal{L} \leq \alpha(X, \Delta, \mathcal{L}, g) 
  \] holds. 
\end{theorem}

\subsection{Outline of paper}
In Section \ref{section2} we provide several useful characterizations of boundedness (and algebraic hyperbolicity) that we will use in our proofs later. For example, we show that a projective scheme $X$ over $k$ is bounded over $k$ if and only if, for every reduced projective scheme $C$ equidimensional of dimension one over $k$, the moduli space of maps from $C$ to $X$ is of finite type over $k$; see Theorem \ref{thm:test_on_reds} for a precise statement.

In Section \ref{section3} we study the problem of extending maps from curves to a scheme over the function  field of a Dedekind domain. The main result here is Corollary \ref{cor:homisproper}.

With the preliminary results in Section \ref{section2} and Section \ref{section3} at hand, we proceed to prove the results stated in the introduction on the Zariski-countable openness of several loci in Section \ref{section:zco}. We start by proving Theorem \ref{thm:demailly} on the Zariski-countable openness of algebraic hyperbolicity, and then prove similar statements for boundedness. We then prove  Theorem \ref{thm:specialisation} on specialising pseudo-algebraic hyperbolicity.

Section \ref{section:zco} also contains a proof of the Zariski-countable openness of the locus of varieties for which every subvariety is of general type (Theorem \ref{thm:specialisation}); this is essentially an application of Nakayama's theorem (Theorem \ref{thm:nakayama}). Section \ref{section:zco} also contains a proof of the fact that all notions of hyperbolicty discussed in this paper are stable under generisation (Theorem \ref{thm:generisation}).

Section \ref{section:zco} concludes with a proof of Theorem \ref{thm:mildly-bounded-is-zco} which says that the locus of mildly bounded varieties is Zariski-countable open. We then deduce several consequences (Corollary \ref{cor:mildly-bounded-gens} and Corollary \ref{cor:mb_persists}).
 
 Section \ref{section:mb2} is dedicated to proving the mild boundedness of semi-abelian varieties (Proposition \ref{prop:semiabvar_is_mildbounded}). Do note that we include stronger results for the sake of future reference (see, for example, Proposition \ref{prop:semiabelian-r+g-points0}). We use the mild boundedness of semi-abelian varieties to prove that projective surfaces with no rational curves and of "maximal Albanese dimension"   are mildly bounded (Theorem \ref{thm:surfaces_are_mb}). 
 
 In Section \ref{section6} we apply Silverman's specialisation theorem (for families of abelian varieties), the finiteness theorem of Arakelov-Parshin,  and a "uniform" version of the fact that abelian varieties are  mildly bounded to give several examples of mildly bounded varieties. In particular, we prove Theorem \ref{thm:families_of_ab_var_are_mb}.

In Sections \ref{section7} and \ref{section:faltingsvojta} we study the persistence of arithmetic hyperbolicity. The proofs here  rely on the results of Sections \ref{section:mb2} and Section \ref{section6} (notably the mild boundedness of semi-abelian varieties). For example, we  prove Theorem \ref{thm:semi_abvars_intro}. Also, we show that a projective groupless surface which admits a non-constant map  to an abelian variety is mildly bounded (Theorem \ref{thm:surfaces_q}). We then focus on  Theorem \ref{thm:surfaces_intro}; this can be viewed as our main result on the Persistence Conjecture (Conjecture \ref{conj:pers}).

In the final section, we look at different notions of hyperbolicity and boundedness modulo a closed subset $ \Delta$, where all failures of a given variety to have the property in question are contained within $\Delta$. We use this to introduce several "pseudofied" notions of hyperbolicity, and prove  Theorems \ref{thm:ps_alg}, \ref{thm:alghyp_to_bounded}, and \ref{thm:uniformity}.

\begin{ack} 
We thank Frank Gounelas for explaining to us that being of general type is a Zariski-open condition (Theorem \ref{thm:nakayama}). We are grateful to Jason Starr for many helpful discussions, and especially his help in proving Corollary \ref{cor:homisproper}.   We also thank  Damian Brotbek and Carlo Gasbarri for helpful discussions on hyperbolicity in families. We thank Ekaterina Amerik for motivating discussions on mildly bounded varieties. We are most grateful to the referee for their very careful reading of our manuscript, the many useful comments (especially concerning the proof of Lemma \ref{lemma:closed_imm}), and their patience.
 The first  and second named authors gratefully acknowledge support from SFB Transregio/45.  The third named author is partially supported by a grant from the 
Simons Foundation/SFARI (522730, LK). 
\end{ack}

\begin{con}
Throughout this paper, we let $k$ denote an algebraically closed field of characteristic zero.  A variety over $k$ is a finite type separated reduced $k$-scheme.

Let $X$ be a finite type separated scheme over $k$ and let $A\subset k$ be a subring. A model for $X$ over $A$ is a pair $(\mathcal{X},\phi)$ with $\mathcal{X}\to \Spec A$ a finite type separated scheme and $\phi:\mathcal{X}\times_A k \to X$ an isomorphism of schemes over $k$. We will often omit $\phi$  from our notation.

If $X$ is a projective variety and the base field (which is usually $k$) is understood we will sometimes simply say that "$X$ is bounded" (instead of "$X$ is bounded over $k$"). 
\end{con}

 \section{Characterising boundedness and algebraic hyperbolicity}\label{section2}
 Recall that  a projective variety $X$ over   $k$ is algebraically hyperbolic (over $k$) if there is an ample line bundle $\mathcal{L}$ on $X$ and a real number $\alpha_{X,\mathcal{L}}$ such that, for every smooth projective irreducible curve $C$ over $k$ of genus $g$ and every morphism $f:C\to X$, the inequality \[ \deg_C f^\ast \mathcal{L} \leq \alpha_{X,\mathcal{L}} \cdot g \] holds. We mention that one could also ask for a weaker bound on $\deg_C f^\ast \mathcal{L}$ which only depends on   $X$, $L$ and the genus $g$ of $C$, but which is not necessarily linear in $g$. This leads to an a priori weaker notion of boundedness which is referred to as \emph{weakly boundedness} by Kov\'acs-Lieblich \cite[\S 1]{KovacsLieblich}. By \cite[Theorem~1.14]{JKa}, their notion is the same as  the notion of \emph{boundedness} introduced in \cite{JKa} (which is the notion we use in this paper).
 
  In this section we first record the fact  that maps from a possibly singular curve into an algebraically hyperbolic projective variety also satisfy similar boundedness properties.   We then prove similar statements for the notions of boundedness, $(n,m)$-boundedness and mild boundedness.
 
 \subsection{Testing algebraic hyperbolicity}
 Our starting point is the following characterisation of algebraic hyperbolicity.
  The geometric genus $p_g(D)$ of $C$ is the sum of the genera of the components $D_i$ ($i = 1, \ldots, n$) of the normalisation $\widetilde{D}$ of $D$. 
 
\begin{lemma}  \label{lem:stablehyperbolicity}
Let $X$ be an algebraically hyperbolic projective scheme over $k$ with an ample line bundle $\mathcal{L}$. Let $\alpha_{X,\mathcal{L}}$ be a real number such that, for every smooth projective irreducible curve $C$ and every morphism $f \colon C \rightarrow X$, the inequality $\deg f^*\mathcal{L} \leq \alpha_{X,\mathcal{L}} \cdot g(C)$ holds. Then, for every   reduced  projective scheme $D$ equidimensional of dimension 1  over $k$ and every morphism $f \colon D \rightarrow X$, the inequality  $\deg f^*\mathcal{L} \leq \alpha_{X,\mathcal{L}} \cdot p_g(D)$ holds.
\end{lemma}

\begin{proof}
Let $\pi \colon \widetilde{D} \rightarrow D$ be the normalisation. Then $\deg \pi^*f^* \mathcal{L} = \deg f^*\mathcal{L}$ by \ \cite[Proposition~7.3.8]{Liu2002}. Moreover, by assumption, the inequality $$   \deg \pi^*f^* \mathcal{L} \leq \sum_{i=1}^n \alpha_{X,\mathcal{L}} \cdot g(D_i) $$ holds. Also, by definition, we have that $$ \sum_{i=1}^n    g(D_i)=  p_g(D). $$ This implies the statement. 
\end{proof}

    In some sense,  most projective varieties should be algebraically hyperbolic, and this philosophy is confirmed by the work of many authors   \cite{Brotbek1, Brotbek2, ChL, Deb, Demailly, Div, DivFer, Mourougane, RoulleauRousseau, Rousseau1}.   
 
 \subsection{Testing boundedness}
The following lemma will allow us to test boundedness on reduced curves (Theorem \ref{thm:test_on_reds}).
 
 \begin{lemma}\label{lemma:closed_imm}
 Let $X$ be a projective scheme over $k$. Let $g \colon C'\to C$ be a finite  surjective birational morphism of projective reduced schemes  equidimensional of dimension one  over $k$. Then, the natural morphism of schemes
 \[
 \underline{\Hom}_k(C,X) \to \underline{\Hom}_k(C',X)
 \] is a closed immersion.
 \end{lemma}
 \begin{proof}
Let $P\in \mathbb{Q}[t]$ be a polynomial and let $H' := \underline{\Hom}_k^P(C',X)\subset \underline{\Hom}_k(C',X)$ be the subscheme of morphisms $C'\to X$ with Hilbert polynomial $P$.  (We compute the Hilbert polynomial with respect to fixed ample line bundles on $C'$ and $X$, respectively.) Note that $H'$ is a quasi-projective scheme over $k$.

Let $H\subset \underline{\Hom}_k(C,X)$ be the inverse image of $H'$ under the morphism  $\underline{\Hom}_k(C,X) \to \underline{\Hom}_k(C',X)$. To prove the lemma, it suffices to show that $H\to H'$ is a closed immersion.  (In particular,  there are   finitely many polynomials $Q_1,\ldots, Q_n\in \mathbb{Q}[t]$ such that $H \subset  \underline{\Hom}_k^{Q_1}(C,X)  \sqcup \ldots \sqcup  \underline{\Hom}_k^{Q_n}(C,X)  
$.)
     
    To show that $H \to H'$ is a closed immersion, we will   construct a morphism of sheaves of modules $\mathcal{N}_{H'}\to \mathcal{O}_{H'}$  such that the closed subscheme of $H'$ defined by the image of this morphism exactly corresponds to the image of $H$ in $H'$. In order to construct this morphism $\mathcal{N}_{H'} \to \mathcal{O}_{H'}$, we will make use of a technical result in \cite{EGAIII}.
    
     Let $F' \colon H' \times C' \to X$ be the   evaluation morphism.  Note that the graph of $F'$ given by $$E'=(\mathrm{pr}_1, \mathrm{pr}_2, F') \colon H' \times C' \longrightarrow H' \times C' \times X$$       is a closed immersion.  Let
$$
{E'}^\# \colon {\mathcal O}_{H' \times C' \times X} \longrightarrow \mathop{E'_*} {\mathcal O}_{H' \times C'}$$ be the induced morphism. Let $G:=(\mathrm{pr}_1, g\circ \mathrm{pr}_2, \mathrm{pr}_3) \colon H'\times C' \times X\to H'\times C \times X$.    Let $\mathcal A$ be the image of the morphism  ${\mathcal O}_{H' \times C \times X}\longrightarrow \mathop{G_*E'_*} {\mathcal O}_{H' \times C'}$ defined as the composition 
\[\xymatrix{
{\mathcal O}_{H' \times C \times X}         \ar[rr]^{G^\#} &               & 
\mathop{G_*} {\mathcal O}_{H' \times C' \times X}    \ar[rr]^{G_\ast E'^\#}  & & \mathop{G_*E'_*} {\mathcal O}_{H' \times C'}}.
\] Note that $\mathcal{A}$ is a coherent sheaf on the quasi-projective $k$-scheme $H'\times C \times X$.
Let $(\mathrm{pr}_1, \mathrm{pr}_2) \colon H' \times C \times X \rightarrow H' \times C$ be the projection, and note that $(\mathrm{pr}_1, g \circ \mathrm{pr}_2) \colon H' \times C' \rightarrow H' \times C$ is the composition $(\mathrm{pr}_1, \mathrm{pr}_2) \circ G \circ E'$. Then we consider $\mathop{(\mathrm{pr}_1, \mathrm{pr}_2)_*} \mathcal{A}$ as a subsheaf of $$\mathop{(\mathrm{pr}_1, \mathrm{pr}_2)_* G_* E'_*} \mathcal{O}_{H'\times C'} = \mathop{(\mathrm{pr}_1, g \circ \mathrm{pr}_2)_*} \mathcal{O}_{H' \times C'}.$$
To summarise informally, for any open $U \subset H' \times C$ with inverse image $V \subset H' \times C'$, we can describe the sheaves that we are considering as follows:
\[
\begin{tikzcd}
\mathrm{Im}(\mathrm{pr}_1, g \circ \mathrm{pr}_2)^\#(U) \arrow[r,hook]{r}{\dagger_U} \arrow[d,equal]
&(\mathrm{pr}_1, \mathrm{pr}_2)_{*} \mathcal{A}\,(U) \arrow[r,hook] \arrow[d,equal]
&(\mathrm{pr}_1, g \circ \mathrm{pr}_2)_{*} \mathcal{O}_{H' \times C'}(U) \arrow[d,equal]	\\
\left\{\begin{array}{c}t \textrm{ only depending}\\ \textrm{on $\rho$ and $g(c')$}\end{array}\right\} \arrow[r,hook]{r}{\dagger_U}
&\left\{\begin{array}{c}t \textrm{ only depending}\\ \textrm{on $\rho$, $g(c')$, and $\rho(c')$}\end{array}\right\} \arrow[r,hook]
&\{\textrm{regular functions } t(\rho, c') \textrm{ on } V\}
\end{tikzcd}
\]
In this diagram, $\dagger_U$ is a map of $\mathcal{O}_{H' \times C}(U)$-modules coming from the map $\dagger \colon \mathrm{Im}(\mathrm{pr}_1, g \circ \mathrm{pr}_2)^\# \to (\mathrm{pr}_1, \mathrm{pr}_2)_* \mathcal{A}$ of $\mathcal{O}_{H' \times C}$-modules, $\rho$ denotes a set of coordinates for $H'$ and $c'$ a set of coordinates for $C'$, i.e., we think of $\rho$ as describing a morphism $C' \to X$ and $c'$ as describing a point in $C'$. By abuse of language, we use $g(c')$ and $\rho(c')$ to describe the coordinates of the images of this point in $C$ and $X$ respectively, i.e., $g(c')$ denotes $g^\#(c')$ and $\rho(c')$ denotes $F'^\#(x)$, where $x$ are coordinates for $X$.

Note that the sheaf $\mathrm{Coker} (\mathrm{pr}_1, g \circ \mathrm{pr}_2)^\#$ is supported on $H'\times C^{\textrm{sing}}$ and   free of finite rank as $\mathcal{O}_{H'}$-module (through the projection on the first coordinate).
Consider the composite morphism of sheaves on $H'\times C$ given by $$\mathop{(\mathrm{pr}_1, \mathrm{pr}_2)_*} \mathcal A \longhookrightarrow \mathop{(\mathrm{pr}_1, g \circ \mathrm{pr}_2)_*} {\mathcal O}_{H' \times C'} \twoheadrightarrow \mathrm{Coker} (\mathrm{pr}_1, g \circ \mathrm{pr}_2)^\#,$$ call it $f_{H'}$.  
Now, for any $k$-scheme $S$, we   consider  a family   of morphisms   from $C'_S$ to $X_S$  over $S$ as a morphism $\varphi \colon S \to H'$.  For every morphism $\varphi\colon S\to H'$, we denote the induced morphism $(\varphi \circ \mathrm{pr}_1,\mathrm{pr}_2):S\times C\to H'\times C$ by $\varphi_C$. Given   a  morphism $\varphi \colon S\to H'$, we then consider the pull-back of $f_{H'}$ to $S\times C$:   \[f_{H',\varphi}:= \mathop{\varphi_C^\ast} f_{H'} \colon \mathop{\varphi_C^\ast} \mathop{(\mathrm{pr}_1, \mathrm{pr}_2)_*} \mathcal A \rightarrow  \mathop{\varphi_C^\ast} \mathrm{Coker} (\mathrm{pr}_1, g \circ \mathrm{pr}_2)^\#.\] 
We claim that $f_{H',\varphi}$
 is the zero map if and only if $\varphi$ factors through $H\to H'$. Indeed, informally speaking, $f_{H',\varphi}$ is zero if and only if the map $\dagger$ becomes an isomorphism after pulling back to $S\times C$, which is exactly the case when locally on $S \times C$, the coordinates of $\rho(c')$ can be described in terms of those of $\rho$ and $g(c')$, i.e., when the morphism $C'_{S} \to X_{S}$ (associated to $\varphi$ considered as $S$-point of $H'$) factors through $g_S \colon C'_S \to C_S$. We give a precise argument   in Lemma \ref{lemma:formal-lemma}.

We will apply \cite[Corollaire~7.7.8]{EGAIII} to the   projection morphism $H' \times C \to H'$, which  is projective, and the sheaves ${\mathcal F} := \mathrm{Coker} (\mathrm{pr}_1, g \circ \mathrm{pr}_2)^\#$ and ${\mathcal G} := (\mathrm{pr}_1, \mathrm{pr}_2)_* \mathcal A$.  Note that $\mathcal{G}$ is a coherent sheaf on the quasi-projective scheme $H'\times C$, as $(\mathrm{pr}_1,\mathrm{pr}_2)$ is proper \cite[Tag~02O5]{stacks-project}.  In particular, as the resolution property holds for all quasi-projective schemes over a field, \cite[Tag~0F87]{stacks-project}, the sheaf $\mathcal{G}$ is the cokernel of a morphism of finite rank locally free sheaves on $H'\times C$.

By \cite[Corollaire~7.7.8]{EGAIII}, $f_{H'}$ is equivalent to a  morphism $\tilde{f}_{H'} \colon \mathcal N_{H'} \to {\mathcal O}_{H'}$. We claim that this morphism has the property that the base change $\tilde{f}_{H',\varphi} \colon \mathop{\varphi^*} \mathcal N_{H'} \to \mathop{\varphi^*} {\mathcal O}_{H'}$ is the zero map if and only if $f_{H',\varphi}$ is the zero map.
Indeed, by adjointness,  $\tilde{f}_{H',\varphi}$ is zero if and only if the associated morphism $\mathcal N_{H'} \to \mathop{\varphi_* \varphi^*} \mathcal{O}_{H'}$ is zero. By loc.\ cit., this is equivalent to the associated morphism $\mathcal{G} \to \mathcal{F} \otimes_{\mathcal{O}_{H'}} \mathop{\varphi_* \varphi^*} \mathcal{O}_{H'}$ being zero. As $\mathcal{F}$ is free of finite rank (as an $\mathcal{O}_{H'}$-module), we can compose with the natural isomorphism coming from the projection formula (see  \cite[Tag 01E8]{stacks-project}) 	 to get that this is equivalent to the map $\mathcal{G} \to \mathop{\varphi_* \varphi^*} \mathcal{F}$ being zero. Finally, by adjointness, this is equivalent to $f_{H', \varphi} \colon \mathop{\varphi^*} \mathcal{G} \to \mathop{\varphi^*} \mathcal{F}$ being zero, as claimed.

Therefore, $\tilde{f}_{H',\varphi}$ is the zero map if and only if $\varphi$ factors through $H \to H'$. In particular, the closed subscheme defined by $\mathrm{Im}(\tilde{f}_{H'})$ is exactly the image $H$ inside $H'$, as we wanted to prove.
 \end{proof}
 
 \begin{lemma}\label{lemma:formal-lemma}
 In the notation used in the proof of Lemma \ref{lemma:closed_imm}, the map $$f_{H', \varphi} \colon \mathop{\varphi_C^\ast} \mathop{(\mathrm{pr}_1, \mathrm{pr}_2)_*} \mathcal A \rightarrow  \mathop{\varphi_C^\ast} \mathrm{Coker} (\mathrm{pr}_1, g \circ \mathrm{pr}_2)^\#$$ of $\mathcal{O}_{S \times C}$-modules is the zero map if and only if $\varphi \colon S \to H'$ factors through $H$.
 \end{lemma}
 
 \begin{proof}
 The $S$-point $\varphi$ of $H'$ corresponds to a morphism $F_{\varphi}' \colon S \times C' \to X$. The factoring of $\varphi$ through $H$ is equivalent to $F_{\varphi}'$ factoring through some morphism $F_{\varphi} \colon S \times C \to X$. This is equivalent to the existence of a map $E_{\varphi}$ as in the commutative diagram below; we will have $E_{\varphi} = (\mathrm{pr}_1, \mathrm{pr}_2, F_{\varphi})$.

To prove the lemma,  we  argue on affine charts.
Let $(s,c) \in S \times C$ be some point. Its inverse image $\{s\} \times g^{-1}(c) \subset S \times C'$ is finite. As $X$ is projective, there exists an affine open $\mathrm{Spec}(R) \subset X$ containing the finite set $F'_\varphi(\{s\} \times g^{-1}(c))$. Then $(S \times C') \setminus (F'_\varphi)^{-1}(\mathrm{Spec}(R))$ is a closed subset of $S \times C'$, and as $g$ is finite and hence proper, the image of this closed subset under $(\mathrm{pr}_1, g \circ \mathrm{pr}_2)$ is also closed. Let $\mathrm{Spec}(A) \subset S \times C$ be an affine open neighbourhood of $(s,c)$ contained in the complement of this image. As $g$ is finite and hence affine, the inverse image of $\mathrm{Spec}(A)$ in $S \times C'$  is also affine, say $\mathrm{Spec}(A')$. Moreover, $A \to A'$ is an injective map, because $g \colon C' \to C$ is surjective, $C$ is reduced, and $S \to \mathrm{Spec}(k)$ is flat. The diagram on the right indicates the corresponding diagram of rings on the affine opens that we just defined (note that all tensor products below are over $k$). In particular,  $\theta$ is the natural map corresponding to the restriction of $(\mathrm{pr}_1, \mathrm{pr}_2, F_{\varphi}')$. 
 
 \[\begin{tikzcd}[row sep=large, column sep=huge]
 	S \times C' \arrow{r}{(\mathrm{pr}_1, \mathrm{pr}_2, F'_{\varphi})} \arrow[swap]{d}{(\mathrm{pr}_1, g \circ \mathrm{pr}_2)} \arrow{dr} & S \times C' \times X \arrow{d}{(\mathrm{pr}_1, g \circ \mathrm{pr}_2, \mathrm{pr}_3)} & &
 	A' & A' \otimes R \arrow[swap]{l}{\theta}
 	 \\ 	
 	S \times C \arrow[dotted]{r}{E_\varphi} &S \times C \times X & &
 	A \arrow[hook]{u} &A \otimes R \arrow[hook]{u} \arrow[dotted]{l} \arrow[swap]{ul}{\theta|_{A \otimes R}}
 \end{tikzcd}\]
 
We now explain what the condition that $f_{H',\varphi} = 0$ means on these affine charts. To do so, note that the pullback of $\mathcal{A}$ to $S\times C\times X$ is  the image of $\mathcal{O}_{S \times C \times X}$ along the diagonal map in the diagram,  which on the level of rings corresponds to $\theta(A \otimes R)$, considered as an $A \otimes R$-module. Moreover, $\mathop{\varphi_C^*} \mathop{(\mathrm{pr}_1,\mathrm{pr}_2)_*} \mathcal{A}$ corresponds to $\theta(A \otimes R)$, but now considered as an $A$-module, and $\mathop{\varphi_C^*} \mathrm{Coker}(\mathrm{pr}_1, g \circ \mathrm{pr}_2)^\#$ corresponds to $A/A'$, also considered as an $A$-module.

Therefore, $f_{H',\varphi} = 0$ if and only if the map $\theta(A \otimes R) \to A/A'$  of $A$-modules is zero, which is equivalent to $\theta(A \otimes R) \subset A$. This is exactly the case when $A \otimes R \to A'$ factors through $A$, i.e.\ when the map indicated by the dotted arrow in the right hand diagram exists. As this map will be unique, these maps   glue when we vary the point $(s,c)$ and we obtain the map $E_{\varphi}$. We conclude that $f_{H',\varphi} = 0$ if and only if there exists a map $E_{\varphi}$ as in the commutative diagram above.  Since the latter map $E_{\varphi}$ exists if and only if $\varphi$ factors over $H\to H'$,  this concludes the proof.
\end{proof}
 
We now use  Lemma \ref{lemma:closed_imm} to show that a projective scheme $X$ is bounded  if and only if, for every reduced curve $C$, the moduli space of maps from $C$ to $X$ is of finite type.
 
\begin{theorem}[Testing boundedness on reduced objects]\label{thm:test_on_reds}
Let $X$ be a projective scheme over $k$.  Then the following are equivalent.
\begin{enumerate}
\item The projective scheme $X$ is bounded over $k$.
\item The projective scheme $X$ is $1$-bounded over $k$, i.e., for every smooth projective irreducible curve $C$ over $k$, the scheme $\underline{\Hom}_k(C,X)$ is of finite type over $k$.
\item For every  reduced  projective (not necessarily irreducible nor smooth) scheme  $C$ equidimensional of dimension one over $k$, the scheme $\underline{\Hom}_k(C,X)$ is of finite type over $k$.
\end{enumerate}
\end{theorem}
 \begin{proof} By definition, a bounded variety is $1$-bounded, so that $(1)\implies (2)$. Moreover, 
 by \cite[Theorem~9.3]{JKa}, we have that $(2)\implies (1)$. It is clear that $(3) \implies (2)$. Therefore, to prove the theorem, it suffices to show that $(2)\implies (3)$.
 
 Assume that $X$ satisfies $(2)$.
Let $C$ be a projective reduced scheme equidimensional of dimension one over $k$. Let $C_1,\ldots, C_n$ be the irreducible components of $C$. Let $C'$ be the normalisation of $C$ in the product of function fields $K(C_1)\times \ldots \times K(C_n)$, and note that $C'\to C$ is a finite birational surjective morphism. Moreover, $C'$ is a smooth projective curve over $k$. For $i = 1,\ldots, n$, let $C'_i$ be the connected component of $C'$ lying over $C_i$.  By Lemma \ref{lemma:closed_imm}, the natural morphism of schemes
\[
\underline{\Hom}_k(C,X) \to \underline{\Hom}_k(C',X)=\prod_{i=1}^n \underline{\Hom}_k(C'_i,X)
\] is a closed immersion. Since $X$ satisfies property $(2)$, for all $i=1,\ldots,n$, the scheme $\underline{\Hom}_k(C_i',X)$ is of finite type over $k$.  Since closed immersions of schemes are of finite type, we conclude that $\underline{\Hom}_k(C,X)$ is of finite type over $k$.
 \end{proof}
 
We now prove a similar result for ``pointed boundedness''.

 \begin{proposition}[Testing pointed boundedness on reduced objects]  
\label{prop:mnbounded} Let $X$ be a projective variety over $k$, let $n\geq 1$ and let $m\geq 1$ be an integer.  Then the following are equivalent.
\begin{enumerate}
\item The projective variety $X$ is $(n,m)$-bounded over $k$.
\item For every   projective connected variety $C$ equidimensional of dimension one over $k$, every $c$ in $C(k)$, and every $x$ in $X(k)$, the set  $ \Hom_k([C,c],[X,x])$ is   finite.
\end{enumerate}
 \end{proposition}
 \begin{proof} Note that $(2)$ implies that $X$ is $(1,1)$-bounded. Therefore, by \cite[Proposition~8.2~and~Remark~4.3]{JKa}, it follows that $X$ is $(n,m)$-bounded. This shows that $(2)\implies (1)$.
  
 To prove that $(1)\implies (2)$, we argue as follows. First, as $X$ is $(n,m)$-bounded, it follows from \cite[Lemma~4.6]{JKa} and \cite[Theorem~8.4]{JKa} that, for every smooth projective irreducible curve $C'$ over $k$, every $c'$ in $C'(k)$, and every $x$ in $X(k)$, the set $\Hom_k([C',c'],[X,x])$ is finite. Now, let $C$ be a projective connected variety equidimensional of dimension one over $k$, let $c$ be in $C(k)$, and let $x$ be in $X(k)$. We prove that $\Hom([C,c],[X,x])$ is a finite set by induction on the number $N$ of irreducible components of $C$ not containing $c$.  Let $C_1,\ldots,C_n$ be the irreducible components of $C$. Let $C'\to C$ be the normalisation of $C$ in the product of the function fields $K(C_1)\times \ldots \times K(C_n)$. Let $C_i'$ be the connected component of $C'$ lying over $C_i$.
 
 Assume that $N=0$, i.e., $c$ lies on every irreducible component of $C$.   For every $i=1,\ldots,n$, let $c_i'$ in $C_i'$ be a point mapping to $c$ in $C$.   As $C'\to C$ is surjective, the natural map of sets
 \[
 \Hom_k([C,c],[X,x]) \to \prod_i \Hom_k([C'_i,c'_i],[X,x])
 \] is injective. Since $C'_i$ is a smooth projective    irreducible  curve over $k$, for every $i$, the set $\Hom_k([C_i',c_i'],[X,x])$ is finite, so that $\Hom_k([C,c],[X,x])$ is finite, as required.  
 
 If $N>0$, after renumbering if necessary, we may and do assume that $c$ does not lie on $C_n$ and that the projective reduced scheme $D:= C_1\cup \ldots\cup C_{n-1}$ is connected.   By the induction hypothesis, the set 
 \[
 \Hom_k([D,c],[X,x]) 
 \] is finite. Let $d$ be a point in $D\cap C_n \subset C$. Then, the set
 \[
A:= \{ f(d) \ | \ f\in \Hom_k([D,c],[X,x])\}
 \] is a finite subset of $X$. Therefore,  as the map of sets
 \[
 \Hom_k([C,c],[X,x]) \subset \Hom_k([D,c],[X,x]) \times \bigcup_{y\in A} \Hom_k([C_n,d],[X,y])
 \] is injective and $\Hom_k([C_n,d],[X,y])$ is finite, we conclude that $\Hom_k([C,c],[X,x])$ is finite.
 \end{proof}

 \subsection{Testing mild boundedness} We show that the notion of mild boundedness (Definition \ref{defn:mild_bounded}) can   be tested on reduced curves. The curves do not even need to be irreducible in this case.

\begin{lemma}\label{lem:mildly-bounded-for-stable}
Let $X$   be a mildly bounded variety over $k$ and let $C$ over $k$ be a   finite type separated reduced scheme over $k$ whose irreducible components are one-dimensional. Then there exist an integer $n$ and distinct points $c_1, \ldots, c_n \in C(k)$ such that for every $x_1, \ldots, x_n \in X(k)$ the scheme $$\underline{\Hom}_k([C,(c_1, \ldots, c_n)], [X, (x_1, \ldots, x_n)])$$ is finite over $k$. 
\end{lemma}
\begin{proof}
Let $C_1, \ldots, C_{\ell}$ be the irreducible components of $C$. For $i = 1,\ldots,{\ell}$, let $D_i$ be the locus of points in $C_i$ that are smooth as a point of $C$. As $X $ is mildly bounded over $k$, there exist points $d_{1,1}, \ldots, d_{1,n_1} \in {D_1}(k)$, points $d_{2,1}, \ldots, d_{2,n_2} \in {D_2}(k)$, ..., and points $d_{\ell,1}, \ldots, d_{\ell, n_{\ell}} \in {D_{\ell}}(k)$ such that $$H_i = \underline{\Hom}_k([{D_i}, (d_{i,1}, \ldots, d_{i,n_i})], [X, (x_1, \ldots, x_{n_i})])$$ is finite for every $i = 1,\ldots, \ell$ and $x_1, \ldots, x_{n_i} \in X(k)$.

Let $c_1, \ldots, c_n$ be the points $d_{1,1}, \ldots, d_{\ell, n_{\ell}}$ considered as points in $C(k)$. Then we claim that
$$H = \underline{\Hom}_k([C, (c_1, \ldots, c_n)], [X, (x_1, \ldots, x_n)])$$
is finite for all $x_1, \ldots, x_n \in X(k)$. For any morphism $f \colon C \to X$ such that $f(c_i) = x_i$ for all $i = 1,\ldots,n$, the composition $f_i \colon {D_i} \to C \to X$ lies in the finite scheme $H_i$. As a morphism is fixed by its restriction to the dense open subscheme $\bigcup_{i=1}^{\ell} D_i$, this gives an embedding $H \hookrightarrow H_1 \times \ldots \times H_{\ell}$, which proves that $H$ is finite over $k$.
\end{proof}

 \section{Extending maps over Dedekind schemes}\label{section3}
 A morphism of schemes $\mathcal{C} \to S$ is a semi-stable curve over $S$ if it is a proper flat morphism whose geometric fibres are connected semi-stable curves; see \cite[Definition~10.3.1]{Liu2002}.

 The following lemma is  a mild generalisation of a well-known extension property for rational maps from big opens of \emph{smooth} varieties to projective varieties with no rational curves.

  \begin{lemma}\label{lem:extension1} Let $S$ be a   Dedekind scheme with function field $K$, and let $X\to S$ be a projective  morphism of schemes. Let $\mathcal C$ be a  semi-stable curve over  $S$ such that $\mathcal C_K$ is smooth projective irreducible.
 If every geometric closed fibre of $X\to S$ contains no rational curves, then every morphism $\mathcal C_K\to X_K$ extends to a morphism $\mathcal C\to X$.
 \end{lemma}
 \begin{proof}  
  Since $\mathcal{C} \to S$ is semi-stable, the exceptional locus of the  minimal resolution of singularities $\mathcal{C'} \to \mathcal{C}$   is a forest of $\mathbb P^1$'s; see \ \cite[Corollary~10.3.25]{Liu2002}. Also, as the geometric closed fibres of $X\to S$ do not contain any rational curves, every geometric fibre of $X\to S$ contains no rational curves \cite[Theorem 1.5]{JVez}. In particular, the geometric fibres of $X \times_S \mathcal{C'} \rightarrow \mathcal{C'}$ do not contain rational curves. Therefore,  since $\mathcal{C}'$ is an integral regular noetherian scheme, the rational section $\mathcal{C}' \dashrightarrow X\times_S \mathcal{C}'$ extends to a morphism  \cite[Proposition~6.2]{GLL}. We now show that the induced morphism $\mathcal{C}'\to X$ factors over $\mathcal{C}$.

Indeed, note that the rational curves in the exceptional locus of $\mathcal{C'} \to \mathcal{C}$ are contracted to a point in $X$. This implies that the morphism $\mathcal{C'} \to X$   factors over a morphism $\mathcal{C} \to X$; see \cite[Lemme~8.11.1]{EGAII}.
 \end{proof}

  \begin{corollary}\label{cor:homisproper} 
Let  $X \to S$ be a projective morphism of noetherian schemes such that all geometric fibres do not contain a rational curve. Fix a relatively ample line bundle $\mathcal{L}$ on $X$. Let $g,n \geq 0$ be integers such that $2g - 2 + n > 0$, and let $\univcurvebar{g,n} \to  \overline{\mathcal{M}_{g,n}}$ be the universal curve over the stack $\overline{\mathcal{M}_{g,n}}$ of $n$-pointed stable curves of arithmetic genus $g$ over $S$.  Let $d\geq 0$ be an integer and let $\mathcal{H}_d^{g,n} = \underline{\Hom}_{\overline{ \mathcal{M}_{g,n}}}^d(\univcurvebar{g,n}, X \times \overline{\mathcal{M}_{g,n}})$ be the algebraic stack of morphisms of degree $d$ with respect to $\mathcal{L}$. Then the natural morphism  $\rho \colon \mathcal{H}_d^{g,n} \to   \overline{\mathcal{M}_{g,n}}$ is proper. 
 \end{corollary}

 \begin{proof} It suffices to prove   the existence part of the valuative criterion for properness.
We prove this    by induction on $g$ and $n$. Let $R$ be a discrete valuation ring with field of fractions $K$. Let $\mathcal{V}_d^{g,n} = \rho^{-1}(\mathcal{M}_{g,n})$. Suppose we have a point $\varphi \in \mathcal{H}_d^{g,n}(K)$ and a curve $\mathcal{C} \in \overline{\mathcal{M}_{g,n}}(R)$ such that $\mathcal{C}_K \cong \rho(\varphi)$.

In the case $\varphi \in \mathcal{V}_d^{g,n}(K)$, the morphism $\varphi \colon \mathcal{C}_K \to X_K$ extends to a morphism $\mathcal{C} \to X_R$ by Lemma \ref{lem:extension1}.
In the case $\varphi \notin \mathcal{V}_d^{g,n}(K)$, the point lies in the image of one of the clutching morphisms as described in \cite[Definition~3.8]{KnudsenII}. As these clutching morphisms are finite and hence proper, the statement now follows from the induction hypothesis. 
 \end{proof}

\section{Zariski-countable openness of the hyperbolic locus}\label{section:zco}
 Let $S$ be a noetherian scheme over $\QQ$ and let $X \to S$ be a projective morphism. We start with presenting Demailly's proof of Theorem \ref{thm:demailly} by using the language of algebraic stacks. Following Demailly, we will make use of the following simple lemma.

\begin{lemma}\label{lem:countabletopology}
Let $(X, \mathcal{T})$ be a noetherian topological space. Then there exists another topology $\mathcal{T}^{\mathrm{cnt}}$, or $\mathcal{T}$-countable, on $X$ whose closed sets are the countable union of $\mathcal{T}$-closed sets.
\end{lemma}

\begin{proof}
The only non-trivial thing to check is that an arbitrary intersection $$V = \bigcap_{i \in I} \bigcup_{j=1}^{\infty} C_{ij},$$ where $C_{ij} \subset X$ are $\mathcal{T}$-closed sets, is $\mathcal{T}^{\mathrm{cnt}}$-closed. The proof goes by noetherian induction on $(X,\mathcal{T})$. If for every $i \in I$ there is a $j \in \ZZ_{>0}$ such that $C_{ij} = X$, then we are done. If this is not the case, we take an $i \in I$, such that $C_{ij} \subsetneq X$ for all $j \in \ZZ_{>0}$. Then for every $j \in \ZZ_{>0}$ the set $V \cap C_{ij}$ is $\mathcal{T}^{\mathrm{cnt}}$-closed in $C_{ij}$ and hence in $X$. Then the union $V = \bigcup_{j=1}^\infty V \cap C_{ij}$ is also $\mathcal{T}^{\mathrm{cnt}}$-closed.
\end{proof}

\begin{remark}[Varieties with no rational curves]\label{remark:pure}
We follow \cite[\S3]{JKa} and say that a proper variety $X$ over an algebraically closed field  $F$ is \emph{pure} if and only if, for every smooth variety $T$ over $F$ and every dense open subscheme $U\subset T$ with $\mathrm{codim}(T\setminus U)\geq 2$, we have that every morphism $U\to X$ extends to a morphism $T\to X$. Note that a proper variety $X$ over $F$ is  {pure} if and only if it has no rational curves \cite[Lemma 3.5]{JKa}. This terminology will allow us to simplify some of the proofs below.
\end{remark}

 \begin{proof}[Stack-theoretic version of Demailly's proof of Theorem \ref{thm:demailly}]
Let $\mathcal{L}$ be a relatively ample line bundle on $X$ over $S$.   
Let $\mathcal{M}_g$ be the stack of smooth projective curves of genus $g$ over $S$, and let $\mathcal{U}_g \to \mathcal{M}_g$ be the universal curve. The Hom-stack $\underline{\mathrm{Hom}}_{\mathcal{M}_g}(\mathcal{U}_g, X \times \mathcal{M}_g)$ is the countable union over $d \geq 0$ of the finitely presented stacks $\mathcal{H}_{g,d} := \underline{\mathrm{Hom}}_{\mathcal{M}_g}^{d}(\mathcal{U}_g, X \times \mathcal{M}_g)$ of morphisms $f$ which have (fibrewise) degree $d$ with respect to $\mathcal{L}$.  Let $S_{g,d}$ be the  image of $\mathcal{H}_{g,d}$ in $S$ under the structure map $\mathcal{H}_{g,d} \to S$. As the structure map  $\mathcal{H}_{g,d} \to S$ is quasi-compact, locally of finite presentation and $S$ is quasi-compact and quasi-separated, the set $S_{g,d}$ is a constructible subset of $S$; see \cite[Tag~054J]{stacks-project} or \cite[Th\'eor\`eme~5.4.9]{LMBbook}.

For any $\beta \in \ZZ_{>0}$, consider $$S_{\beta} = \bigcup\limits_{\substack{(g,d) \in \ZZ_{\geq 0} \times \ZZ \\ d > \beta \cdot g}} {S_{g,d}};$$ this is the locus of $s$ in $S$ for which $X_{\overline{s}}$ is not algebraically hyperbolic  over $\overline{k(s)}$ with constant $\beta$. Note that this is a countable union, over $j = 1,2,3, \ldots$, of locally closed subsets $U_i \cap V_i$ with $U_i \subset S$ open and $V_i \subset S$ closed. Without loss of generality, we will assume that $V_i$ is irreducible and $\overline{U_i \cap V_i} = V_i$. In particular, for each $i$, the generic point $\eta_i$ of $V_i$ lies in $S_\beta$. Now, we claim that $S_\beta$ is stable under specialisation.

 To prove the claim consider two points $s,t \in S$ such that $s$ specialises to $t$ and $s \in S_\beta$. If $X_t$ is not pure (Remark \ref{remark:pure}), then $t$ is clearly contained in $S_\beta$, so assume $X_t$ is pure. As $s \in S_{\beta}$, there is a morphism from a smooth curve $C/k(s)$ of certain genus $g$ to $X_s$ of degree greater than $\beta \cdot g$. Now, to prove that we also get a morphism of the the same degree from a curve $D/k(t)$ of the same genus to $X_t$ we proceed as follows. First, we may and do assume that $S$ is affine. Then, we choose a chain of specializations $$s = p_0 \rightarrow p_1 \rightarrow \ldots \rightarrow p_\ell = t$$ with each specialization of codimension one. Localizing at $p_i$ and modding out $p_{i+1}$ we obtain a discrete valuation ring to which we can apply Lemma \ref{lem:extension1}. In this way, by repeated application of Lemma \ref{lem:extension1}, we   get a morphism of the same degree from a curve $D/k(t)$ of the same genus to $X_t$.  
Hence, $t \in S_{\beta}$. 
 
 Since $S_{\beta}$ is a countable union of locally closed subsets and stable under specialisation,   it is a Zariski-countable closed subset.
It now follows that the non-algebraically hyperbolic locus $$S^{\textrm{nh}} = \bigcap_{\beta =1}^\infty S_{\beta}$$ is also Zariski-countable closed, as the Zariski-countable topology is a topology on $S$ by Lemma \ref{lem:countabletopology}. \qedhere
 \end{proof}

\begin{remark} Let us briefly  assume that $S$ is a positive-dimensional integral  variety over  an algebraically closed field $k$ of characteristic zero. Let $K$ be the function field of $S$, and suppose that $X_{\overline{K}}$ is  algebraically hyperbolic over $\overline{K}$.  If $k$ is \emph{uncountable}, then it follows from Theorem \ref{thm:demailly}  that there is an $s$ in $S(k)$ such that $X_s$ is  algebraically hyperbolic. If we assume Lang's conjecture (or, more precisely, the equivalence of (2) and (6) in \cite[Conjecture~12.1]{JBook}) , then the hypothesis on the cardinality of $k$ is unnecessary. Indeed,  to explain this, we assume for simplicity that $k\subset \CC$. Since $\CC$ is uncountable, it follows from  Theorem \ref{thm:demailly} that  there is an $s$ in $S(\CC)$ such that $X_s$ is algebraically hyperbolic over $\CC$. In particular, by Lang's conjecture (which we are assuming to hold for now), the complex analytic space $X_s^{\an}$ is Kobayashi hyperbolic. Now,  as the fibres of $X_{\CC}^{\an}\to S_{\CC}^{\an}$ are compact, it follows from a theorem of Brody   that  there is an analytic open neighbourhood $U\subset S^{\an}$ such that, for every $u$ in $U$, the fibre $X_u$ is Kobayashi hyperbolic  (see \cite[Theorem~3.11.1]{KobayashiBook}). Now, since  $S(k)\subset S(\CC)$ is  a dense subset  of $S_{\CC}^{\an}$ with respect to the complex analytic topology, there is a point $s$ in $U\cap S(k)$. We see that $X_s$ is algebraically hyperbolic over $k$. 
\end{remark}

\begin{remark}
In the analytic setting, we cannot hope for the locus of points in the base with hyperbolic fibre to be Zariski open. For example, if we have a relative smooth proper curve $\mathcal{C} \to \CC^*$ with precisely one non-hyperbolic fibre, then we can pull-back this family along the exponential map $\CC \to \CC^*$ to obtain a family $\mathcal{X} \to \CC$ such that the set of $s$ in $\CC$ with $\mathcal{X}_s$ non-hyperbolic is a countably infinite subset of $\CC$.
\end{remark}

The notion of pseudo-algebraic hyperbolicity should   be Zariski open in families in light of Vojta's conjecture  (Conjecture \ref{conj:vojta}). However, as we currently do not know whether pseudo-algebraic hyperbolicity is stable under generisation, we also do not know whether the locus of points in the base for which the fibre is pseudo-algebraically hyperbolic is in fact Zariski-countable open.
Nonetheless,  our next result shows that the locus of pseudo-algebraically hyperbolic fibres contains a Zariski-countable open. As the reader will notice, the proof of this result is similar to the proof of  Theorem \ref{thm:demailly}.

\begin{proposition}\label{prop:specialising_hyperbolicity} Assume $S$ is integral with function field $K=K(S)$, and let $K\to \overline{K}$ be an algebraic closure. Let $\Delta \subset X $ be a closed subscheme such that  $X_{\overline{K}}$ is algebraically hyperbolic modulo $\Delta_{\overline{K}}$ over $\overline{K}$. Then, for every algebraically closed field $k$ of characteristic zero and a very general $s$ in $S(k)$, the projective scheme $X_s$ is algebraically hyperbolic modulo $\Delta_s$ over $k$.  
\end{proposition}
\begin{proof} We may and do assume that $k$ is uncountable.
Let $\mathcal{L}$ be a relatively ample line bundle on $X$ over $S$.  Let $$\beta = \alpha_{X_{\overline{K}}, \Delta_{\overline{K}}, \mathcal{L}_{\overline{K}}}$$
be the constant as in the definition of algebraic hyperbolicity for $X_{\overline{K}}$ modulo $\Delta_{\overline{K}}$.

Let $\mathcal{M}_g$ be the stack of smooth proper curves of genus $g$ over $S$, and let $\mathcal{U}_g \to \mathcal{M}_g$ be the universal curve. The Hom-stack $\underline{\mathrm{Hom}}_{\mathcal{M}_g}(\mathcal{U}_g, X \times \mathcal{M}_g)\setminus \underline{\mathrm{Hom}}_{\mathcal{M}_g}(\mathcal{U}_g, \Delta \times \mathcal{M}_g)$ is the countable union over $d \geq 0$ of the finitely presented stacks $\mathcal{H}_{g,d} := \underline{\mathrm{Hom}}_{\mathcal{M}_g}^{d}(\mathcal{U}_g, X \times \mathcal{M}_g)\setminus \underline{\mathrm{Hom}}_{\mathcal{M}_g}(\mathcal{U}_g, \Delta \times \mathcal{M}_g)$ of morphisms $f$ which have (fibrewise) degree $d$ with respect to $\mathcal{L}$.  Let $S_{g,d}$ be the   image of $\mathcal{H}_{g,d}$ in $S$ under the structure map $\mathcal{H}_{g,d} \to S$. Then $S_{g,d}$ is a constructible subset of $S$.

Hence, for any $d > \beta \cdot g$, the closure $\overline{S_{g,d}}$ of $S_{g,d}$ inside $S$ does not equal $S$, as $X_{\overline{K}}$ is algebraically hyperbolic modulo $\Delta_{\overline{K}}$. As $k$ is uncountable, we have $$\bigcup\limits_{\substack{(g,d) \in \ZZ_{\geq 0} \times \ZZ \\ d > \beta \cdot g}} \overline{S_{g,d}} \neq S.$$

Let  $s$ be a point in  $S(k)$  such that, for every $g\geq 0$ and $d>\beta \cdot g$, the point $s$ is not in $S_{g,d}$. Then, the projective scheme $X_s$ is algebraically hyperbolic modulo $\Delta_s$ over $k$. This concludes the proof.
\end{proof}

\begin{proof}[Proof of Theorem \ref{thm:specialisation2}]
 This follows from Proposition \ref{prop:specialising_hyperbolicity}.
\end{proof}

\newcommand{\stacks}[1]{\cite[\href{https://stacks.math.columbia.edu/tag/#1}{#1}]{stacks-project}}

Next, we will prove that the 1-bounded locus is Zariski-countable open. For this, we first need two intermediate results.

\begin{lemma}\label{lem:purityiscountableopen}
The set $S^0$ consisting of these $s$ in $S$ such that $X_{\overline{s}}$ contains a rational curve is Zariski-countable closed in $S$.
\end{lemma}
\begin{proof}
This is well-known; see for instance the arguments proving \cite[Lemma~3.7]{Debarrebook1}. (Namely, since a rational curve can not specialize to a higher genus curve, this locus is closed under specialization.  One then combines this with the fact that  the scheme $\underline{\Hom}_S(\mathbb{P}^1_S,X)$ is a countable union of quasi-projective schemes over $S$.)
\end{proof}

\begin{proposition}\label{prop:1b_is_zco}
The subset of $s$ in $S$ such that $X_s$ is $1$-bounded, is Zariski-countable open.  
\end{proposition}

\begin{proof}\newcommand{\stcurves}{\overline{\mathcal{M}_{g}}}

Let $\mathcal{L}$ be a relatively ample line bundle on $X$. For $g>1$, let $\overline{\mathcal{M}_g}$ be the stack of stable curves of genus $g$ over $S$.
Note that $\stcurves$ is noetherian (as $S$ is noetherian), and that the natural morphism   $\kappa \colon \stcurves \to S$ is proper.

Let $\univcurvebar{g} \to \stcurves$ be the universal curve over $\stcurves$, and let $\mathcal{H}_d^g = \underline{\Hom}_{\stcurves}^d(\univcurvebar{g}, X \times \stcurves)$ and $\rho \colon \mathcal{H}_d^g \to \stcurves$ be the structure morphism. Let $S^0$ be as in Lemma \ref{lem:purityiscountableopen}. Let $T_d^g = \rho(\mathcal{H}_d^g) \cup \kappa^{-1}(S^0) $. A priori, this is a countable union of constructible subsets in $\overline{\mathcal{M}_g}$. Because of Corollary \ref{cor:homisproper}, $\rho$ satisfies the existence part of the valuative criterion of properness outside of $S^0$, i.e.,  for every $s \in S$ not in $S^0$, the base-change of $\rho$ along $\overline{\mathcal{M}_g}\times_S \mathcal{O}_{S,s}\to \overline{\mathcal{M}_g}$ satisfies the existence part of the valuative criterion of properness.

This shows that $T_d^g$ is closed under specialisation.   Since $T_d^g$ is a countable union of constructible subsets and is closed under specialization, we see that it is  Zariski-countable closed.

For each $n \in \ZZ_{>0}$ consider $V_n^g = \bigcup_{d=n}^\infty T_d^g$; this is the locus of stable curves admitting a morphism of degree at least $n$ to $X$, together with $\kappa^{-1}(S^0)$. Now we are interested in $V^g = \bigcap_{n=1}^\infty V_n^g$, which is the locus of curves admitting a morphism of arbitrary high degree, together with $\kappa^{-1}(S^0)$. Since $S$ is an integral noetherian scheme, the subset $V^g = \bigcap_{n=1}^\infty V_n^g$  is Zariski-countable closed (Lemma \ref{lem:countabletopology}).  Hence, the image $S^g = \kappa(V^g)$ in $S$ is Zariski-countable closed. For $S^1$, the locus where there are morphisms of arbitrary high degree from a genus 1 curve to $X$, we can use $\mathcal{M}_{1,1}$ instead to prove that this is Zariski-countable closed.

In particular $\bigcup_{g\geq 0} S^g$ is Zariski-countable closed. This concludes the proof, as the locus where $X_s$ is not 1-bounded equals $\bigcup_{g \geq 0} S^g$ by the fact that $1$-boundedness can also be tested on stable curves (Theorem \ref{thm:test_on_reds}).
\end{proof}

\begin{corollary}
The subset of $s$ in $S$ such that $X_s$ is bounded is Zariski-countable open.  
\end{corollary}
\begin{proof}
 Since a projective scheme over $k$ is bounded if and only if it is $1$-bounded over $k$ (see \cite[Theorem~9.3]{JKa}), the corollary follows from Proposition  \ref{prop:1b_is_zco}.
\end{proof}

\begin{proposition}\label{prop:11b_is_zco}
The subset of $s$ in $S$ such that $X_s$ is $(1,1)$-bounded, is Zariski-countable open.
\end{proposition}

\begin{proof}\newcommand{\stcurves}{\overline{\mathcal{M}_{g}}}
Let $g$, $S^0$, $\stcurves$, $\univcurvebar{g}$ and $\mathcal{H}_d^g$ be as in the proof of Proposition \ref{prop:1b_is_zco}. Now we consider the morphism
$$\tau \colon \space \univcurvebar{g} \times_{\stcurves} \mathcal{H}_d^g \to \univcurvebar{g} \times_S X$$ $$( (C, c), \varphi) \mapsto ((C,c), \varphi(c)).$$
The fibre $\tau^{-1}\{((C,c), x)\}$ is exactly $\Hom^d( [C,c], [X,x])$, the morphisms of degree $d$ mapping $c$ to $x$. Again we let $\kappa \colon \univcurvebar{g} \times_S X \to S$ be the structure map, which is still proper. We proceed in the same way as in the proof of Proposition \ref{prop:1b_is_zco}, taking $T_d^g = \mathrm{im}(\tau) \cup \kappa^{-1}(S^0)$, which is then proved to be a Zariski-countable closed set. The properness over $S$ of all schemes involved causes the morphism $\tau$ to also satisfy the existence part of the valuative criterion for properness as in Corollary \ref{cor:homisproper}.

Then $V^g = \bigcap_{n=1}^\infty \bigcup_{d=n}^\infty T_d^g$ is again Zariski-countable closed, as $\univcurvebar{g} \times_S X$ is still noetherian. Hence, $S^g = \kappa(V^g)$ is Zariski-countable closed in $S$. As also done in Proposition \ref{prop:1b_is_zco}, $S^1$ can be defined and shown to be Zariski-countable closed in $S$ using $\mathcal{M}_{1,1}$.
 Hence, $\bigcup_{g=0}^\infty S^g$ is Zariski-countable closed in $S$. The latter is exactly the locus of $s$ in $S$ such that $X_s$ is not $(1,1)$-bounded by Proposition \ref{prop:mnbounded}.
\end{proof}

\begin{corollary}
For $m,n > 0$, the subset of $s$ in $S$ such that $X_s$ is $(n,m)$-bounded, is Zariski-countable open.
\end{corollary}
\begin{proof}
Since $m\geq 1$, a projective scheme over $k$ is $(n,m)$-bounded if and only if it is $(1,1)$-bounded over $k$ (see \cite[\S8]{JKa}). Therefore, the corollary follows from Proposition \ref{prop:11b_is_zco}.
\end{proof}

 \subsection{Varieties of general type}\label{section:nakayama} Let $k$ be an algebraically closed field of characteristic zero.
 Let $\mathcal{M}^\textrm{pol}$ be the stack over $k$ whose objects over a $k$-scheme $S$ are pairs $(f:X\to S, \mathcal{L})$ with $f:X\to S$ a flat proper finitely presented morphism and $\mathcal{L}$ an $f$-relative ample line bundle on $X$. Note that $\mathcal{M}^\textrm{pol}$ is a locally finitely presented algebraic stack over $k$ with affine diagonal \cite[Section~2.1]{JLFano}. This follows from \cite[Tag~0D4X]{stacks-project} and \cite{dJS100}. The additional datum of a polarisation (i.e., the  $f$-relative ample line bundle) is necessary to ensure  the algebraicity of the stack $\mathcal{M}^\textrm{pol}$.
 
 Recall that a proper scheme $X$ over an algebraically closed field $K$ is of general type (over $K$) if, for every irreducible   component $X'$  of $X$, there is a resolution of singularities $\tilde{X}\to X'_\textrm{red}$   such that $\omega_{\tilde{X}}$ is a big line bundle.  In other words, a proper scheme over $K$ is of general type if \emph{every} irreducible component of $X$ is of general type. We refer the reader to \cite{Lazzie1, Lazzie2} for basic properties of varieties of general type.
 
 Let $\mathcal{M}^\textrm{gt}$ be the substack of $\mathcal{M}^\textrm{pol}$ whose objects over a $k$-scheme $S$ are pairs $(f:X\to S, \mathcal{L})$ in $\mathcal{M}^\textrm{pol}$ such that the geometric fibres of $f:X\to S$ are proper schemes of general type.
 
 \begin{theorem}[Nakayama]\label{thm:n}
 The substack $\mathcal{M}^\mathrm{gt}$ of  $\mathcal{M}^\mathrm{pol}$  is an open substack.
 \end{theorem}
 \begin{proof} This is a consequence of Nakayama's theorem (Theorem \ref{thm:nakayama}). 
 \end{proof}

For every polynomial $h$ in $\mathbb{Q}[t]$, we let $\mathcal{M}^\textrm{pol}_h\subset \mathcal{M}^{\textrm{pol}}$ be the substack of pairs $(f:X\to S,\mathcal{L})$ such that, for every geometric point $s$ of $S$, the Hilbert polynomial of the pair $(X_s,\mathcal{L}_s)$ over the algebraically closed field $k(s)$ equals $h$. Analogously, we define $\mathcal{M}_h^\textrm{gt} = \mathcal{M}_h^\textrm{pol}\times_{\mathcal{M}^{\textrm{pol}}} \mathcal{M}^\textrm{gt}$.
The following proposition is a well-known consequence of the theory of Hilbert schemes.

 \begin{proposition}\label{prop:123} The stack $\mathcal{M}^\mathrm{pol}$ is a countable disjoint union of the finitely presented open and closed substacks $\mathcal{M}_h^\mathrm{pol}$, and its open substack $\mathcal{M}^\mathrm{gt}$ is a countable disjoint union of the finitely presented open and closed substacks $\mathcal{M}_h^\mathrm{gt}$.
 \end{proposition}
 \begin{proof}
The first statement follows from  \cite[Tag~0D4X]{stacks-project}, and the second statement follows from the first by Theorem \ref{thm:n}.
 \end{proof}

 \begin{proof}[Proof of Theorem \ref{thm:specialisation}]
 Let $S$ be a noetherian scheme over $\QQ$, and let $f:X\to S$ be a projective morphism.  To show that the set of $s$ in $S$ such that every subvariety of $X_s$ is of general type is Zariski-countable open, we fix an $f$-relative ample line bundle $\mathcal{L}$ on $X$. 
 Let $\mathrm{Hilb}_{X/S}\to S$ be the Hilbert scheme of the projective morphism $X\to S$. Consider the forgetful morphism $\mathrm{Hilb}^h_{X/S}\to \mathcal{M}^\textrm{pol}_h$ which associates to a closed $S$-flat subscheme $Z\subset X$ with Hilbert polynomial $h$ (with respect to $\mathcal{L}$) the corresponding object $(Z\to S, \mathcal{L}|_Z)$  of $\mathcal{M}_h^\textrm{pol}$. Let $\mathrm{Hilb}^{\textrm{gt},h}_{X/S}$ be the inverse image of the open substack $\mathcal{M}_h^\textrm{gt}$ in $\mathcal{M}_h^\textrm{pol}$ (Theorem \ref{thm:n}) under the forgetful morphism.
 
For $h$ in $\QQ[t]$, we let 
$
\mathrm{Hilb}_{X/S}^{\textrm{n-gt},h} = \mathrm{Hilb}_{X/S}^h\setminus \mathrm{Hilb}_{X/S}^{\textrm{gt}, h}.
$ Note that $\mathrm{Hilb}_{X/S}^{\textrm{n-gt},h}$ is a closed subscheme of the quasi-projective $S$-scheme $\mathrm{Hilb}_{X/S}^{h}$.
 For $h$ in $\QQ[t]$, let $S_h\subset S$ be the image of $\mathrm{Hilb}_{X/S}^{\textrm{n-gt},h}\to S$. Since $\mathrm{Hilb}_{X/S}^h\to S$ is a proper morphism \cite[Theorem~5.1]{Nitsure} and $\mathrm{Hilb}_{X/S}^{\textrm{n-gt},h}$ is closed in $\mathrm{Hilb}_{X/S}$,  we see that $S_h$ is a closed subset of  $S$. 
 Note that the locus of $s$ in $S$ such that $X_s$ has an integral subvariety which is not of general type is   
 \[
 \bigcup_{h\in \QQ[t]} S_h.
 \]  We conclude that it is a countable union of closed subschemes of $S$, as required.
 \end{proof}

The fact that the locus of algebraic hyperbolicity is Zariski-countable open implies that this locus is stable under generisation. This allows us to easily prove Theorem \ref{thm:generisation}.

\begin{proof}[Proof of Theorem \ref{thm:generisation}]   
Note  that $(ii)$ is proven in \cite{JVez}. To prove $(iii)$, note that the set  $S^{\mathrm{ah}}$ of $s$ in $S$ with $X_s$ algebraically hyperbolic is Zariski-countable open in $S$ by  Theorem \ref{thm:demailly}. By assumption, $S^{\mathrm{ah}}$ is non-empty, so that  the generic point of $S$ lies in $S^{\mathrm{ah}}$.  This proves $(iii)$. To prove $(iv)$ and $(v)$, we argue in a similar manner employing   Theorems   \ref{thm:demailly-b} and \ref{thm:demailly-mn} instead of Theorem \ref{thm:demailly}, respectively. Similarly, to prove $(i)$, we argue in a similar manner employing instead Nakayama's theorem (or rather its consequence  Theorem \ref{thm:specialisation}).
 \end{proof}

\subsection{Mildly bounded varieties in families}

In this section we  prove that the set of $s$ in $S$ such that $X_s$ is mildly bounded  (Definition \ref{defn:mild_bounded})   is Zariski-countable open (Theorem \ref{thm:mildly-bounded-is-zco}). To do this, we use the following lemma.
\begin{lemma}\label{lem:wholefibre}
Let $f \colon T \to S$ be a finite type flat morphism of noetherian   schemes. If   $Y \subset T$ is a Zariski closed (resp.\ Zariski-countable closed) subset of $T$, then the locus of $s$ in $S$ such that $Y_s = T_s$ is a Zariski closed (resp.\ Zariski-countable closed) subset of $S$.
\end{lemma}
\begin{proof}
First consider the case $Y \subset T$ is Zariski closed.  Since $f$ is a finite type flat morphism of noetherian schemes, it follows that $f$ is an open map \cite[Tag~01UA]{stacks-project}.  In particular, if $U:=T\setminus Y$, then $f(U)$ is open in $S$. This is exactly the complement of points $s$ in $S$ such that $Y_s = T_s$.

Now suppose $Y = \bigcup_{i=1}^\infty Y_i$ is a Zariski-countable closed and $Y_i \subset T$ is closed. Without loss of generality we may and do assume that $Y_1 \subset Y_2 \subset \cdots$. Then, as $T$ is noetherian, we have that $Y_s = T_s$ if and only if there exists an $i$ such that $Y_{i,s} = T_s$. Let $S_i \subset S$ be the (closed) locus of $s$ in $S$ satisfying $Y_{i,s} = T_s$. Then, the locus of $s$ in $S$ for which $Y_s = T_s$, equals $\bigcup_{i=1}^\infty S_i$. Since $\bigcup_{i=1}^\infty S_i$ is Zariski-countable closed (by definition), this concludes the proof.
\end{proof}

To  prove Theorem \ref{thm:mildly-bounded-is-zco} we will   use the characterisation of mildly bounded varieties stated in Lemma \ref{lem:mildly-bounded-for-stable}. This equivalent definition of mild boundedness allows us to use moduli-theoretic arguments similar to those employed in the proof of Theorem \ref{thm:demailly}, \ref{thm:demailly-b}, and \ref{thm:demailly-mn}, respectively.

\begin{proof}[Proof of Theorem \ref{thm:mildly-bounded-is-zco}]\newcommand{\stcurves}[1]{\overline{\mathcal{M}_{g,#1}}}
Let $g$, $\stcurves{n}$, and $\mathcal{H}_{g,n}^d$ be as in Corollary \ref{cor:homisproper}. Without loss of generality we will assume that every geometric fibre of $X \to S$ is pure (Remark \ref{remark:pure}). The proof below can be adjusted for the non-pure case by adjoining the inverse image of the non-pure locus of $S$ to each of the Zariski(-countable) closed sets appearing in this proof, in the same way as we did in the proof of Proposition \ref{prop:1b_is_zco}.

In the notion of mild boundedness, there are possibly non-projective (smooth irreducible) curves $C$ appearing. We will consider each such smooth irreducible curve $C$ as an open subset of its smooth projective closure $\overline{C}$. Then $\overline{C} \setminus C$ consists of finitely many points, say $m$ points, and in this way we can consider $C$ as a (possibly non-unique) point of $\stcurves{m}$.

Consider the evaluation morphism $$\tau \colon \space \mathcal{H}_{g,n+m}^d \to X^n \times \stcurves{n+m}$$ $$\left(\left[C, (c_1, \ldots, c_n, d_1, \ldots, d_m)\right], \varphi \colon C \to X\right) \mapsto \left( (\varphi(c_1), \ldots, \varphi(c_n)), \left[C, (c_1, \ldots, c_n, d_1, \ldots, d_m)\right]\right)$$
As $X$ is proper over $S$, the properness of $\tau$ follows immediately from Corollary \ref{cor:homisproper}. Hence, the image $A_{g,n,m}^d$ of $\tau$ in $X^n \times \stcurves{n+m}$ is closed. This is the locus of $n$-tuples of possibly non-distinct points in $X$ and $n$-pointed (non-projective) curves $C$ admitting an $n$-pointed morphism of degree $d$ to $X$. Let $B_{g,n,m} = \bigcap_{d'=1}^\infty \bigcup_{d=d'}^\infty A_{g,n,m}^d$ be the locus of $n$-tuples of points in $X$ and $n$-pointed curves $C$ admitting $n$-pointed morphisms of unbounded degree to $X$. This locus is Zariski-countable closed by Lemma \ref{lem:countabletopology}.

Now let $C_{g,n,m}$ be the projection of $B_{g,n,m}$ under the projection $X^n \times \stcurves{n+m} \to \stcurves{n+m}$, which is Zariski-countable closed in $\stcurves{n+m}$ as $X$ is proper over $S$.

 Let $f \colon \stcurves{n+m} \to \stcurves{m}$ be the forgetful map and $D_{g,n,m} \subset \stcurves{m}$ the locus of $s \in \stcurves{m}$ such that $f^{-1}(s) = (C_{g,n,m})_s$. Then, it follows from Lemma \ref{lem:wholefibre} that $D_{g,n,m}$ is Zariski-countable closed in $\stcurves{m}$. Similarly, we have that $E_{g,m} := \bigcap_{n=1}^\infty D_{g,n,m}$ is Zariski-countable closed by Lemma \ref{lem:countabletopology}. The latter is the locus of (non-projective, non-irreducible) curves obtained by removing $m$ points from a projective curve of genus $g$ for which the mild boundedness condition fails.

Let $\rho_{g,m} \colon \stcurves{m} \to S$ be the proper structure morphism. As mild boundedness can be tested on curves which are not irreducible (Lemma \ref{lem:mildly-bounded-for-stable}), the locus where $X \to S$ is not mildly bounded is $\bigcup_{g=1}^\infty\bigcup_{m=0}^\infty \rho_{g,m}(E_{g,m})$. This is now Zariski-countable closed.
\end{proof}

\begin{proof}[Proof of Corollary \ref{cor:mildly-bounded-gens}]
This follows from Theorem \ref{thm:mildly-bounded-is-zco} and the fact that a Zariski-countable closed set of (the scheme) $S$ contains the generic point of $S$ if and only if this set equals $S$ itself.
\end{proof}

\begin{proof}[Proof of Corollary \ref{cor:mb_persists}] We may and do assume that $L$ is algebraically closed. Let $K $ be a finitely generated subfield of $L$ containing $k$ with $\overline{K}=L$ and let $S$ be an integral variety over $k$ with function field equal to $K$ and let $s\in S(k)$ be a $k$-point of $S$.  Consider $\mathcal{X} = X\times S \to S$, and note that $\mathcal{X}_s = X$ is mildly bounded over $k$. Since the locus of mildly bounded varieties is stable under generisation 
  (Corollary \ref{cor:mildly-bounded-gens}), it follows that $X_{K}$ (hence $X_L$) is mildly bounded.
\end{proof}

 \begin{corollary} Let $k$ be an \textbf{uncountable} algebraically closed field of characteristic zero, and let $k\subset L$ be an algebraically closed field of finite transcendence degree over $k$. If $X$ is a projective variety over $k$, then $X$ is mildly bounded over $k$ if and only if $X_L$ is mildly bounded over $L$. 
 \end{corollary}
 \begin{proof}
 If $X$ is mildly bounded over $k$, then $X_L$ is mildly bounded over $L$ by Corollary \ref{cor:mb_persists}. Now, assume $X_L$ is mildly bounded over $L$. Let $S$ be an integral variety over $k$ whose dimension equals the transcendence degree of $L$ over $k$ and whose  function field $K(S) $ is contained in $L$.  Consider $\mathcal{X}:= X\times S$ as a projective scheme over $S$.  Then, as the set of $s$ in $S$ such  that $X_s$ is mildly bounded is Zariski-countable open (Theorem \ref{thm:mildly-bounded-is-zco}) and $k$ is uncountable, it follows that there is an $s$ in $S(k)$ such that the variety $X=\mathcal{X}_s$ is mildly bounded.   
 \end{proof}
 
\section{Mild boundedness}\label{section:mb2}
\newcommand{\Homnc}{\Hom^{\textrm{nc}}}
Let $k$ be an algebraically closed field of characteristic zero. In this section we study mildly bounded varieties (Definition \ref{defn:mild_bounded}). 
We start by showing that $\mathbb{A}^1_k$ (hence $\mathbb{P}^1_k$) is not mildly bounded over $k$.

\begin{proposition}\label{prop:a1}
The curve $\mathbb{A}^1_k$ over $k$ is not mildly bounded over $k$. 
\end{proposition}
\begin{proof}
Indeed, if $c_1, \ldots, c_n \in \mathbb{A}^1_k(k)$ are distinct points and $x_1, \ldots, x_n \in \mathbb{A}^1_k(k)$ are arbitrary, then there exist morphisms $\varphi \colon \mathbb{A}^1_k \to \mathbb{A}^1_k$ of arbitrary high degree such that $\varphi(c_i) = x_i$ by using Lagrange interpolation. 
\end{proof}

\begin{proposition}\label{prop:mb} Assume $k$ is uncountable.  Let $X$ be a projective scheme over $k$, and let $\mathcal{L} $ be an ample line bundle on $X$. If $X$ is \emph{not} mildly bounded over $k$, then there is a smooth projective irreducible curve $C$ over $k$ such that, for every $d\geq 1$, the moduli scheme $\underline{\Hom}_k^{\geq d}(C,X)$ of morphisms $f:C\to X$  with $\deg_C f^\ast \mathcal{L}\geq d$ is positive-dimensional.
\end{proposition}
\begin{proof} If $X$ has a rational curve, then we can take $C=\mathbb{P}^1_k$.  (We do not use   that $k$ is uncountable here.) Thus, we may and do assume that $X$ has no rational curves.  Let $C$ be a smooth projective irreducible curve over $k$ which does not satisfy the mild boundedness condition for $X$. Note that, as $X$ has no rational curves, for every integer $d\geq 1$, the scheme $\underline{\Hom}_k^{\leq d}(C,X)$ is a proper scheme over $k$ and that, for every $c$ in $C$ and $x$ in $X$, the set $\Hom_k^{\leq d}([C,c],[X,x])$ is finite; see \cite[\S3]{JKa}. Therefore, for every $d\geq 1$, the set $\Hom_k^{\geq d}(C,X)$ must have infinitely many $k$-points (otherwise $C$ would satisfy the mild boundedness condition). Suppose that this scheme is zero-dimensional. Then its $k$-points form  a countable infinite set. Let $f_1,\ldots$ be the elements of $\Hom_k^{\geq d}(C,X)$.  For every $i \neq j$, the locus $C_{ij} \subset C(k)$ where $f_i$ and $f_j$ agree consists of finitely many points, as $X$ is separated over $k$. Since $C(k)$ is uncountable, there exists a point $c \in C(k)$ such that $c \notin C_{ij}$ for all $i \neq j$. As $X$ over $k$ does not satisfy the mild boundedness condition with respect to $C$,   there must be an $x \in X(k)$ such that
$$H = \Hom_k^{\geq d}([C, c], [X, x])$$
is   infinite.   However, by the choice of $c$, the set  $H$ can only contain at most one element of $\{f_1, f_2, \ldots \}$. Hence, we obtain a contradiction and we conclude that $\Hom_k^{\geq d}(C,X)$ must be uncountable, so that the scheme $\underline{\Hom}_k^{\geq d}(C,X)$ has (a component of) dimension at least one.  
\end{proof}

\begin{remark}
Let $k$ be an uncountable algebraically closed field and  suppose that $X$ is a variety over $k$ which is \emph{not} mildly bounded. Then, the argument used in the proof of Proposition \ref{prop:mb} shows that there is a       smooth   quasi-projective irreducible curve $C$ over $k$ such that the set of non-constant morphisms $C \to X$  is uncountable. 
\end{remark}

\begin{lemma}\label{lem:four-statements}
Let $X$ be a projective variety over $k$. Consider the following statements.
\begin{enumerate}
\item The variety $X$ is mildly bounded over $k$.
\item The variety $X$ has no rational curves.
\item For every curve $C$, there is an integer $d\geq 1$ such that  $\Hom^{\geq d}(C,X)$ is zero-dimensional, i.e., $\Hom(C,X)$ has only finitely many positive-dimensional components.
\item $X$ is groupless.
\end{enumerate}  Then we have $(1)\implies (2)$, $(4)\implies (2)$, and $(3) \implies (4)$. If $k$ is uncountable, then we also have $(3)\implies (1)$.
\end{lemma}

\begin{proof}
The implication $(1) \implies (2)$ follows from Proposition \ref{prop:a1}. The implication  $(4) \implies (2)$ is an  immediate consequence of the definitions. 
For the implication $(3) \implies (4)$, assume that $X$ is not groupless. Then there  is a non-zero abelian variety $A$ and a non-constant morphism $\phi\colon A\to X$.    Let $\iota \colon C \hookrightarrow A$ be a curve which is not contracted by $\phi$. Then $\iota$ can be composed with any endomorphism of $A$. In particular, it can be composed with multiplication by $n \in \ZZ_{>0}$   and any translation with a point of $A(k)$. This gives infinitely many components of (strictly) positive dimension in $\Hom(C,X)$. This shows that $(3)\implies (4)$. Finally, to conclude the proof, we may assume that $k$ is uncountable. Now, the implication $(3) \implies (1)$ follows from Proposition \ref{prop:mb}.
\end{proof}

Note that Conjecture \ref{conj:new} predicts the equivalence of $(1)$ and $(2)$. In fact,  we conjecture something stronger. 

\begin{conjecture}
In the situation of Lemma \ref{lem:four-statements}, $(1)\iff (2)$ and $(3) \iff (4)$.
\end{conjecture}

\begin{example}\label{ex:genus-greater-than-2}
Let $X$   be a smooth projective curve of genus at least two over $k$. Then we claim that $X$ is mildly bounded. Indeed, if $C \to X$ is a morphism from a smooth irreducible quasi-projective curve to $X$, then it extends to a morphism from its smooth projectivisation $\overline{C}$ to $X$, and the set of non-constant morphisms $\overline{C}\to X$ is finite by    De Franchis's theorem.  In particular, we see that $X$ is mildly bounded over $k$. Below, we will give a simpler (more direct) proof of the mild boundedness of any smooth quasi-projective irreducible curve $X$ which is neither isomorphic to $\mathbb{A}_k^1$ nor $\mathbb{P}^1_k$.
\end{example}
 
To prove that abelian varieties are mildly bounded we will use that most smooth curves inject into  their generalised Jacobian (also referred to as the semi-Albanese variety).   This will be explained in more detail in the rest of this section. Definitions and constructions of this (semi-)Albanese variety can be found in \cite{Mochizuki1, SerreAlgGroups}. If $X$ is a geometrically  integral variety over $k$, we let $\mathrm{Alb}(X)$ be its semi-Albanese variety over $k$; this exists by \cite[Corollary~A.11.(i)]{Mochizuki1}. 

Given a point $x$ in $X(k)$, we will refer to the universal morphism $X\to \mathrm{Alb}(X)$ associated to $(X,x)$ as an Abel-Jacobi morphism. We will usually suppress the choice of a base point in $X$.

  \begin{lemma}\label{lem:345}
 Let $\overline{C}$ be a smooth projective irreducible curve of genus $g$ over $k$ and let $C\subset \overline{C}$ be a dense open such that $\#(\overline{C} \setminus C) = r > 0$. Fix an Abel-Jacobi map $\iota: C \to \mathrm{Alb}(C)$ to the Albanese variety of $C$ (associated to the choice of a base point in $C(k)$).   For any integer $N\geq 0$, let $\rho_N \colon C^{N} \to \mathrm{Alb}(C)$ be given by $(c_1,\ldots,c_{N})\mapsto \iota(c_1)+\ldots+\iota(c_{N})$. Then $\rho_{2g+2r-2}$ is surjective.  
 \end{lemma}
\begin{proof}
In \cite[Proposition~A.3]{Mochizuki1}, it is already claimed that there is an $N$ such that $\rho_N$ is surjective. As we are considering curves, we know more about the structure of $\mathrm{Alb}(C)$.
Let $D$ be the divisor of $\overline{C}$ obtained by taking each of the points of $\overline{C} \setminus C$ once. Then there is a generalised Jacobian in the sense of \cite[Chap.\ V]{SerreAlgGroups} associated to $D$. Using \cite[Lemme 6]{SerreAlbanese}, we see that this generalised Jacobian is isomorphic to $\mathrm{Alb}(C)$. 

Hence, $\mathrm{Alb}(C)$ is an extension of $\mathrm{Jac}(\overline{C})$ by $\Gm^{r-1}$ and $\rho_{g+r-1}$ is birational.  In particular, the image $\im(\rho_{g+r-1})$ contains a dense  open subset of $\mathrm{Alb}(C)$.  Therefore, for any point $P \in \mathrm{Alb}(C)$, the intersection of $\im(\rho_{g+r-1})$ with $P - \im(\rho_{g+r-1})$ must be non-empty. Hence, the morphism $\rho_{2g+2r-2}$ is surjective.
\end{proof}

\begin{proposition} \label{prop:semiabelian-r+g-points0}
Let $r > 0$ be an integer.
Let $\overline{C}$ be a smooth projective irreducible curve over $k$ of genus $g$ and let $C\subset \overline{C}$ be a dense open subscheme such that $\# (\overline{C}\setminus C) = r$. Then, there are points $c_0,\ldots, c_{2(g+r-1)^2} \in C(k)$ such that, for every semi-abelian variety $X$ over $k$ and every $x_0,\ldots, x_{2(g+r-1)^2} \in X(k)$, the set 
\[
\Hom_k\!\left([C, (c_0,\ldots,c_{2(g+r-1)^2})], [X,(x_0,\ldots,x_{2(g+r-1)^2})]\right)
\] has at most one element. 
\end{proposition}

\begin{proof}
Let $X$ be any semi-abelian variety over $k$. Let $c_0 \in C(k)$ be a point and let $A$ be the Albanese variety of $C$ over $k$ together with the associated Abel-Jacobi map $\iota: C \to A$ mapping $c_0$ to the identity in $A$.

Note that a morphism $C \to X$ mapping $c_0$ to the identity in $X$, gives rise to a morphism $A\to X$ of group schemes (see \cite[Proposition~A.3]{Mochizuki1}). We will prove that such morphisms are determined by the image of at most $g+r-1$ carefully chosen points on $A$.
In order to find these points, note that there is an exact sequence $1 \to \Gm^{r-1} \to A \to J \to 0$, where ${J}$ is the Jacobian of $\overline{C}$, cf.\ the description of the generalised Jacobian given in the proof of Lemma \ref{lem:345}.

We will first look at morphisms ${J} \to X$.
Suppose we have an isogeny $J_1 \times \ldots \times J_n \to {J},$ such that $J_1, \ldots, J_n$ are simple abelian varieties over $k$. Then $n \leq g$ and by composition we get an injection
$$\Hom_{\mathrm{Grp}/k}({J},X) \hookrightarrow \Hom_{\mathrm{Grp}/k}(J_1 \times \ldots \times J_n, X).$$
For $1\leq i\leq n$, let $g_i \in J_i(k)$ be a point of infinite order and note that  $g_i$ generates a (Zariski) dense subgroup of $J_i$. Therefore, any morphism of group schemes $\psi:J_i \to X$ is determined by the image $\psi(g_i)$ of $g_i$. Now, for $1\leq i \leq n$, we let $j_i$ be the image of $(0, \ldots, 0, g_i, 0, \ldots, 0) \in J_1 \times \ldots \times J_n$ in ${J}$. Then, a morphism of semi-abelian varieties $\psi: {J} \to X$ is determined by the images $\psi(j_1), \ldots, \psi(j_n)$ of $j_1, \ldots, j_n$.

On the other hand, group morphisms $\Gm^{r-1} \to X$ are determined by the images of the points $\ell_1 = (2, 1, \ldots, 1), \ell_2 = (1,2,1,\ldots,1), \ldots,$ and $\ell_{r-1} = (1, \ldots, 1, 2)$ in $\Gm^{r-1}(k)$.

In particular, if $\ell_{r}, \ldots, \ell_{r+n-1} \in J(k)$ are points mapping to $j_1, \ldots, j_n$, then we claim that a morphism of group schemes $\varphi \colon A \to X$ is determined by $\varphi(\ell_1), \ldots, \varphi(\ell_{r+n-1})$. Indeed, if $\varphi'$ is another morphism with the same images, then $\varphi - \varphi'$ is trivial on $\Gm^{r-1}$ and factors through ${J}$, where it is also trivial.

By Lemma \ref{lem:345}, for each $1\leq i\leq n+r-1$, there exist points $c_{i,1}, \ldots, c_{i,2g+2r-2} \in C(k)$ such that $\iota(c_{i,1}) + \ldots + \iota(c_{i,2g+2r-2}) = \ell_i$. Now, we claim that, for any $1+(n+r-1)(2g+2r-2)$-tuple of points $x_{0}, x_{1,1}, x_{1,2}, \ldots, x_{n+r-1,2g+2r-2} \in X(k)$, the set $$H = \Hom([C,(c_{0}, c_{1,1}, \ldots, c_{n+r-1,2g+2r-2})], [X, (x_{0}, x_{1,1}, \ldots x_{n+r-1,2g+2r-2})])$$ is finite. Indeed, if we change the group structure on $X$ such that $x_0$ becomes the identity, then any morphism $f \colon C \to X$ in $H$ gives rise to a homomorphism of group schemes $h \colon A \to X$. Now, by definition, for each $i = 1, \ldots, n+r-1$, we have that $$h(\ell_i) = h(\iota(c_{i,1})) + \ldots + h(\iota(c_{i,2g+2r-2})) = f(c_{i,1}) + \ldots + f(c_{i,2g+2r-2}) = x_{i,1} + \ldots + x_{i,2g+2r-2}$$ is fixed. This implies that  $H$ can have at most one element.
The proposition is now proven by relabelling the points in $C$ and appending some extra points in case $n < g$.
\end{proof}

We obtain the following uniform finiteness statement for (not necessarily hyperbolic) curves.

\begin{lemma}\label{lem:curves_are_mb}  
Let $C$ be a smooth affine curve over $k$. Then, there is an integer $m\geq 1$ and points $c_1,\ldots, c_m\in C(k)$ such that, for every projective variety $X$ over $k$  of dimension at most one without rational curves over $k$ and every $x_1,\ldots, x_m\in X(k)$, the set 
\[
\Hom_k([C,(c_1,\ldots,c_m)],[X,(x_1,\ldots,x_m)])
\] 
is finite. 
\end{lemma}
\begin{proof} It suffices to prove the required finiteness statement for smooth projective connected varieties $X$ of dimension at most one without rational curves. However, since a smooth  projective irreducible curve $X$ over $k$ with no rational curves embeds into  an abelian variety (e.g., the  Jacobian of $X$),  the result follows from Proposition \ref{prop:semiabelian-r+g-points0}.
\end{proof}

\begin{proof}[Proof of Proposition \ref{prop:semiabvar_is_mildbounded}]
The mild boundedness of semi-abelian varieties follows immediately from Proposition \ref{prop:semiabelian-r+g-points0}.
\end{proof}

In fact, we can prove another (slightly less effective) finiteness result for semi-abelian varieties. We briefly explain this in the following remark.

\begin{remark}
 If $Y$ is a variety over $k$, then there are $y_1,\ldots, y_m$  in $Y(k)$ such that, for every semi-abelian variety $X$ and every $x_1,\ldots,x_m$ in $X(k)$, the set $$\Hom([Y,(y_1,\ldots,y_m)], [X,(x_1,\ldots,x_m)])$$ is finite.  To prove this, replacing $Y$ by a suitable dense open subscheme if necessary, we may and do assume that $Y$ is a smooth affine integral variety which maps injectively into   its Albanese variety $\mathrm{Alb}(Y)$.  Then, we pick a high power $Y^n$ which generates $\mathrm{Alb}(Y)$ and we use the argument in the proof of Proposition \ref{prop:semiabelian-r+g-points0} to construct an integer $m\geq 1$ and points $y_1,\ldots,y_m$ with the desired property. The integer $m$ and the points $y_1,\ldots, y_m$ only depend on $Y$ (i.e., they do not depend  on $X$).
\end{remark}

\begin{remark}
If $\mathbb{F}$ is an algebraically closed field of positive characteristic with  positive transcendence degree over its prime field and $A$ is an abelian variety over $\mathbb{F}$, then the proof of Proposition \ref{prop:semiabelian-r+g-points0} shows that $A$ is mildly bounded over $\mathbb{F}$. However, if $E$ is an elliptic curve over  $\overline{\mathbb{F}_p}$, then $E$ is not  mildly bounded over $\overline{\mathbb{F}_p}$.
Indeed, for any finite subset of points $p_1, \ldots, p_m \in E(\overline{\mathbb{F}_p})$ there is some integer $N > 0$ such that $N \cdot p_i = 0$ for all $i$, but then $(m_{jN})_{j=1}^\infty$, where $m_{jN} \colon E \to E$ is the multiplication by $jN$, is an infinite family of distinct morphisms mapping the $p_i$'s to the same point. Hence, the assumption that $\mathbb{F}$ has transcendence degree at least one over its prime field is necessary in the previous statement.
\end{remark}

 \begin{corollary} \label{cor:qf_semi_is_mb}
 Let $X$ be a variety over $k$ which admits a quasi-finite morphism to a semi-abelian variety over $k$. Then, for any field extension $k\subset L$, the variety $X_L$ is mildly bounded over $L$.
 \end{corollary}
 \begin{proof} Let $G$ be a semi-abelian variety over $k$ and let $X\to G$ be a quasi-finite morphism of varieties over $k$. Note that the induced morphism $X_L\to G_L$ is quasi-finite. Since $G_L$ is mildly bounded over $L$ (Proposition \ref{prop:semiabelian-r+g-points0}), the quasi-finiteness of $X_L\to G_L$ implies that   $X_L$ is mildly bounded over $L$.
 \end{proof}
 
\begin{corollary}\label{cor:curves_are_mbb}
Let $X$ be an integral one-dimensional variety over $k$ whose normalisation is not isomorphic to $\mathbb{A}^1_k$ nor $\mathbb{P}^1_k$. Then $X$ is mildly bounded over $k$.
\end{corollary}
\begin{proof}
Let $\widetilde{X}$ be the normalisation of $X$. As $\widetilde{X}$ is neither $\mathbb{A}^1_k$ nor $\mathbb{P}^1_k$,  the Abel-Jacobi map $\widetilde{X} \to \mathrm{Alb}(\widetilde{X})$ is injective. Hence,  by Corollary \ref{cor:qf_semi_is_mb},  the curve $\widetilde{X}$ is mildly bounded over $k$. As any non-constant map from a smooth curve $C$ over $k$ to $X$ factors uniquely through $\widetilde{X}$,   it follows that $X$ is also mildly bounded over $k$.
\end{proof}

The results in this section are motivated by Conjecture \ref{conj:new} which predicts that a projective variety with no rational curves is mildly bounded. Note that Corollary \ref{cor:curves_are_mbb} proves Conjecture \ref{conj:new} for one-dimensional projective varieties. We now prove Theorem \ref{thm:surfaces_are_mb} which says that our conjecture holds for projective surfaces, under suitable assumptions on the Albanese map.
\begin{proof}[Proof of Theorem \ref{thm:surfaces_are_mb}]
Let $X$  be a projective surface with no rational curves (as in the statement), let $A$ be an abelian variety, and let $f\colon X\to A$ be a     morphism which is generically finite onto its image. Let $X\to B\to A$ be the Stein factorisation of $f$, where $X\to B$ is a morphism with connected fibres and $B\to A$ is finite.  Let $B'\subset  B$ be the image of the morphism $X\to B$. Since the morphism $X\to B$ is  generically finite onto its  image $B'$   with  geometrically connected fibres, we have that the induced morphism $p:X\to B'$ is birational. Let $Z\subset B'$ be a proper closed subset  such that $p$ induces an isomorphism $X\setminus \Delta  \to B'\setminus Z$, where $\Delta = p^{-1}(Z)$.

Now, to show that $X$ is mildly bounded, let $C$ be a smooth irreducible curve over $k$.   Since $A$ is mildly bounded over $k$ (Proposition \ref{prop:semiabvar_is_mildbounded}) and $B'\to A$ is finite, it follows that $B'$ is mildly bounded over $k$.  Moreover, by Corollary \ref{cor:curves_are_mbb},  as $\Delta\subset X$ is at most one-dimensional and does not admit a non-constant morphism from $\mathbb{P}^1_k$ (by our assumption that $X$ has no rational curves), we have that $\Delta$ is also mildly bounded. Thus, we may choose an integer $n\geq 1$ and points $c_1,\ldots, c_n$ in $C(k)$ such that, for every $b_1,\ldots,b_n$ in $B'(k)$ and $d_1,\ldots,d_n$ in $\Delta(k)$, the sets
\[
\Hom_k([C,(c_1,\ldots,c_n)], [B',(b_1,\ldots,b_n)])
\] and 
\[
\Hom_k([C,(c_1,\ldots,c_n)], [\Delta,(d_1,\ldots,d_n)])
\]are finite.  Let $x_1,\ldots,x_n$ be points in $X(k)$. To show that the set
\[
\Hom_k([C,(c_1,\ldots,c_n)], [X,(x_1,\ldots,x_n)]) 
\] is finite, we note that it is  the union of its subsets $$  \Hom_k([C,(c_1,\ldots,c_n)], [X,(x_1,\ldots,x_n)]) \setminus \Hom_k(C,\Delta)  $$ and $$ \Hom_k([C,(c_1,\ldots,c_n)], [\Delta,(x_1,\ldots,x_n)]).$$ The former set maps  injectively to  $\Hom([C,(c_1,\ldots,c_n)], [B', (p(x_1),\ldots, p(x_n))])$, and is therefore finite. Moreover,  
$$ \Hom_k([C,(c_1,\ldots,c_n)], [\Delta,(x_1,\ldots,x_n)])$$ is finite by our choice of $c_1,\ldots,c_n$. (This set is by definition empty if there is an $i$ with $x_i\not\in \Delta$.)  We conclude that $X$ is mildly bounded, as required.
\end{proof}

 \section{Applying Silverman's specialisation theorem}\label{section6}

For a regular extension of fields $K \subset L$ and an abelian variety $A$ over $L$, the trace of $A$ with respect to  $K\subset L$ is a universal map $T_L \to A$ from an abelian variety $T$ over $K$ in the sense that all such other maps factor through it, see for example \cite[chap.\ VII]{LangAV} or \cite{ConradLangNeron}. The trace can be viewed as a measure of how close $A$ is to being a constant abelian variety. 

 In this section, as usual, we let $k$ be an algebraically closed field of characteristic zero.
 \begin{lemma}\label{lem:traceless}  
Let $C$ be an integral curve over $k$ and let $X \to C$ be an abelian scheme such that  the $K(C)/k$-trace of the abelian variety $X_{K(C)}$   is trivial.  Then there is a point $c \in C(k)$ such that for every $x \in  {X}(k)$ there are only finitely many sections $\sigma \colon C \to X$ with $\sigma(c) = x$.
\end{lemma}

\begin{proof}
Since $X_{K(C)}$ has trivial $K(C)/k$-trace, by the theorem of Lang-N\'eron \cite[Theorem~2.1]{ConradLangNeron}, the group $X(C)$ of sections of $X\to C$ is finitely generated. We now descend all the necessary data (including the elements of the finitely generated group $X(C)$) to a finitely generated subfield of $k$.

Let $L$ be a finitely generated  field over $\QQ$ contained inside $k$, let $\mathcal{C}$ be an integral curve over $L$, let $\mathcal{C}_k \cong C$ be an isomorphism over $k$, and let $\mathcal{X}\to \mathcal{C}$ be an abelian scheme such that $\mathcal{X}_k \to \mathcal{C}_k  $ is isomorphic to $X\to C$ over $k$ and such that the group of sections $\mathcal{X}(\mathcal{C})$ of $\mathcal{X}\to \mathcal{C}$ equals $X(C)$.

Note that the $K(\mathcal{C})/L$-trace of $\mathcal{X}_{K(\mathcal{C})}$ is zero.   
Let $\{\sigma_1, \ldots, \sigma_{r}\}$ be a basis for the free part of $\mathcal{X}(\mathcal{C}) = \mathcal{X}_{K(\mathcal{C})}(K(\mathcal{C}))$. 
Since $L$ is a finitely generated  field over $\QQ$, by Silverman's specialisation theorem \cite[Theorem~1]{Wazir}, there is a finite field extension $L'/L$ contained in $k$ and  a point  $c \in \mathcal{C}(L') \subset C(k) $ such that $\sigma_{1}(c), \ldots, \sigma_{r}(c) \in X_c(k)$ are still independent.

Let $c \in C(k)$ be such a point. Then, for a fixed $x \in X(k)$, there are only finitely many $\sigma \in X(C)$ with $\sigma(c) = x$. Indeed, for each of the finitely many torsion sections $\tau \in X(C)^{\mathrm{tors}}$, there is at most one tuple $(n_1,\ldots, n_r)$ of integers, such that $(\tau + n_1\sigma_1 + \ldots + n_r\sigma_r)(c) = x$, due to the independence of $\sigma_{1}(c), \ldots, \sigma_r(c)$.
\end{proof}

We now  combine Lemma \ref{lem:traceless} with the ``uniform'' mild boundedness of abelian varieties   (Proposition \ref{prop:semiabelian-r+g-points0}) to prove the following result.  
\begin{lemma}\label{lem:sections}
Let $k$ be an algebraically closed field of characteristic zero. Let $C$ be an integral   curve  over $k$ and let $\mathcal{X}\to C$ be an abelian scheme. Then there is an $n\geq 1$ and $c_1,\ldots, c_n \in C(k)$ such that, for every $x_1,\ldots,x_n \in \mathcal{X}(k)$, there are only finitely many sections $\sigma:C\to \mathcal{X}$  of $\mathcal{X}\to C$ with $\sigma(c_i)  = x_i$. 
\end{lemma}
\begin{proof}
The proof goes by induction on the relative dimension of $\mathcal{X} \to C$. The case of relative dimension 0 is trivial. Now assume the statement is true for all abelian schemes of smaller relative dimension.

Let $T \to \mathcal{X}_{K(C)}$ be the $K(C)/k$-trace of $\mathcal{X}_{K(C)}$. In case, $T = 0$, the result follows from Lemma \ref{lem:traceless}, so assume $T \neq 0$.
 Let $W$ be the cokernel. Shrinking $C$ if necessary, we can spread this out to an exact sequence
$$\mathcal{T} \to \mathcal{X} \to \mathcal{W}$$
of abelian varieties over $C$, such that $\mathcal{T}$ is the base change of an abelian variety over $k$ and the relative dimension of $\mathcal{W} \to C$ is smaller than that of $\mathcal{X} \to C$. It now suffices to prove that $\mathcal{T}$ and $\mathcal{W}$ have the property as stated in the statement of the lemma. 

Indeed, suppose $c_1, \ldots, c_{n_t}, d_1, \ldots, d_{n_w} \in C(k)$ are points such that $$\Hom_C([C, (c_1, \ldots, c_{n_t})], [\mathcal{T}, (t_1, \ldots, t_{n_t})])$$ $$\textrm{ and } \qquad \Hom_C([C, (d_1, \ldots, d_{n_w})], [\mathcal{W}, (w_1, \ldots, w_{n_w})])$$
are finite for any $t_1, \ldots, t_{n_t} \in \mathcal{T}(k)$ and $w_1, \ldots, w_{n_w} \in \mathcal{W}(k)$. Now consider morphisms $$\varphi, \varphi' \in \Hom_C([C, (c_1, \ldots, c_{n_t}, d_1, \ldots, d_{n_w})], [\mathcal{X}, (x_1, \ldots, x_{n_t+n_w})]),$$ where $x_1, \ldots, x_{n_t+n_w} \in \mathcal{X}(k)$ are arbitrary points. For these morphisms, there are only finitely many possibilities for the composed map $C \to \mathcal{X} \to \mathcal{W}$. If we suppose that $\varphi$ and $\varphi'$ give rise to the same composed map $C \to \mathcal{W}$, then their difference $\nu = \varphi - \varphi'$ factors through a morphism $\rho \colon C \to \mathcal{T}$. Moreover, $\rho(c_i)$ maps to $\nu(c_i) = 0 \in \mathcal{X}$ for each $i \in \{1, \ldots, n_t\}$. As the kernel of $\mathcal{T} \to \mathcal{X}$ is finite, this means that there are only finitely many candidates for the points $\rho(c_i)$. For each of the finitely many choices for $\rho(c_i) \in \mathcal{T}(k)$, there are only finitely many possibilities for $\rho$. Hence, there are indeed finitely many possibilities for $\varphi$ as well.

To conclude the proof, let us show  that $\mathcal{T}$ and $\mathcal{W}$ have the property as in the statement of the lemma.
In the case of $\mathcal{T}$, we write $\mathcal{T}$ as $A \times_{k} K(C)$ for a certain abelian variety $A$ over $k$. Then there is a bijection
$$\Hom_{C}(C, A \times_{k} C) = \Hom_{k}(C, A),$$
and the result follows from Proposition \ref{prop:semiabvar_is_mildbounded}.  In the case of $\mathcal{W}$, as the relative dimension of $\mathcal{W}$ over $C$ is smaller than the relative dimension of $\mathcal{X}$ over $C$, the result follows from the induction hypothesis.
\end{proof}

The following lemma generalises the above lemma (Lemma \ref{lem:sections}) for abelian schemes to semi-abelian schemes.

\begin{lemma}\label{lem:sections2}
Let $k$ be an algebraically closed field of characteristic zero. Let $C$ be an integral   curve  over $k$ and let $\mathcal{X}\to C$ be a semi-abelian scheme. Then there is an $n\geq 1$ and $c_1,\ldots, c_n \in C(k)$ such that, for every $x_1,\ldots,x_n \in \mathcal{X}(k)$, there are only finitely many sections $\sigma:C\to \mathcal{X}$ with $\sigma(c_i)  = x_i$. 
\end{lemma}
\begin{proof} Note that the generic fibre  $\mathcal{X}_{K(C)}$ of $\mathcal{X}\to C$  is an extension by a torus over $K(C)$ and an abelian variety over $K(C)$.  Therefore, we may choose  a smooth integral curve $D$ over $k$ and a quasi-finite (flat dominant) morphism $D\to C$  such that $\mathcal{X}_D = \mathcal{X}\times_C D$ is an extension of
$\mathbb{G}_{m,D}^{\ell}$ and an abelian scheme $\mathcal{A}$ over $D$, i.e., there is an exact sequence $$1 \to \Gmu{D}^{\ell} \to \mathcal{X}_D \stackrel{\nu}{\to} \mathcal{A} \to 0.$$ By Lemma \ref{lem:sections}, we find $d_1, \ldots, d_{m} \in D(k)$ such that for every $a_1, \ldots, a_m \in \mathcal{A}(k)$ the set $$\Hom_D([D, (d_1, \ldots, d_m)], [\mathcal{A}, (a_1, \ldots, a_m)])$$ of $m$-pointed sections is finite. Using Proposition \ref{prop:semiabelian-r+g-points0}, we find points $d_{m+1}, \ldots, d_n \in D(k)$ such that, for every $g_{m+1}, \ldots, g_n \in \Gmu{D}^\ell(k)$, the set $$\Hom_D([D, (d_{m+1}, \ldots, d_n), [\Gmu{D}^\ell, (g_{m+1}, \ldots, g_n)])$$ of $(n-m)$-pointed sections is finite. Let $c_1, \ldots, c_n \in C(k)$ be the images of the points $d_1, \ldots, d_n$. These points have the desired property using arguments  similar to those used in the proof of Lemma \ref{lem:sections}.  

Indeed, let $x_1, \ldots, x_n \in \mathcal{X}(k)$ be arbitrary. Then any
$$\varphi \in \Hom_C([C, (c_1, \ldots, c_n)], [\mathcal{X}, (x_1, \ldots, x_n)])$$
lifts to a morphism
$$\varphi_D \in H_{y_1, \ldots, y_n} := \Hom_D([D, (d_1, \ldots, d_n)], [\mathcal{X}_D, (y_1, \ldots, y_n)]),$$
where $y_1, \ldots, y_n \in \mathcal{X}_D(k)$ are lifts of $x_1, \ldots, x_n$. Since there are only finitely many such lifts, it suffices to prove that $H_{y_1, \ldots, y_n}$ is finite. Let $\varphi_D, \varphi_D' \in H_{y_1, \ldots, y_n}$ be two elements, and let $\sigma, \sigma' \colon D \to \mathcal{A}$ be the composition of $\varphi_D$ and $\varphi_{D'}$ with $\nu$. Then $\sigma$ and $\sigma'$ lie in the finite set
$$\Hom_D([D, (d_1, \ldots, d_m)], [\mathcal{A}, (\nu(y_1), \ldots, \nu(y_n))])$$
Assuming $\sigma$ and $\sigma'$ are equal, we will now prove that there are only finitely many possibilities for $\varphi_{D} - \varphi_{D}'$. Indeed, this difference $\varphi_D - \varphi_D'$ factors through $\Gmu{D}^{\ell}$ and lies in the finite set
$$\Hom_D([D, (d_{m+1}, \ldots, d_n)], [\Gmu{D}^{\ell}, (1, \ldots, 1)]).$$
This proves that there are only finitely many elements in $H_{y_1,\ldots,y_n}$, which concludes the proof of the lemma.
\end{proof}

We now combine the above results  with the fact that semi-abelian varieties are mildly bounded in a ``uniform'' sense (Proposition \ref{prop:semiabelian-r+g-points0}). 
\begin{corollary}\label{cor:mb_of_total_space} 
Let $k$ be an algebraically closed field of characteristic zero.
Let $S$ be an integral   variety  over $k$ and let $f:\mathcal{X}\to S$ be a semi-abelian scheme.  Assume that for every smooth integral curve $C$ over $k$, the set of non-constant morphisms $C\to S$ is finite. Then the variety $\mathcal{X}$ is mildly bounded over $k$.
\end{corollary}

\begin{proof} 
Let $C$   be a smooth irreducible curve over $k$. Choose a point $c_0$ in $C(k)$.    Since there are only finitely many non-constant morphisms $C\to S$ (by assumption), for any $x_0$ in $X(k)$ and every $\rho:C\to X$ with $\rho(c_0)=x_0$, there are only finitely many possibilities for the composed morphism $\nu:C\to X \to S$.  We now choose $x_0$ in $X(k)$ and consider such morphisms $\rho:C\to X$ such that $\rho(c_0) = x_0$.

 If $\nu$ maps $C$ onto a point (which has to be $f(x_0)$), then we already saw in Proposition \ref{prop:semiabelian-r+g-points0} that there exists an integer $m\geq 1$ and  points $c_{1}, \ldots, c_{m} \in C(k)$, only depending on $C$ (not depending on $x_0$), such that
$$\Hom_k([C, (c_0, c_{1}, \ldots, c_{m})], [\mathcal{X}_{f(x_0)}, (x_0, x_{1}, \ldots, x_m)])$$
is finite for any $x_{1}, \ldots, x_m \in \mathcal{X}_{f(x_0)}(k)$.

In the case  that $\nu$ is non-constant,  we consider the base change $\mathcal{X}_{\nu} = \mathcal{X} \times_{f, S,\nu} C$ which is a semi-abelian scheme over $C$. Then by Lemma \ref{lem:sections2}, there are points $c_{\nu,1}, \ldots, c_{\nu, n_{\nu}} \in C(k)$, such that
$$\Hom_C([C, (c_{\nu,1}, \ldots, c_{\nu,n_{\nu}})], [\mathcal{X}_{\nu}, (x_1, \ldots, x_{n_{\nu}})])$$
is finite for any $x_{1}, \ldots, x_{n_\nu} \in \mathcal{X}_{\nu}(k)$.  Note that $$\{c_0,c_1, \ldots, c_m\} \cup \bigcup_{\nu} \{c_{\nu,1}, \ldots, c_{\nu,n_\nu}\}$$ is a finite subset of $C(k)$ by our assumption that the set of non-constant morphisms $C\to S$ is finite.
In particular, if $c_0, \ldots, c_n$ are its elements, then $$\Hom([C, (c_0, \ldots, c_n)], [\mathcal{X}, (x_0, \ldots, x_n)])$$ is finite for any $x_0, \ldots, x_n \in \mathcal{X}(k)$. This  proves that $\mathcal{X}$ is mildly bounded over $k$.
\end{proof}

\begin{proof}[Proof of Theorem \ref{thm:families_of_ab_var_are_mb}]
By De Franchis-Severi's theorem, the hyperbolic curve $S$ satisfies the   finiteness statement required to apply Corollary \ref{cor:mb_of_total_space}.
\end{proof}

We now show that an analogue of  Theorem \ref{thm:families_of_ab_var_are_mb} also holds for the universal Jacobian. To prove this statement, one   replaces the use of De Franchis-Severi in the proof of Theorem \ref{thm:families_of_ab_var_are_mb} by the finiteness theorem of Arakelov-Parshin.

\begin{corollary} Let $k$ be an algebraically closed field of characteristic zero, let $N\geq 3$ be  an integer, let $g\geq 2$, and let $M$ be the fine moduli space  of smooth proper geometrically irreducible genus $g$ curves with level $N$ structure over $k$. Let $X\to M$ be the universal curve (with level $N$ structure) and let $\mathrm{Jac}(X)\to M$ be its Jacobian. Then the quasi-projective variety $\mathrm{Jac}(X)$ is mildly bounded over $k$.
\end{corollary}
\begin{proof}
By Arakelov-Parshin's theorem \cite{ArakelovShaf, Parshin0} (\emph{formerly} the Shafarevich conjecture for curves), for every   curve $C$ over $k$, the set of $C$-isomorphism classes of non-isotrivial smooth proper curves $X\to C$ of genus $g$ (with $g\geq 2$) is finite. In particular (by the definition of the moduli space $M$),  for every curve $C$ over $k$, the set of non-constant morphisms $C\to M$ is finite. Therefore, the variety $M$ satisfies the   finiteness property required to apply   Corollary \ref{cor:mb_of_total_space} (with $S:=M$ and $\mathcal{X}:=\mathrm{Jac}(X)$).
\end{proof}

To prove the final result of this section, we will need the following well-known (and very simple) lemma which says that every variety is ``locally'' hyperbolic (in any sense of the word).
\begin{lemma}\label{lem:obv}
Let $S$ be an integral variety over $k$. Then, there is a dense open subscheme $U\subset S$ such that $U$ is arithmetically hyperbolic over $k$ and, for every smooth curve $C$ over $k$, the set of non-constant morphisms $C\to U$ is finite.
\end{lemma}
\begin{proof} Note that $X:=\mathbb{A}^1_k\setminus\{0,1\}$ is arithmetically hyperbolic over $k$ (by Siegel-Mahler-Lang's theorem). Moreover, for an integral curve $C$ over $k$, the De Franchis-Severi theorem implies that the set of non-constant morphisms $C\to X$ is finite. Now, to prove the lemma, 
we may and do assume that $S$ is affine, say $S\subset \mathbb{A}_k^N$.  Choose $a,b$ in $\mathbb{A}^1(k)$ such that $S\cap (\mathbb{A}^1_k\setminus \{a,b\})^N$ is non-empty (hence dense and open in $S$). Define $U:=S\cap (\mathbb{A}^1_k\setminus \{a,b\})^N$ and note that $U$ maps with finite fibres to $X^N$. In particular, $U$ is arithmetically hyperbolic over $k$ and, for every integral curve $C$ over $k$, the set of non-constant morphisms $C\to U$ is finite. This proves the lemma.
\end{proof}
We now record the following application of the above results. Our motivation for proving the following result is that it is a first step towards reducing part of Lang--Vojta's   conjecture concerning rational points on hyperbolic projective varieties to $\Qbar$.

\begin{theorem} \label{thm:spreadout}
Let $k\subset L$ be an extension of algebraically closed fields of characteristic zero, and let $X$ be a semi-abelian variety over $L$. Then there exists an integral arithmetically hyperbolic variety $S$ over $k$ whose function field $K(S)$ is a subfield    of $L$, and a semi-abelian scheme $\mathcal{X}$ over $S$ such that $\mathcal{X}\times_S L \cong X$ over $L$ and $\mathcal{X}$ is mildly bounded over $k$.
\end{theorem}

\begin{proof}
By   spreading out arguments, there exists a finitely generated $k$-algebra $A\subset L$ with $S:=\Spec A$  an integral variety over $k$, and a semi-abelian scheme $\mathcal{X}\to S$ such that $\mathcal{X}_L\cong X$. Replacing $S$ by a dense open subscheme if necessary, by Lemma \ref{lem:obv}, the variety $S$ is arithmetically hyperbolic over $k$, and for every curve $C$ over $k$, the set of non-constant morphisms $C\to S$ is finite.    The result now follows from Corollary \ref{cor:mb_of_total_space}. 
\end{proof}

\section{A criterion for persistence of arithmetic hyperbolicity}\label{section7}
 Recall that a variety over $k$ is arithmetically hyperbolic over $k$ if every model for $X$ over a $\ZZ$-finitely generated subalgebra of $k$ has only finitely many sections (Definition \ref{defn:arhyp}); see also \cite{JBook, JAut, JLalg, JXie}.  The Lang--Vojta conjecture predicts that  arithmetic hyperbolicity persists over field extensions, i.e., if $X$ is arithmetically hyperbolic over $k$ and $ L$ is an algebraically closed field  containing  $k$, then $X_L$ should be arithmetically hyperbolic over $L$; this is the Arithmetic Persistence Conjecture (Conjecture \ref{conj:pers}). Although the Arithmetic Persistence Conjecture is not known to hold in general, the following result is  useful   for verifying it in many cases.
 
\begin{theorem}[Criterion for Persistence] \label{thm:gen_crit}
Let $X$ be an arithmetically hyperbolic variety over $k$ such that $X_K$ is mildly bounded for all subfields $k\subset K\subset L$. Then $X_L$ is arithmetically hyperbolic over $L$.
\end{theorem}
 \begin{proof}
 See   \cite[Theorem~1.6]{JAut}.
 \end{proof}
 
 As is shown in \cite{JAut}, Theorem \ref{thm:gen_crit}   follows  from  the following result.
 
 \begin{theorem}\label{thm:gen_crit0} 
 Let $X$ be an arithmetically hyperbolic mildly bounded variety over $k$. If $L/k$ is an extension of algebraically closed fields of transcendence degree one, then $X_L$ is arithmetically hyperbolic over $L$.
 \end{theorem}
 \begin{proof}
 See \cite[Lemma~4.3]{JAut}.
 \end{proof}
 
 Applications of the above two  results are given in \cite[\S4]{JAut}, and \cite{AJL, JLitt, JSZ, JXie}.
 As a first application of the results of this paper and the above criterion for persistence of arithmetic hyperbolicity, we obtain the following corollary for projective varieties.
 
 \begin{corollary}\label{cor:below}
 Let $X$ be a projective arithmetically hyperbolic  mildly bounded variety over $k$. Then, for every field extension $k\subset L$, the projective variety $X_L$ is arithmetically hyperbolic over $L$.
 \end{corollary}
 \begin{proof}
 It follows from the definition of arithmetic hyperbolicity that we may   assume   $L$ has finite transcendence degree over $k$. Then, by Corollary \ref{cor:mb_persists}, as $X$ is mildly bounded and projective over $k$, for every subfield $k\subset K\subset L$, the variety $X_K$ is mildly bounded over $K$. Thus, by the above criterion for persistence (Theorem \ref{thm:gen_crit}), we conclude that $X_L $ is arithmetically hyperbolic over $L$.
 \end{proof}
 
 As another application of the results of this paper, we can show that the Arithmetic Persistence Conjecture holds for  varieties which admit a quasi-finite morphism to some semi-abelian variety.
 
  \begin{theorem}\label{thm:vojta1} Let $k\subset L$ be an extension of algebraically closed fields of characteristic zero.
 Let $X$ be a variety over $k$ which admits  a quasi-finite morphism to some semi-abelian variety over $k$.
 Then $X$ is arithmetically hyperbolic over $k$ if and only if $X_L$ is arithmetically hyperbolic over $L$.
 \end{theorem}
 \begin{proof} 
 Note that, by Corollary    \ref{cor:qf_semi_is_mb}, for every subfield $k\subset K\subset L$, the variety $X_K$ is mildly bounded over $K$. Therefore, the result follows   from the criterion for persistence (Theorem \ref{thm:gen_crit}).
 \end{proof}
 
\begin{proof}[Proof of Theorem \ref{thm:semi_abvars_intro}]
This follows from Theorem \ref{thm:vojta1}.
\end{proof}
 
 In the following remark we explain why our results (and definitions) simplify proofs of  Faltings's finiteness theorem for finitely generated $\ZZ$-algebras on affine opens of abelian varieties.
 
 \begin{remark}[Faltings's theorem over finitely generated fields via specialisation]\label{remark:vojta}
 In \cite{FaltingsLang1} Faltings proved that, if $A$ is an abelian variety  over $\Qbar$ and $D\subset A$ is an ample effective divisor, the affine variety $A-D$ is arithmetically hyperbolic over $\Qbar$. That is,  for every number field $K$, every finite set of finite places $S$ of $K$, every model $\mathcal{A}$ for $A$ over $\OO_{K,S}$ and every model $\mathcal{D}\subset \mathcal{A}$ for $D\subset A$ over $\OO_{K,S}$, the set $(\mathcal{A}\setminus \mathcal{D})(\OO_{K,S})$ is finite. One sometimes also says that the set of $(D,S)$-integral points on $A$ is finite. We will now use our results to prove  the following extension of Faltings's theorem which does not seem to appear explicitly in the literature, although  related results are mentioned in \cite[\S5]{Faltings3}.

 \begin{theorem}[Faltings + $\epsilon$]\label{thm:fal}
 Let $L$ be an algebraically closed field of characteristic zero, let $A$ be an abelian variety over $L$, and let $D\subset A$ be an ample effective divisor. Then $A-D$ is arithmetically hyperbolic over $L$.
 \end{theorem}
 \begin{proof} (We use Faltings's finiteness theorem over $\Qbar$, the fact that abelian varieties are mildly bounded, and a specialisation argument.) 
 
 First, we may and do assume that $L$ has finite transcendence degree over $\mathbb{Q}$ and contains $\overline{\mathbb{Q}}$. We argue by induction on the transcendence degree $n=\mathrm{trdeg}_{\QQ}(L)$ of $L$ over $\QQ$. If $n = 0$, then $L=\Qbar$ so that the statement  follows from Faltings's theorem \cite[Corollary~6.2]{FaltingsLang1}. 
 
Now assume that $n>0$ and choose an algebraically closed subfield $k\subset L$ such that $\mathrm{trdeg}_{k}(L)=1$, so that $\mathrm{trdeg}_{\QQ}(k) =n-1$. We now apply Theorem \ref{thm:spreadout} to the abelian variety $A$ over $L$ and the extension $k\subset L$. Thus, 
  we choose   an   arithmetically hyperbolic integral variety $S$ over $k$ whose function field $K(S)$ is a subfield    of $L$   with $\overline{K(S)}=L$, an abelian scheme $\mathcal{A}$ over $S$ such that  $\mathcal{A}$ is mildly bounded over $k$ and $\mathcal{A}\times_{k} L $ is isomorphic to $ A$ over $L$. Now, shrinking $S$ if necessary, we may and do assume that 
  the ample divisor $D$ on $A$ extends to a relatively ample divisor $\mathcal{D}\subset \mathcal{A}$ on $\mathcal{A}$. Define $\mathcal{X}:=\mathcal{A}\setminus \mathcal{D}$. Note that, by the induction hypothesis, for every $s$ in $S(k)$, the fibre $\mathcal{X}_s$ of $\mathcal{X} \to S$ over $s$ is arithmetically hyperbolic over $k$ (as it is the complement of an ample divisor in an abelian variety over $k$). Therefore, as $S$ is arithmetically hyperbolic over $k$ and every $k$-fibre of $\mathcal{X}\to S$ is arithmetically hyperbolic, it is straightforward to see that the variety $\mathcal{X}$ is arithmetically hyperbolic over $k$ (see \cite[Lemma~4.11]{JLalg} for details). Since $\mathcal{X}$ is mildly bounded over $k$ and arithmetically hyperbolic over $k$, it follows from  Theorem \ref{thm:gen_crit0} that $X=\mathcal{X}_{\overline{K(S)}}$ is arithmetically hyperbolic over $L$.  This concludes the proof.
 \end{proof}
 \end{remark}
\subsection{Two remarks on integral points on abelian varieties} 
It is well-known that the set of rational points on an abelian variety over a number field is potentially dense, and there are ``many'' different proofs of this fact. In the following remark, we explain how the mild boundedness of abelian varieties implies the potential density of rational points. 
 
 \begin{remark}[Potential density of rational points on abelian varieties]  
 Let $A$ be an abelian variety over $k$. It is well-known that there is a finitely generated subfield $K\subset k$ and a model $\mathcal{A}$ for $A$ over $K$ such that the subset $\mathcal{A}(K)$ of $A(k)$ is dense in $A$. Indeed, this is due to  Frey--Jarden \cite{FreyJarden}. In this remark, we explain how to reprove this result. Firstly, using standard arguments, just as in the proof of Proposition \ref{prop:semiabelian-r+g-points0}, we may and do assume that $A$ is simple.  Let $ L$ be an uncountable algebraically closed field containing $k$. Then, as the torsion in $A(L)$ is countable and $L$ is an uncountable algebraically closed field, we see that $A(L)$ contains a point $P$ of infinite order.  Choose    a finitely  generated subfield $K_1\subset L$  and a model $\mathcal{A}$  for $A_L$ over $K_1$ such that $P$ lies in the subset $ \mathcal{A}(K_1)$ of $A(L)$.  Then   $\mathcal{A}(K_1)$ is infinite, so that in particular $A_L$ is not arithmetically hyperbolic over $L$.   Since $A_K$ is mildly bounded over $K$ for any subfield $k\subset K\subset L$ by Corollary \ref{cor:qf_semi_is_mb} and $A_L$ is not arithmetically hyperbolic over $L$, it follows that $A$ is not arithmetically hyperbolic over $k$. In particular, there is a finitely generated field $K\subset k$ and a model $\mathcal{A}$ for $A$ over $K$ such that $\mathcal{A}(K)$ is infinite. In particular, $\mathcal{A}(K)$ is dense, as $A$ is   simple.
 \end{remark}
 
 Another   application of the mild boundedness of abelian varieties is given in the following remark in which we briefly discuss Hassett-Tschinkel's arithmetic puncture problem \cite[Problem~5.3]{HT1}. This conjecture is also a consequence of conjectures of Campana \cite{Campana2}.
 
 \begin{remark}[Hassett-Tschinkel puncture problem]  \label{remark:ht}
Let $A$ be a simple abelian variety over $\Qbar$ and let $D\subset A$ be a closed subset with $\mathrm{codim}(A\setminus D)\geq 2$. In \cite[Problem~5.3]{HT1} Hassett and Tschinkel conjecture   that  there exists a number field $K$, a finite set of finite places $S$ of $K$, a model $\mathcal{A}$ for $A$ over $\OO_{K,S}$  and a model $\Delta\subset \mathcal{A}$ for $D\subset A$ over $\OO_{K,S}$  such that $(\mathcal{A}\setminus \Delta)(\OO_{K,S})$ is dense in $A$. In particular, Hassett and Tschinkel conjecture that  $A\setminus D$ is  \textbf{not} arithmetically hyperbolic over $\Qbar$. We note that Hassett-Tschinkel's  conjecture is in fact a special case of Campana's more general conjectures on potential density of integral points on log-special varieties \cite{Campana2}. Now, since abelian varieties are mildly bounded over any algebraically closed field of characteristic zero (Corollary \ref{cor:qf_semi_is_mb}), we can make the following observation.   \\
\begin{center} \emph{Suppose   there is an algebraically closed field $k$ containing $\Qbar$ such that $(A\setminus D)_k$ is not arithmetically hyperbolic over $k$. Then,   $A\setminus D$ is not arithmetically hyperbolic over $\Qbar$.} \\ [0.5cm]
\end{center} 
 Thus, in other words, to solve Hassett-Tschinkel's arithmetic puncture problem for the variety $A\setminus D$  it suffices to show the infinitude of integral points on $A\setminus D$ over some finitely generated subring of $\mathbb{C}$. This is arguably (most likely) easier to achieve, as can be seen in the case that $D=\emptyset$.  The existence of a point of infinite order in $A(\CC)$ follows from the countability of the group of torsion points on $A$ and the uncountability of $\CC$.
\end{remark}

  \section{Application to  irregular surfaces  }\label{section:faltingsvojta}

 The aim of this section is to prove the mild boundedness of certain surfaces. We then use this to prove the Arithmetic Persistence Conjecture \ref{conj:pers} for   surfaces which admit a non-constant map to some abelian variety. The most general result we obtain on the Persistence Conjecture in this section is Theorem \ref{thm:vojta2}.

 \begin{lemma}[Grauert-Manin + $\epsilon$]\label{lem:mordell} Let $C$ be a smooth integral curve over $k$, and let $X\to C$ be a quasi-projective flat morphism of integral schemes whose fibres are groupless  of dimension at most one.  Let $c\in C(k)$ and let $x\in X(k)$. Then the set of sections $\sigma:C\to X$ of $X\to C$ with $\sigma(c) = x$ is finite. 
 \end{lemma}
 \begin{proof} Replacing $C$ by a dense open if necessary, we may (and do) assume that $X\to C$ is smooth of relative dimension one, so that  $X$ is a smooth integral surface over $k$. Moreover, replacing $X\to C$ by its Stein factorization if necessary, we may and do assume that the fibres of $X\to C$ are geometrically connected (and thus geometrically irreducible). Since the fibres of $X\to C$ are groupless smooth quasi-projective irreducible curves, there is a finite \'etale morphism $Y\to X$ such that the geometric generic fibre of $Y\to C$ is of genus at least two. Replacing $X$ by $Y$, we may  assume that  the geometric fibres of $X\to C$ are smooth projective curves with genus at least two.  Moreover, to prove the lemma, we may assume that the set of sections $X(C)$ is infinite. In this case, by Grauert-Manin's theorem
\cite{Grauert65, Manin63},  there is a smooth projective irreducible curve $X_0$ over $k$ and an isomorphism $X\cong X_0\times_k C$ over $C$. Furthermore,  the set $X(C)\setminus X_0(k)$ is finite. (This latter finiteness statement follows from the theorem of De Franchis-Severi.)
 As there is precisely one element $\sigma$ in the subset $X_0(k)$ of $X(C)$ with  $\sigma:C\to X$ with $\sigma(c) = x$, this concludes the proof.
 \end{proof}
 
 \begin{lemma}\label{lem:geometric_statement} For every smooth irreducible curve $C$ over $k$, there is an integer $n\geq 1$ and points $c_1,\ldots, c_n$ having the following property: for any  quasi-projective variety $X$  over $k$ which admits a morphism $\varphi:X\to G$ to a semi-abelian variety $G$ over $k$ with groupless fibres of dimension at most one and   every $x_1,\ldots,x_n$ in $X(k)$, the set  $$\Hom_k([C,(c_1,\ldots,c_n)],[X,(x_1,\ldots,x_n)])$$ is finite. In particular, the variety $X$ is mildly bounded over $k$.
 \end{lemma}

\begin{proof}   (We will apply Grauert-Manin's theorem in the form of Lemma \ref{lem:mordell}.  We could avoid appealing to Grauert-Manin's theorem and instead use  our results on the mild boundedness of total spaces of abelian schemes. However, appealing to the Grauert-Manin theorem allows us to give a shorter proof.) 

Let $C$ be a smooth affine irreducible curve over $k$.  Choose $n\geq 1$ and $c_1,\ldots,c_n$ in $C(k)$ as in Proposition \ref{prop:semiabelian-r+g-points0}, so that for every semi-abelian variety $A$ over $k$ and every $a_1,\ldots,a_n$ in $A(k)$ the set   \[
\Hom_k([C,(c_1,\ldots,c_n)], [A,(a_1,\ldots,a_n)])
\] is finite.  

Let $H$ be the subset of $\Hom_k(C,X)$ of morphisms $f:C\to X$ such that the composed morphism $C\to X\to G$ is constant.   If the set $\Hom_k([C,(c_1,\ldots,c_n)],[X,(x_1,\ldots,x_n)])\cap H$ is non-empty, then     $x_1,\ldots,x_n$ lie in $X_{\varphi(x_1)}$ and 
\begin{align*}
 \Hom_k([C,(c_1,&\ldots,c_n)],[X,(x_1,\ldots,x_n)]) \cap H\\ 
&=  \Hom_k([C,(c_1,\ldots,c_n)], [X_{\varphi(x_1)}, (x_1,\ldots,x_n)])
\end{align*} which is finite by the defining property of $c_1,\ldots,c_n$, and the fact that $  X_{\varphi(x_1)}$ is at most one-dimensional and groupless. Thus, to prove the lemma, it suffices to show that 
  \[\Hom_k([C,(c_1,\ldots,c_n)],[X,(x_1,\ldots,x_n)]) \, \setminus \, H \] is finite.
  
  Now, let $g_i:=\varphi(x_i)$ and note that 
  \[
  \Hom_k^{\mathrm{nc}}([C,(c_1,\ldots,c_n)], [G,(g_1,\ldots,g_n)]),
  \] the set of non-constant morphisms $C \to G$ mapping $c_i$ to $g_i$ for $i = 1, \ldots, n$, is finite.  Let $\nu_1,\ldots,\nu_r$ be its elements. Then, for every $1\leq i \leq r$, we consider the morphism $\varphi_i:X_{i}\to C$, where $X_{i} :=X\times_{\varphi, G, \nu_i} C$ is the pull-back of $X\to G$ along $\nu_i:C\to G$. As the morphism $\varphi_i$ has groupless fibres of dimension at most one, it follows from Grauert-Manin's theorem (Lemma \ref{lem:mordell})  that there are only finitely many  sections which map $c_1$ to $x_1$.

In particular,
\begin{align*}
\Hom_k&([C,(c_1,\ldots,c_n)],[X,(x_1,\ldots,x_n)]) \, \setminus \, H\\
&= \bigcup_{i=1}^r \, \Hom_C([C, (c_1, \ldots, c_n)], [X_i, (x_1, \ldots, x_n)])
\end{align*}
is finite. This finishes the proof of the statement.
\end{proof}

\begin{proof}[Proof of Theorem \ref{thm:surfaces_q}]
This follows from Lemma \ref{lem:geometric_statement}.
\end{proof}

  \begin{lemma}\label{lem:vojtatje}
Let $k\subset L$ be an extension of algebraically closed fields of characteristic zero and let $X$ be a quasi-projective  variety over $k$ which  admits a   morphism to some semi-abelian variety over $k$ with fibres of dimension at most one. If $L/k$ has transcendence degree one and  $X$ is arithmetically hyperbolic over $k$, then $X_L$  is arithmetically hyperbolic over $L$.
\end{lemma}

\begin{proof} (We  adapt the proof of  \cite[Lemma~4.3]{JAut}, and use both Grauert-Manin's function field Mordell conjecture and the fact that abelian varieties are mildly bounded.)

Let $G$ be an abelian variety over $k$ and let $f:X\to G$ be a non-constant morphism. Let $A\subset k$ be a $\ZZ$-finitely generated subring, let $\mathcal{X}$ be a projective model for $X$ over $A$, let $\mathcal{G}$ be an abelian scheme over $A$, and let $F:\mathcal{X}\to \mathcal{G}$  be a morphism of $A$-schemes such that $F_k = f$. Let $B\subset L$ be a $\ZZ$-finitely generated subring  containing $A$. To prove that $X_L$ is arithmetically hyperbolic over $L$, we show that the set $\mathcal{X}(B)$ of morphisms $\Spec B \to \mathcal{X}$  over $A$ is finite.  Replacing $B$ by a larger $\ZZ$-finitely generated subring of $L$ if necessary, we may and do assume that the affine scheme $\mathcal{C}:=\Spec B$ is smooth over $A$.

 If $\dim B = \dim A$, then $B$ is contained in $k$, so that the finiteness of $\mathcal{X}(B)$ follows from the assumption that $X$ is arithmetically hyperbolic over $k$.
Therefore, we may and do assume that $\dim \mathcal{C} = \dim A+ 1$, i.e., the ``arithmetic scheme'' $\mathcal{C}$ is a ``curve'' over $A$.
Define $C:=\mathcal{C}\times_A k$, and note that $C$ is a smooth affine one-dimensional scheme over $k$. 

Note that, as $X$ is arithmetically hyperbolic over $k$, it follows that $X$ is groupless over $k$ (see \cite[\S3]{JAut}). In particular, as $X$ is groupless over $k$, the fibres of $f:X\to G$  are groupless. Therefore,  as the fibres of $f:X\to G$ are groupless of dimension at most one, we can  apply Lemma \ref{lem:geometric_statement}. Thus, we choose an integer  $n\geq 1$ and points $c_1,\ldots, c_n$ in $C(k)$  such that, for every $x_1,\ldots,x_n$ in $X(k)$, the set 
\[
\Hom_k([C,(c_1,\ldots,c_n)],[X,(x_1,\ldots,x_n)])
\] is finite. Next, we choose   a $\ZZ$-finitely generated subring $A'\subset k$   containing $A$ such that the   points $c_1,\ldots, c_n$ in $C(k)$ descend to sections $\sigma_1,\ldots, \sigma_n$ of $\mathcal{C}'=\mathcal{C}\times_A A'$ over $A'$. Since $\mathcal{X}(\mathcal{C})\subset \mathcal{X}(\mathcal{C}') = \Hom_A(\mathcal{C}',\mathcal{X})$, it suffices to show that $\mathcal{X}(\mathcal{C}')$ is finite.  

By assumption, the variety $X$ is arithmetically hyperbolic over $k$, so that $\mathcal{X}(A')$ is finite. Also, we have the following inclusion of sets  
\[
\mathcal{X}(\mathcal{C}') \subset \bigcup_{(x_1,\ldots,x_n)\in \mathcal{X}(A')^n} \Hom_{A'}([\mathcal{C}',(\sigma_1,\ldots,\sigma_n)],[\mathcal{X}_{A'},(x_1,\ldots,x_n)]),
\] where $ \Hom_{A'}([\mathcal{C}',(\sigma_1,\ldots,\sigma_n)],[\mathcal{X}_{A'},(x_1,\ldots,x_n)]) $  is the set of morphisms $P:\mathcal{C}'\to \mathcal{X}_{A'}$ such that $P(\sigma_1) = x_1,\ldots, P(\sigma_n)=x_n$.  Therefore, it suffices to show that, for any choice of  (not necessarily pairwise distinct) elements  $x_1,\ldots, x_n$ in the finite set $ \mathcal{X}(A')$, the set 
\[ 
\Hom_{A'}([\mathcal{C}',(\sigma_1,\ldots,\sigma_n)],[\mathcal{X}_{A'},(x_1,\ldots,x_n)])
\] is finite. Thus, let us fix sections $x_1,\ldots,x_n$ of $\mathcal{X}(A')$.   As there are more morphisms over $k$ than over $A'$, we have  the following inclusion of sets
\[
\Hom_{A'}([\mathcal{C}',(\sigma_1,\ldots,\sigma_n)],[\mathcal{X}_{A'},(x_1,\ldots,x_n)]) \subset \Hom_k([C,(c_1,\ldots,c_n)], [X,(x_{1,k}, \ldots,x_{n,k})]).
\]  By our choice of $c_1,\ldots,c_n$ in $C(k)$, the  latter  set is finite. We conclude that $X_L$ is arithmetically hyperbolic over $L$, as required.
\end{proof}

 \begin{theorem}\label{thm:vojta2} Let $k\subset L$ be an extension of algebraically closed fields of characteristic zero and let $X$ be a quasi-projective  variety  over $k$ which  admits a   morphism to some semi-abelian variety over $k$ whose fibres are of dimension at most one. If $X$ is arithmetically hyperbolic over $k$, then $X_L$  is arithmetically hyperbolic over $L$.
 \end{theorem}
\begin{proof}  We may and do assume that $L$ has finite transcendence degree over $k$, say $n\geq 0$. We now argue by induction on $n$.  By our assumption that $X$ is arithmetically hyperbolic over $k$, we may assume that $n>0$. Let $k\subset K\subset L$ be an algebraically closed subfield such that $K$ has transcendence degree $n-1$ over $k$. Then,   the induction hypothesis implies that $X_K$ is arithmetically hyperbolic over $K$. Now,  as $X_K$   admits a morphism to  some semi-abelian variety over $K$ with fibres of dimension at most one,  the theorem   follows from Lemma \ref{lem:vojtatje}, as $L$ has transcendence degree one over $K$.
\end{proof}

\begin{corollary}\label{cor:surfaces_text}
Let $k\subset L $ be an extension of algebraically closed fields of characteristic zero and let $X$ be an integral projective surface over $k$ which admits a non-constant morphism to some abelian variety.  If $X$ is   arithmetically hyperbolic  over $k$, then $X_L$ is arithmetically hyperbolic over $L$.
\end{corollary}
\begin{proof}
If $A$ is an abelian variety over $k$ and $X\to A$ is a non-constant morphism, then the fibres of $X\to A$ are at most one-dimensional, so that the result follows from Theorem \ref{thm:vojta2}.
\end{proof}

 \subsection{Proof of Theorem \ref{thm:surfaces_intro}}  \label{section:weird}
We will prove Theorem \ref{thm:surfaces_intro}   using Corollary \ref{cor:surfaces_text}. However, deducing Theorem \ref{thm:surfaces_intro} from Corollary \ref{cor:surfaces_text} involves understand a subtle difference between the a priori finiteness of rational points on projective varieties and the finiteness of integral points; we refer the reader to \cite[\S7]{JBook} for a discussion of this phenomenon which does not occur when one studies rational points on projective varieties over number rings, but   occurs (in some situations) when studying rational points valued in finitely generated fields with positive transcendence degree over $\mathbb{Q}$. To deal with this subtlety in a systematic manner, we introduce the notion of a ``pure model''. 
 
  \begin{definition}[Pure model]\label{defnnn}
 Let $X$ be a variety over $k$. Let $A\subset k$ be a subring. A model $\mathcal{X}$ for $X$ over $A$ is \emph{pure over $A$} (or: \emph{satisfies the extension property over $A$}) if, for every smooth finite type integral scheme $T$ over $A$, every dense open subscheme $U\subset T$ with $T\setminus U$   of codimension at least two in $T$, and every morphism $f:U\to \mathcal{X}$, there is a (unique) morphism $\overline{f}:T\to \mathcal{X}$   extending the morphism $f$.   (The uniqueness of the extension $\overline{f}$ follows from our convention that a model for $X$ over $A$ is separated.)
  \end{definition}
   
  Note that $X$ is pure over $k$ (in the sense of Remark \ref{remark:pure}) if, for $A=k$, it has a pure model over $A$  (in the sense of Definition \ref{defnnn}).
  
  \begin{definition} A variety $X$ over $k$ has an \emph{arithmetically-pure model} if there is a $\ZZ$-finitely generated subring $A\subset k$ and a pure model $\mathcal{X}$ for $X$ over $A$.
  \end{definition}

\begin{remark}\label{remark:pp0}
Let $X$ be a proper variety over $k$ which has an arithmetically-pure model. Then $X$ has no rational curves (i.e., $X$ is pure over $k$).
\end{remark}

\begin{remark}\label{remark:pp}
Let $A\subset k$ be a regular $\mathbb{Z}$-finitely generated subring.
 Let $X$ be a proper variety over $k$.
 A model $\mathcal{X}$ for $X$ over $A$  with no rational curves in any geometric fibre is pure over $A$ by \cite[Proposition~6.2]{GLL}. On the other hand,  a pure model for $X$ over $A$ might have rational curves in every special fibre (of positive characteristic), as can be seen by considering  the moduli space of principally polarised abelian surfaces over $\mathbb{Z}$.
\end{remark}

 The relevance of  proper varieties with an arithmetically-pure model should be clear from the following result.
  
  \begin{theorem}\label{thm:pure}
Let $X$ be an arithmetically hyperbolic proper variety over $k$ which has an arithmetically-pure model.   Then, for every finitely generated subfield $K\subset k$ and every model $\mathcal{X}$ for $X$ over $K$, the set $ \mathcal{X}(K)$ is finite. 
  \end{theorem}
  \begin{proof}    Let $A\subset k$ be a smooth $\ZZ$-finitely generated subring and let $\mathcal{X}$ be a  pure proper model for $X$ over $A$.  It suffices to show that for any finitely generated subfield $K\subset k$ containing $A$, the set $\mathcal{X}(K)$ is finite. To do so, let $B\subset k$ be a $\ZZ$-finitely generated subring containing $A$ with fraction field $K$ such that $\Spec B \to \Spec A$ is smooth. Let $T:=\Spec B$ and note that, for every $x$ in $\mathcal{X}(K)$, by the valuative criterion for properness, there is a dense open subscheme $U\subset T$ with $\mathrm{codim}(T\setminus U)\geq 2$ such that $x:\Spec K\to \mathcal{X}$ extends to a morphism $U\to \mathcal{X}$. Since $\mathcal{X}$ is a  pure model over $A$, this morphism extends uniquely to a morphism $T\to U$. This shows that $\mathcal{X}(B)= \mathcal{X}(K)$. Since $X$ is arithmetically hyperbolic over $k$, we have that $\mathcal{X}(B)$ is finite, so that $\mathcal{X}(K)$ is finite, as required. 
  \end{proof}

  \begin{lemma}\label{lem:existence_pure}
 Let $X$ be a projective integral groupless surface over $k$. If there is an abelian variety $G$ over $k$ and  a non-constant morphism $f:X\to G$, then $X$ has an arithmetically-pure model.
 \end{lemma}
 \begin{proof}    Let $Y$ be the image of $f$.  Since $G$ is an abelian variety, it follows that there is a regular $\mathbb{Z}$-finitely generated subring $A\subset k$ such that  $Y$ has a projective model over $A$
  whose  geometric  fibres do not have any rational curves. (More precisely, choose $A$ such that $Y\subset G$ extends to a closed immersion $\mathcal{Y}\subset \mathcal{G}$ with $\mathcal{G}$ an abelian scheme over $A$. Then, as the geometric fibres of $\mathcal{G}\to \Spec A$ are abelian varieties, they have no rational curves. This implies that the geometric fibres of $\mathcal{Y}\to \Spec A$ have no rational curves, as required.)
  
    Let $X\to Y'\to Y$ be the Stein factorisation of $X\to Y$, so that $X\to Y'$ has geometrically connected fibres and $Y'\to Y$ is finite and surjective.     By spreading out the finite surjective morphism $Y'\to Y$, replacing $\Spec A$ by a dense open if necessary, we may assume that $Y'$ has a projective model $\mathcal{Y}'$ over $A$ whose geometric fibres do not have any rational curves.

 The fibres of $f:X\to Y'$ are of dimension at most one (since $f$ is non-constant) and groupless (since $X$ is groupless).  Therefore, since being groupless  is a Zariski open property for curves,  replacing $\Spec A$ by a dense open if necessary, there is  a  projective   model $\mathcal{X}$  for $X$ over $A$ and a model $\mathcal{X}\to \mathcal{Y}'$  for $X\to Y'$ whose geometric fibres are groupless.  In particular, the fibres of $\mathcal{X}\to \mathcal{Y}'$ do not contain rational curves. Therefore, as the geometric fibres of $\mathcal{Y}'\to \Spec A$ do not contain rational curves, we conclude that the geometric fibres of $\mathcal{X}\to \Spec A$ do not contain rational curves. This implies that the model $\mathcal{X}$ is pure over $A$ (Remark \ref{remark:pp}). 
 \end{proof}
 
 \begin{theorem}\label{thm:surfaces_text2} Let $K$ be a finitely generated field over $\QQ$ and 
 let $X$ be a projective integral surface over  $K$. Assume that there is an abelian variety $A$ over $\overline{K}$ and a non-constant morphism from $X_{\overline{K}}$   $A$.    If, for every \textbf{finite} extension $L$ over $K$, the set $ {X}(L)$ is finite,  then, for every  \textbf{finitely generated} extension $M$  of $K$, the set   $X(M)$  is finite.
 \end{theorem}
 \begin{proof}
 Let $k:=\overline{K}$ and note that  the assumption implies that $X_k$   is arithmetically hyperbolic over $k$.  Now, since $X_k$    admits a non-constant morphism to some abelian variety $A$ over $k$,  for every algebraically closed field $k'$ containing $k$, the projective variety $X_{k'}$ is arithmetically hyperbolic over $k'$ by Corollary \ref{cor:surfaces_text}.   Since arithmetically hyperbolic varieties over $k'$ are groupless over $k'$ (see \cite[\S 3]{JAut}) and $X_{k'}$ admits a non-constant morphism to the abelian variety $A_{k'}$, it follows from Lemma \ref{lem:existence_pure} that   $X_{k'}$   has an arithmetically-pure model. Thus,   the result follows from Theorem \ref{thm:pure}.
 \end{proof}
 
 \begin{proof}[Proof of Theorem \ref{thm:surfaces_intro}] This follows from Theorem \ref{thm:surfaces_text2}.
 \end{proof}
 
\begin{remark}
The \emph{a priori} difference between the finiteness of rational points  and the finiteness of integral points on a projective variety over a finitely generated field $K$ naturally leads to Vojta's notion of  ``near-integral points''; see \cite{VojtaLangExc} and also \cite[Section~7]{JBook}.
\end{remark}

 \section{Pseudo-algebraic hyperbolicity}\label{section:pseudo}
 We have followed  Demailly  in our definition of    algebraic hyperbolicity. We now extend Demailly's notion of algebraic hyperbolicity by allowing for an ``exceptional locus'' on which the desired property fails. Lang and Kobayashi  use the term ``pseudo'' for such more general notions (see \cite{KobayashiBook, Lang2}).
 \begin{definition}
 Let $X$ be a projective scheme over $k$ and let $\Delta$ be a closed subscheme of $X$. We say that $X$ is \emph{algebraically hyperbolic modulo $\Delta$} if, there is an ample line bundle $\mathcal L$ on $X$ and  a real number $\alpha_{X,\mathcal{L}}$ depending only on $X$ and $\mathcal{L}$ such that, for every smooth projective curve $C$ over $k$ and every morphism $f:C\to X$ with $f(C)\not\subset \Delta$, the inequality 
\[
\deg_C f^\ast \mathcal{L} \leq \alpha_{X,\mathcal{L}} \cdot \mathrm{genus}(C)
\] holds.  
 \end{definition}

With the definitions given in the introduction, a projective scheme over $k$ is pseudo-algebraically hyperbolic over $k$ if  and only if there is a proper closed subscheme $\Delta$ of $X$ such that $X$ is algebraically hyperbolic modulo $\Delta$.    Also, needless to stress, a projective scheme is algebraically hyperbolic over $k$ if it is algebraically modulo the empty subset.

We start by showing that the pseudo-algebraic hyperbolicity of $X$ persists  over field extensions.

\begin{lemma}\label{lem:ps_alg} Let $k\subset L$ be an extension of algebraically closed fields of characteristic zero.
Let $X$ be a projective scheme over $k$ and let $\Delta\subset X$ be a closed subscheme. If $X$ is algebraically hyperbolic modulo $\Delta$, then $X_L$ is algebraically hyperbolic modulo $\Delta_L$.
\end{lemma}
\begin{proof} Let $\mathcal{L}$ be an ample line bundle on $X$. We follow the  proof of \cite[Theorem~7.1]{JKa}. Assume that $X_L$ is not algebraically hyperbolic modulo $\Delta_L$ over $L$. Then, we may choose, for every $n\geq 1$, a smooth projective irreducible hyperbolic curve $C_n$ over $L$ and a morphism  $f_n:C_n\to X_L$ with $f_n(C_n)\not\subset \Delta_L$ such that the slope $s(f_n) := \frac{\deg(f_n^* \mathcal{L})}{\mathrm{genus}(C_n)} >n$; note that $X_L$ does not contain any rational curves, except those in $\Delta_L$, as $X$ has this property, so that $\mathrm{genus}(C_n) \neq 0$. 

For every $n\geq 1$, we choose a finitely generated $k$-algebra $A_n\subset L$ with $U_n :=\Spec A_n$, a smooth projective   morphism $\mathcal{C}_n\to U_n$ with geometrically connected fibres, an isomorphism $\mathcal{C}_{n,L}\cong C_n$ over $L$, a model $F_n:\mathcal{C}_n \to X\times U_n$ for $f_n:C_n\to X_L$ over $U_n$, and a point $u_n$ in $U_n(k)$ such that  the image of $F_{n,u_n}$ is not contained in $\Delta$. Note that, for every $n\geq 1$, the slope  $s(F_{n,u_n})$ of the morphism $F_{n,u_n}:\mathcal{C}_{n,u_n}\to X\times\{u_n\} \cong X$ equals the slope $s(f_n)$. 

 For every $n\geq 1$, we write $D_n:=\mathcal{C}_{n,u_n}$ and note that $D_n$ is a smooth projective irreducible curve over $k$. Note that the slope of the morphisms $F_{n,u_n}:D_n \to X$ tends to infinity as $n$ grows, and that the image of $F_{n,u_n}$ is not contained in $\Delta$. This implies that $X$ is not algebraically hyperbolic modulo $\Delta$, as required.
\end{proof}

\begin{proof}[Proof of Theorem \ref{thm:ps_alg}] Let $\Delta\subset X$ be a proper closed subset such that $X$ is algebraically hyperbolic modulo $\Delta$ over $k$. Then it follows from Lemma \ref{lem:ps_alg} that $X_L$ is algebraically hyperbolic modulo $\Delta_L$ over $L$. As $\Delta_L\subset X_L$ is a proper closed subscheme, we conclude that the projective variety $X_L$ is pseudo-algebraically hyperbolic over $L$, as required.
\end{proof}

\begin{remark}\label{remark:similar} The proof of Lemma \ref{lem:ps_alg} also shows the following useful fact. Let $k \subset L$ be an extension of algebraically closed fields, let $X$ be a projective variety over $k$, let $\Delta \subset X$ be a proper closed subscheme, and let $\mathcal{L}$ be an ample line bundle on $X$. Then, for every smooth projective irreducible curve $C_0$ over $L$ and every $f_0 \in \Hom_L(C_0, X_L) \setminus \Hom_L(C_0, \Delta_L)$, there is a smooth projective irreducible curve $C$ over $k$ and  a morphism $f \in \Hom_k(C,X) \setminus \Hom_k(C, \Delta)$ such that $\mathrm{genus}(C) = \mathrm{genus}(C_0)$ and $\deg_C(f^* \mathcal{L}) = \deg_{C_0}(f_0^* \mathcal{L}_L)$.
\end{remark}

 \subsection{Pseudo-boundedness}
We   extend the notion of boundedness introduced in \cite[\S 4]{JKa} to the pseudo-setting.

 \begin{definition}
 Let $X$ be a projective scheme over $k$ and let $\Delta$ be a closed subscheme of $X$. We say that $X$ is \emph{$N$-bounded modulo $\Delta$} if for every normal projective variety $V$ of dimension at most $N$ over $k$ the scheme $\underline{\Hom}(V,X) \setminus \underline{\Hom}(V,\Delta)$ is of finite type over $k$, i.e., if there are only finitely many polynomials occurring as the Hilbert polynomial of a morphism $V \to X$ not mapping into $\Delta$.  
 \end{definition}

It is obvious that algebraically hyperbolic varieties are $1$-bounded. We record this in the following lemma.
  
 \begin{lemma}\label{lem:basic_implications1}    Let $X$ be a projective scheme over $k$ and let $\Delta$ be a closed subscheme.  If $X$ is algebraically hyperbolic modulo $\Delta$, then $X$ is $1$-bounded modulo $\Delta$.
  \end{lemma}
 
  \begin{proof}For every smooth projective irreducible curve $C$ over $k$,    the scheme  $\underline{\Hom}_k(C,X)\setminus \underline{\Hom}_k(C,\Delta)$ is an open subscheme of    scheme $\underline{\Hom}_k(C,X)$. Therefore the statement follows from the definitions.
  \end{proof}

\begin{lemma}\label{lem:pseudoboundedness_generises} Let $S$ be an  integral normal variety over $k$, let $N$ be a positive integer, and let $X \to S$ be a projective  morphism. Let $\Delta \subset X$ be a closed subscheme. Let
 $A \subset S(k)$ be a subset not contained in any countable union of proper closed subsets of $S$.
If  $X_s$ is $N$-bounded  modulo $\Delta_s$ for all $s \in A$, then $X_{\overline{K(S)}}$ is $N$-bounded modulo $\Delta_{\overline{K(S)}}$.
\end{lemma}

\begin{proof}
Write $M = \overline{K(S)}$. Suppose $X_M$ is not $N$-bounded modulo $\Delta_M$. Then there exists a normal projective integral variety $Y$ over $M$ of dimension at most $N$ and a sequence of morphisms $f_n \colon Y \rightarrow X_M$ with pairwise distinct Hilbert polynomials such that $f_n(Y) \not \subset \Delta_M$. Taking a finite extension of $K(S)$ if necessary, standard spreading out arguments show that there is a dense open subscheme  $U \subset S$ and a projective flat geometrically integral model $\mathcal{Y} \to U$ for $Y$ over $U$ whose  geometric fibres are normal (and of dimension at most $N$); see for example \cite[Th\'eor\`eme~9.7.7, Proposition~9.9.1, Proposition~9.9.4]{EGAIVIII}.

For every $n\geq 1$, by standard spreading out arguments, there is a dominant \'etale   morphism $V_n\to U$ with image $U_n\subset U$ and a morphism $F_n \colon \mathcal{Y}_{V_n} \to X_{V_n}$  extending the morphism $f_n:Y\to  X_M$  such that, for every $v$ in $V_n(k)$, the Hilbert polynomial of the morphism $F_{n,v}\colon \mathcal{Y}_{V_n,v}\to X_{V_n,v}$ equals the Hilbert polynomial of $f_n$ and the image     of $F_{n,v}$ is not contained in $\Delta_{v_n}$.

By our assumption on the set $A$, the intersection $ \bigcap_{n=1}^\infty U_n \cap A$ is non-empty.   Let $s$ in $S(k)$ be an element of this intersection. For every $n\geq 1$, let $v_n\in V_n$ be a point lying over $s$ (via $V_n\to U_n$). Note that $\mathcal{Y}_{s} \cong \mathcal{Y}_{V_n,v_n}$ for all $n\geq 1$. Moreover,  the morphisms $F_{n,v_n}:\mathcal{Y}_{s}\cong \mathcal{Y}_{V_n,v_n}\to X_{V_n,v_n} \cong X_s$ have pairwise distinct Hilbert polynomials and their image is not contained in $\Delta_{s}$. This implies that $X_{s}$ is not $N$-bounded modulo $\Delta_s$, as required. 
\end{proof}

\begin{corollary}\label{lem:ps_1b} Let $k\subset L$ be an extension of \textbf{uncountable} algebraically closed fields of characteristic zero.
Let $X$ be a projective scheme over $k$ and let $\Delta\subset X$ be a closed subscheme. If $X$ is $N$-bounded modulo $\Delta$ over $k$, then $X_L$ is $N$-bounded modulo $\Delta_L$ over $L$. 
\end{corollary}
\begin{proof}
This follows from Lemma \ref{lem:pseudoboundedness_generises} and the arguments used in the proof of Corollary \ref{cor:mb_persists}.
\end{proof}

 \begin{remark}
 A pseudo-bounded variety is not necessarily birational to a bounded variety.
 \end{remark}

 \begin{lemma} \label{lem:1_is_n}  Assume $k$ is \textbf{uncountable}.   Let $X$ be a projective scheme over $k$ and let $\Delta$ be a closed subscheme. 
  If $X$ is $1$-bounded modulo $\Delta$ over $k$, then $X$ is bounded modulo $\Delta$ over $k$.
\end{lemma}
\begin{proof} (We adapt the proof of  \cite[Theorem~9.3]{JKa}.)
We show by induction  on $n\geq 1$ that $X$ is $n$-bounded modulo $\Delta$ over $k$. Thus,  let $n  \geq 2$ and assume that    $X$ is $(n-1)$-bounded. 
 Note that, for $V$ a projective normal variety over $k$,  the Hilbert polynomial of a morphism $f: V \to X$ is uniquely determined by the numerical equivalence class of $f^* {\mathcal O} (1)$ in $\NS(V)$; see the proof of \cite[Theorem~9.3]{JKa} for an argument.
 
 Assume  that  $X$ is not $n$-bounded, so that there is a projective normal variety $Y$ of dimension $n$ over $k$ and morphisms $f_1, f_2, f_3,\ldots$ from $Y$ to $X$    with pairwise distinct Hilbert polynomials.
Note that the numerical equivalence classes of $f_1^* {\mathcal O} (1), f_2^\ast \mathcal{O}(1), \ldots$ are pairwise distinct. Define $\Delta_i:=f_i^{-1}(\Delta)\subset Y$, and note that   $\Delta_i$ is a proper closed subscheme of $Y$. Since $k$ is uncountable, there is a smooth ample divisor $D$ in $Y$ such that, for all $i$, we have that $D$ is not contained in $\Delta_i$. By \cite[Lemma~9.1]{JKa}, it follows that $$f_i^*({\mathcal O} (1)) \cdot D^{\dim Y -1} \to \infty, \quad i\to \infty.$$  In particular, we have that $f_i^*({\mathcal O} (1))|_D \cdot D|_D^{\dim Y -2} \to \infty$. Therefore, the morphisms $f_i|_D:D\to X$ have pairwise distinct Hilbert polynomials and satisfy $f_i|_D(D) \not \subset \Delta$. Since $D$ is a  smooth projective variety with $\dim D = n-1 <n$,   this contradicts the fact that $X$ is $(n-1)$-bounded modulo $\Delta$. We conclude that $X$ is $n$-bounded modulo $\Delta$. 
\end{proof}

 \begin{lemma}\label{lem:ps_alg_hyp}  Let $k\subset L$ be an extension of algebraically closed fields.  Let $X$ be a projective scheme over $k$ and let $\Delta$ be a closed subscheme. 
 If   $X$ is algebraically hyperbolic  modulo $\Delta$ over $k$, then $X_L$ is bounded modulo $\Delta_L$ over $L$. 
 \end{lemma}
 \begin{proof} We may and do assume that $L$ is uncountable. By Theorem \ref{thm:ps_alg},   we have that $X_L$ is algebraically hyperbolic modulo $\Delta_L$. In particular, it follows that $X_L$ is   $1$-bounded modulo $\Delta_L$ (Lemma \ref{lem:basic_implications1}).  Therefore, as $L$ is uncountable, it follows from Lemma \ref{lem:1_is_n} that  $X_L$ is bounded modulo $\Delta_L$.  
 \end{proof}
 
 \begin{proof}[Proof of Theorem \ref{thm:alghyp_to_bounded}]
 This follows from Lemma \ref{lem:ps_alg_hyp} (with $k=L$).
 \end{proof}

 \begin{proof}[Proof of Theorem \ref{thm:uniformity}]
  We follow the arguments used to prove \cite[Theorem~1.14]{JKa}.
 
Let $\mathcal{M}:=\mathcal{M}_{g,k}$ be the stack of smooth proper genus $g\geq 2$ curves over $k$  and let $\mathcal{U}\to \mathcal{M}$ be the universal  curve of genus $g$. We claim that the morphism $$\underline{\Hom}_{\mathcal{M}}(\mathcal{U}, X\times \mathcal{M}) \setminus \underline{\Hom}_{\mathcal{M}}(\mathcal{U}, \Delta \times \mathcal{M}) \to \mathcal{M}$$ is of finite type.

To do so,  let $\overline{\mathcal{U}}\to \overline{\mathcal{M}}$ be the universal stable curve of genus $g$ over the   stack $\overline{\mathcal{M}}$ of stable curves of genus $g$. Let $P$ be an integral   projective   scheme over $k$ and let $P\to \overline{\mathcal{M}}$ be a finite flat   surjective morphism of schemes; such a scheme exists by the fact that $\overline{\mathcal{M}}$ is a proper Deligne-Mumford stack with projective coarse space \cite[Theorem~16.6]{LMBbook}. Now, let $Z\to P$ be a finite flat surjective morphism with $Z$ an integral projective bounded scheme over $k$. Let $Y:= \overline{\mathcal{U}} \times_{\overline{M}} Z$ and consider the morphism $Y\to Z$. 
 
  Note that $X\times Z$ is  bounded modulo $\Delta\times Z$.
 Therefore, the scheme $\underline{\Hom}_k(Y, X\times Z)\setminus \underline{\Hom}_k(Y, \Delta\times Z)$ is of finite type over $k$. In particular, the morphism $\underline{\Hom}_{Z}(Y,X\times Z)\setminus \underline{\Hom}_Z(Y, \Delta\times Z)\to Z$ is of finite type. Now, by the definition of Hom-functors, we have that 
\[
\underline{\Hom}_{\overline{\mathcal{M}}}(\overline{\mathcal{U}},X\times \overline{\mathcal{M}}) \times_{\overline{\mathcal{M}}} Z = \underline{\Hom}_{\overline{\mathcal{M}}\times_{\overline{\mathcal{M}}} Z }(\overline{\mathcal{U}} \times_{\overline{\mathcal{M}}} Z ,X\times \overline{\mathcal{M}} \times_{\overline{\mathcal{M}}} Z ) = \underline{\Hom}_{Z}(Y,X\times Z).
\] Therefore,  since, the morphism $\underline{\Hom}_{Z}(Y,X\times Z)\setminus \underline{\Hom}_Z(Y, \Delta\times Z)\to Z$ is of finite type, it follows from   fppf descent that  the morphism $\underline{\Hom}_{\overline{\mathcal{M}}}(\overline{\mathcal{U}},X\times \overline{\mathcal{M}}) \setminus \underline{\Hom}_{\overline{\mathcal{M}}}(\overline{\mathcal{U}}, \Delta\times \overline{\mathcal{M}}) \to \overline{\mathcal{M}}$ is of finite type. By base-change, this proves our claim that the morphism $\underline{\Hom}_{\mathcal{M}}(\mathcal{U}, X\times \mathcal{M}) \setminus \underline{\Hom}_{\mathcal{M}}(\mathcal{U}, \Delta \times \mathcal{M}) \to \mathcal{M}$ is of finite type. As  the latter morphism  is of finite type (for every $g\geq 2$) we see that, for every ample line bundle $\mathcal{L}$ on $X$ and every integer $g\geq 2$, there is an integer $\alpha(X, \Delta, \mathcal{L}, g)$ such that, for every smooth projective irreducible curve $C$ of genus $g$ over $k$ and every morphism $f:C\to X$ with $f(C)\not\subset \Delta$, the inequality 
 \[
 \deg_C f^\ast \mathcal{L} \leq \alpha(X,\Delta,\mathcal{L},g)
 \] holds. This implies the desired statement and concludes the proof.
 \end{proof}

 \begin{corollary}\label{cor:uniformity0}
Let $X$ be a projective scheme over $k$ and let $\Delta$ be a closed subscheme of $X$ such that $X$ is $1$-bounded modulo $\Delta$ over $k$. If $k$ is \textbf{uncountable}, for every ample line bundle $\mathcal{L}$ on $X$ and every integer $g\geq 0$, there is a real number $\alpha(X,\Delta,\mathcal{L},g)$ such that, for every smooth projective irreducible curve $C$ of genus $g$ over $k$  and every morphism $f:C\to X$ with $f(C)\not\subset \Delta$, the inequality 
  \[
  \deg_C f^\ast \mathcal{L} \leq \alpha(X,\Delta,\mathcal{L},g) 
  \] holds. 
 \end{corollary}
  \begin{proof} 
Since the ground field $k$ is assumed to be uncountable, this follows from Lemma \ref{lem:1_is_n}  and  Theorem \ref{thm:uniformity}.
 \end{proof}

\begin{corollary} Let $k\subset L$ be an extension of algebraically closed fields. Let $X$ be a projective scheme over $k$ and let $\Delta$ be a closed subscheme. If $X$ is bounded modulo $\Delta$ over $k$, then $X_L$ is bounded modulo $\Delta_L$ over $L$.
\end{corollary}
\begin{proof}  To prove the corollary, we may and do assume that $L$ is uncountable. Now,  since $X$ is bounded modulo $\Delta$ over $k$, it follows from Theorem \ref{thm:uniformity} that, for every ample line bundle $\mathcal{L}$ on $X$ and every integer $g\geq 0$, there is a real number $\alpha(X,\Delta,\mathcal{L},g)$ such that, for every smooth projective irreducible curve $C$ of genus $g$ over $k$ and every morphism $f:C\to X$ with $f(C)\not\subset \Delta$, the inequality 
\[
\deg_C f^\ast \mathcal{L} \leq \alpha(X,\Delta, \mathcal{L}, g)
\]  holds. In particular, by Remark \ref{remark:similar}, the "same" holds over $L$.  More precisely,  for every smooth projective irreducible curve $C_0$ of genus $g$ over $L$ and every morphism $f_0:C_0\to X_L$ with $f_0(C_0)\not\subset \Delta_L$,  there exists a smooth projective irreducible curve $C$ of genus $g$ over $k$ and a morphism   $f \colon C \to X$ with $f(C) \not\subset \Delta$ such that the (in)equality 
\[
\deg_{C_0} f_0^\ast \mathcal{L}_L = \deg_C f^\ast \mathcal{L} \leq \alpha(X,\Delta, \mathcal{L}, g)
\]  holds.  In particular, the projective scheme $X_L$ is $1$-bounded modulo $\Delta_L$ over $L$. Since $L$ is uncountable, it follows from Lemma \ref{lem:1_is_n}   that $X_L$ is bounded modulo $\Delta_L$ over $L$.
\end{proof}

\begin{remark}
The notion of pseudo-algebraic hyperbolicity is further studied in \cite{JXie}. For example, motivated by Vojta's conjecture and the finiteness theorem of Kobayashi--Ochiai \cite{KobaOchi} for varieties of general type, it is shown in \cite[Theorem~1.11]{JXie} that if $X$ is a projective pseudo-algebraically hyperbolic scheme over $k$ and $Y$ is a projective integral variety over $k$, then the set of surjective morphisms $Y\to X$ is finite.  
\end{remark}

\subsection{Pointed boundedness}
We extend the notion of pointed boundedness introduced in \cite[\S4]{JKa} to the pseudo-setting. Note that this notion is referred to as  ``geometric hyperbolicity'' in \cite[\S 11]{JBook}.

\begin{definition}
Let $X$ be a projective scheme  over $k$ and let $\Delta\subset X$ be a closed subscheme. We  say that $X$ is \emph{$(n,1)$-bounded modulo $\Delta$} if, for every smooth projective connected variety $Y$ of dimension at most $n$ over $k$, every $y$ in $Y(k)$, and every $x$ in $X(k)\setminus \Delta$, the scheme $\underline{\Hom}_k([Y,y],[X,x])$ is of finite type over $k$. 
\end{definition}

\begin{remark} 
Let $k\subset L$ be an extension of \textbf{uncountable} algebraically closed fields of characteristic zero, let $X$ be a projective scheme  over $k$ and let $\Delta\subset X$ be a closed subscheme.  
If $X$ is $(n,1)$-bounded   modulo $\Delta$ over $k$, then $X_L$ is $(n,1)$-bounded modulo $\Delta_L$ over $L$. This is proven in a similar manner as Corollary \ref{lem:ps_1b}. 
\end{remark}

\begin{lemma} Let $X$ be a projective   variety over $k$, let $\Delta\subset X$ be a closed subscheme, and let $n\geq 1$ be an integer. If $X$ is $(1,1)$-bounded modulo $\Delta$, then $X$ is $(n,1)$-bounded modulo $\Delta$.
\end{lemma}
\begin{proof}
As the argument is similar to the proof of  Lemma \ref{lem:1_is_n}, we will be brief on the details. (Note that we do not require $k$ to be uncountable.)

We argue by induction on $n$. Thus, suppose that $X$ is $(n-1,1)$-bounded modulo $\Delta$.
 Assume  that  $X$ is not $(n,1)$-bounded modulo $\Delta$, so that there is a projective smooth variety $Y$ of dimension $n$ over $k$, a point $y$ in $Y(k)$, a point $x \in X(k)\setminus \Delta$, and morphisms $f_1, f_2, f_3,\ldots$ from $Y$ to $X$    with pairwise distinct Hilbert polynomials and $f_i(y)=x$.
Note that the numerical equivalence classes of $f_1^* {\mathcal O} (1), f_2^\ast \mathcal{O}(1), \ldots$ are pairwise distinct. 
  Let $D$ be a smooth ample divisor  in $Y$ which contains $y$. Then, the morphisms $f_i|_D$ send $y$ to $x$. In particular, the image of $f_i|_D$ is not contained in $\Delta$. Then, by \cite[Lemma~9.1]{JKa},  the sequence of integers $ f_i^\ast(\mathcal{O}(1)) \cdot D^{\dim Y-1}$ tends to infinity as $i$ tends to infinity.  Since $D$ is a  smooth projective variety with $\dim D = n-1 <n$,   we obtain a contradiction. We conclude that $X$ is $(n,1)$-bounded modulo $\Delta$. 
\end{proof}

    \begin{proposition}\label{prop:1b_is_geomhyp} Let $X$ be a projective scheme over $k$ and let $\Delta$ be a   closed subscheme of $X$. Then the following are equivalent.
   \begin{enumerate}
   \item The projective variety $X$ is $(1,1)$-bounded modulo $\Delta$.
   \item For every smooth projective irreducible curve $C$ over $k$, every $c$ in $C(k)$, and every $x$ in $X(k)\setminus \Delta$, the set $\Hom([C,c],[X,x])$ is finite. 
   \end{enumerate}
  \end{proposition}
  \begin{proof}  Clearly, it suffices to  show that $(1)\implies (2)$. Thus, let us assume  that there is a sequence $f_1, f_2, \ldots$ of pairwise distinct elements of $\Hom_k([C,c],[X,x])$, where $C$ is a smooth projective irreducible curve, $c\in C(k)$ and $x\in X(k)\setminus \Delta$.   Since $\underline{\Hom}_k([C,c],[X,x]) $ is of finite type (by assumption),   the degree of all the $f_i$ is bounded by some real number (depending only on $C, c, X, x,$ and $\Delta$). Then, it follows that $\underline{\Hom}_k([C,c],[X,x])$ is positive-dimensional, so that by bend-and-break   \cite[Proposition~3.5]{Debarrebook1} there is a rational curve in $X$ containing $x$. This contradicts the  fact that every rational curve in $X$ is contained in $\Delta$.
  \end{proof}

\bibliography{refs}{}
\bibliographystyle{plain}

\end{document}